\documentclass[12pt]{article}

\usepackage[a4paper,hmargin=0.9in,vmargin=1in]{geometry}
\usepackage{amssymb}
\usepackage{amsmath}
\usepackage{amsthm}
\usepackage{stmaryrd}
\usepackage{empheq}
\usepackage{graphicx}
\usepackage{booktabs}
\usepackage{subcaption}
\usepackage[numbers,sort&compress]{natbib}
\usepackage[hidelinks]{hyperref}

\theoremstyle{plain}
\newtheorem{lemma}{Lemma}
\newtheorem{theorem}{Theorem}
\theoremstyle{definition}
\newtheorem{remark}{Remark}

\title{An Implementation-Friendly SDG Scheme based on Cartesian Grids for Stokes Equations with Pressure Robustness and Superconvergence}

\author{
    Bohan Yang\thanks{Department of Mathematics, The Chinese University of Hong Kong, Hong Kong Special Administrative Region. E-mail: \texttt{bhyang@math.cuhk.edu.hk}}
    \and
    Eric T. Chung\thanks{Corresponding author. Department of Mathematics, The Chinese University of Hong Kong, Hong Kong Special Administrative Region. E-mail: \texttt{tschung@math.cuhk.edu.hk}}
}
\date{}

\begin{document}

\maketitle

\begin{abstract}
    This paper develops a staggered discontinuous Galerkin (SDG) scheme based on Cartesian grids for Stokes equations that is simple to implement, intrinsically pressure-robust, and superconvergent for all variables. Instead of the composite meshes used in standard SDG, we construct staggered quadrilateral meshes from Cartesian grids. The scheme takes velocity, pressure, and velocity gradient as unknowns, and employs piecewise-constant spaces with carefully designed staggered continuity. The additional gradient unknowns are locally eliminated via static condensation and can be further removed by mass lumping without loss of accuracy. An explicit pointwise formulation of the scheme is derived, which facilitates implementation and enables a detailed pointwise analysis. We rigorously prove pressure robustness and second-order superconvergence, which holds on general non-uniform Cartesian grids. The scheme is further extended to Navier-Stokes equations by introducing a novel discrete convection term with second-order consistency. Combined with the scalar auxiliary variable (SAV) approach and the Crank-Nicolson (CN) scheme, this yields an unconditionally energy-stable and second-order accurate scheme. Numerical experiments validate the theory and demonstrate accuracy and robustness.
\end{abstract}



\section{Introduction}

The SDG method is a variant of mixed discontinuous Galerkin (DG) methods based on composite meshes, which was initially proposed by Chung and Engquist \cite{chung2006optimal, chung2009optimal} for wave equations, and has since been extensively developed for various fields, including electromagnetics \cite{chung2012staggered, chung2013convergence, zhao2020staggered}, fluid mechanics \cite{chung2012staggereda, kim2013staggered, cheung2015staggered, chung2017analysis, zhao2019priori, kim2020staggered, zhao2020newa, zhao2022staggereda, zhao2022pressure, liu2025analysis}, solid mechanics \cite{lee2016analysis, chung2017discontinuous, zhao2020staggereda, zhao2023lockingfree} and multi-physics \cite{cheung2018mass, zhao2020lowestorder, zhao2021staggered, zhao2022robintype, zhao2023strongly}. The core idea of SDG methods is to partition the computational domain into a primal mesh and its dual mesh, and to enforce the continuity of variables in a staggered manner across interelement boundaries. This staggering, from a DG perspective, yields natural interelement fluxes, thus avoiding the introduction of numerical fluxes or penalty terms, and from a mixed finite element (MFE) perspective, provides compatible pairs of discrete spaces without the necessity of enrichment or stabilization techniques. As a result, the SDG method can directly preserve key structural properties of the underlying continuous problem, such as the adjointness between operators, inf-sup stability, and local mass conservation. Moreover, the composite primal-dual framework facilitates its flexible application to general polygonal and polyhedral meshes \cite{zhao2018staggered, zhao2019staggered, kim2020staggered, zhao2020new,zhao2020staggereda}.

This paper concerns SDG discretizations for Stokes and Navier-Stokes equations. There has been a rich literature: the seminal works \cite{kim2013staggered, cheung2015staggered, chung2017analysis} established the SDG scheme on simplicial meshes; subsequent developments generalized the framework to polygonal and polyhedral meshes \cite{zhao2019staggered, kim2020staggered}; and pressure robustness was investigated in \cite{zhao2022pressure}. In this work, we focus on structured Cartesian grids. Beyond the ease of implementation, such grids admit compact discretizations with high-order accuracy; for example, the classical MAC scheme attains second-order accuracy for both velocity and pressure with a minimal stencil \cite{li2015superconvergence, rui2017stability, li2018superconvergence}. Prior analyses \cite{chung2014staggered, chung2016staggered} have established the superconvergence of the SDG method for certain variables (e.g., the velocity in Stokes flows), by interpreting it as a limiting case of the hybridizable DG method. More recently, tailored SDG schemes on Cartesian grids have been proposed and shown to be stable and optimally convergent \cite{kim2021staggered, kim2024staggered}, but superconvergence and pressure robustness were not addressed. Building on these insights, we develop a zero-order SDG (SDG$_0$) scheme based on Cartesian grids for Stokes and Navier-Stokes equations with the following distinctive features:
\begin{itemize}
    \item straightforward and efficient implementation:
        \begin{itemize}
            \item an explicit pointwise formulation of the SDG scheme is provided,
            \item the additional unknowns (velocity gradient) are partially eliminated via local static condensation and, if desired, can be further removed by mass lumping without loss of accuracy;
        \end{itemize}
    \item second-order superconvergence for all variables---\allowbreak velocity, pressure, and velocity gradient---\allowbreak even on non-uniform grids;
    \item intrinsic pressure robustness, without divergence-free reconstruction \cite{zhao2022pressure} or other post-processing.
\end{itemize}

The main works are summarized as follows. Section~\ref{sec 2} describes the mesh partition and notations. Departing from the standard composite-meshes construction that connects the centerpoint of each element to its vertices, we triangulate each Cartesian block with a uniformly oriented diagonal and, by gluing triangles along horizontal, vertical, and diagonal edges, obtain three staggered quadrilateral meshes. Based on this staggered meshes and employing specific staggered continuity, Section~\ref{sec 3} develops a SDG$_0$ scheme for Stokes equations that treats velocity, pressure and velocity gradient as unknowns. Although several results from the standard SDG theory carry over, they are insufficient to explain the observed second-order accuracy. To this end, Section~\ref{sec 4} introduces tailored basis functions and derives an explicit pointwise formulation of the SDG$_0$ scheme, which facilitates implementation, and, more importantly, enables a detailed pointwise analysis. Besides, local static condensation is applied to eliminate the diagonal entries of the velocity gradient, which can be simply reconstructed as the difference quotients of the velocity. The remaining off-diagonal components are essential for the subsequent discretization of the convection term in Navier-Stokes equations, and can be further removed by standard mass lumping without loss of accuracy, as verified numerically. Section~\ref{sec 5} then presents the error analysis. Inspired by \cite{rui2017stability}, we augment the SDG interpolants with carefully designed second-order corrections to define auxiliary variables and carry out a detailed pointwise analysis. The resulting estimates show that our method is pressure-robust and second-order superconvergent for all variables. Section~\ref{sec 6} extends the method to Navier-Stokes equations. The convection term is discretized by a hybrid strategy that couples a MFE treatment of the velocity gradient with a DG-style discretization of the advective flux, which is proven to be second-order consistent. For time discretization, we employ the scalar auxiliary variable approach together with the Crank-Nicolson scheme \cite{li2020error}, which ensures unconditional energy stability and second-order temporal accuracy. Moreover, it allows for an efficient three-way splitting where each step solves two linear Stokes-type systems and one scalar quadratic equation. Finally, Section~\ref{sec 7} presents several numerical experiments that validate the theoretical results and demonstrate the accuracy and robustness of the proposed method.

To position the proposed scheme within the existing literature, Table~\ref{tab: 1.1} compares it with several representative methods at the lowest order: the classical SDG method for the Stokes equations \cite{kim2013staggered, zhao2019staggered, zhao2022pressure}, a new SDG method for the Brinkman problem \cite{zhao2020newa}, a recently proposed SDG method on rectangular meshes \cite{kim2024staggered}, the MAC scheme \cite{li2015superconvergence, rui2017stability}, the $H(\mathrm{div})$-conforming DG method \cite{cockburn2007note, wang2007new}, and the pressure-robust embedded-hybridized DG (EDG-HDG) method \cite{rhebergen2020embedded, baier-reinio2022analysis}. To enable a uniform comparison across different meshes, the degrees of freedom (DoFs) of all methods are counted on a common Cartesian background grid of $N$ cells: for the SDG methods formulated on polygonal meshes, the primal mesh is taken to be this Cartesian grid and the composite mesh is generated by connecting the centroid of each rectangle to its four vertices; for the DG and EDG-HDG methods, the triangulation is obtained by dividing each rectangle along a diagonal; the remaining methods are discretized directly on the Cartesian grid. We emphasize three points. First, among the SDG family, the present scheme employs the fewest DoFs and is intrinsically pressure-robust without reconstruction. It is also the only SDG method that achieves superconvergence for all variables. Second, as detailed in Section~\ref{sec 4}, the present scheme admits a pointwise formulation that resembles the MAC scheme---sharing the same staggered locations of velocity and pressure and a comparably compact stencil---yet is additionally pressure-robust and attains second-order accuracy for the off-diagonal velocity gradient, which in the MAC scheme is only first-order accurate on non-uniform grids. Third, compared with the pressure-robust DG and EDG-HDG methods, the present scheme uses fewer DoFs while additionally achieving superconvergence for all variables, which has not been established for those methods. In summary, the proposed scheme is the only method in Table~\ref{tab: 1.1} that attains intrinsic pressure robustness together with second-order superconvergence for all variables, at a degree-of-freedom count comparable to the most economical schemes considered.
\begin{table}[htbp]
    \centering
    \small
    \caption{Comparison of the present scheme with representative lowest-order methods for the Stokes problem, organized into (a) discretization and (b) properties. The velocity gradient, velocity, pressure, velocity trace, and pressure trace are denoted by $\boldsymbol{\sigma}$, $\mathbf{u}$, $p$, $\hat{\mathbf{u}}$, and $\hat{p}$, respectively. In panel~(a), the spaces and global degrees of freedom (DoFs) are listed in the same order as the unknowns. Here, $P_0$, $P_1$, and $Q_1$ denote the constant, linear, and bilinear polynomial spaces, respectively, and $BDM_1$ denotes the first-order Brezzi-Douglas-Marini space. The parameter $N$ is the number of cells in the Cartesian background grid. In panel~(b), ''Local elimination'' lists the unknowns that can be removed by local static condensation or mass lumping, and pressure robustness is classified as intrinsic or achieved through a reconstruction operator.}

    \label{tab: 1.1}

    {\itshape (a) Discretization}
    \par\smallskip
    \begin{tabular}{@{}lllll@{}}
        \toprule
        Method & Mesh & Unknowns & Spaces & Global DoFs \\
        \midrule
        \textbf{Present scheme} & Cartesian & $\boldsymbol{\sigma}, \mathbf{u}, p$ & $P_0, P_0, P_0$ & $4N, 2N, N$ \\
        Classical SDG \cite{kim2013staggered, zhao2019staggered, zhao2022pressure} & polygonal & $\boldsymbol{\sigma}, \mathbf{u}, p$ & $P_0, P_0, P_0$ & $8N, 4N, N$ \\
        SDG (Brinkman) \cite{zhao2020newa} & polygonal & $\boldsymbol{\sigma}, \mathbf{u}, \hat{\mathbf{u}}, p$ & $P_0, P_0, P_0, P_0$ & $12N, 4N, 4N, 2N$ \\
        SDG (rect.\ meshes) \cite{kim2024staggered} & Cartesian & $\boldsymbol{\sigma}, \mathbf{u}, p$ & $Q_1, Q_1, P_0$ & $48N, 18N, N$ \\
        MAC \cite{li2015superconvergence, rui2017stability} & Cartesian & $\mathbf{u}, p$ & -- & $2N, N$ \\
        $H(\mathrm{div})$-DG \cite{cockburn2007note, wang2007new} & simplicial & $\mathbf{u}, p$ & $BDM_1, P_0$ & $6N, 2N$ \\
        EDG-HDG \cite{rhebergen2020embedded, baier-reinio2022analysis} & simplicial & $\mathbf{u}, \hat{\mathbf{u}}, p, \hat{p}$ & $P_1, P_1, P_0, P_1$ & $12N, 2N, 2N, 6N$ \\
        \bottomrule
    \end{tabular}

    \par\bigskip

    {\itshape (b) Properties}
    \par\smallskip
    \begin{tabular}{@{}llll@{}}
        \toprule
        Method & Local elimination & Pressure-robustness & Superconvergence \\
        \midrule
        \textbf{Present scheme} & $\boldsymbol{\sigma}$ & intrinsic & $\boldsymbol{\sigma}, \mathbf{u}, p$ \\
        Classical SDG & $\boldsymbol{\sigma}$ & by reconstruction & $\mathbf{u}$ \\
        SDG (Brinkman) & -- & intrinsic & $\mathbf{u}$ \\
        SDG (rect.\ meshes) & $\boldsymbol{\sigma}$ & -- & -- \\
        MAC & -- & -- & $\mathbf{u},p$ \\
        $H(\mathrm{div})$-DG & -- & intrinsic & -- \\
        EDG-HDG & $\mathbf{u},p$ & intrinsic & -- \\
        \bottomrule
    \end{tabular}
\end{table}

\section{Mesh Partition and Notations} \label{sec 2}

In this section, we construct the staggered quadrilateral meshes from Cartesian grids and introduce the necessary notations.

Suppose $\Omega$ is a rectangular domain. Consider a $n_x \times n_y$ Cartesian grid on $\Omega$. The grid points are denoted as $(x_i, y_j)$, where $i = 0, 1, \dots, n_x$ and $j = 0, 1, \dots, n_y$. Let $x_{i + 1/2} = (x_i + x_{i + 1}) / 2$ and $y_{j + 1/2} = (y_j + y_{j + 1}) / 2$ denote the midpoints of the grid segments. For brevity, we adopt the following shorthand for half-integer indices:
\begin{align*}
    i^+ = i + \frac{1}{2}, \quad i^- = i - \frac{1}{2}, \quad
    j^+ = j + \frac{1}{2}, \quad j^- = j - \frac{1}{2}.
\end{align*}
Let $h^x_{i^+} = x_{i + 1} - x_i$ and $h^y_{j^+} = y_{j + 1} - y_j$ denote the lengths of the grid segments. The mesh size is represented as $h = \max \{h^x_{i^+}, h^y_{j^+}\}$. Define the local mesh size at the grid point:
\begin{align*}
    h^x_i &= x_{i^+} - x_{i^-} = \frac{h^x_{i^+} + h^x_{i^-}}{2}, \quad i = 1, 2, \dots, n_x - 1, \\
    h^y_j &= y_{j^+} - y_{j^-} = \frac{h^y_{j^+} + h^y_{j^-}}{2}, \quad j = 1, 2, \dots, n_y - 1, \\
    h^x_0 &= \frac{h^x_{0^+}}{2}, \quad h^x_{n_x} = \frac{h^x_{n_x^-}}{2}, \quad h^y_0 = \frac{h^y_{0^+}}{2}, \quad h^y_{n_y} = \frac{h^y_{n_y^-}}{2}.
\end{align*}

A structured triangular mesh is obtained from the Cartesian grid by dividing each rectangle into two triangles along its diagonal from the top left to the bottom right (see Figure \ref{fig: 2.1}).
\begin{figure}[htbp]
    \centering
    \includegraphics{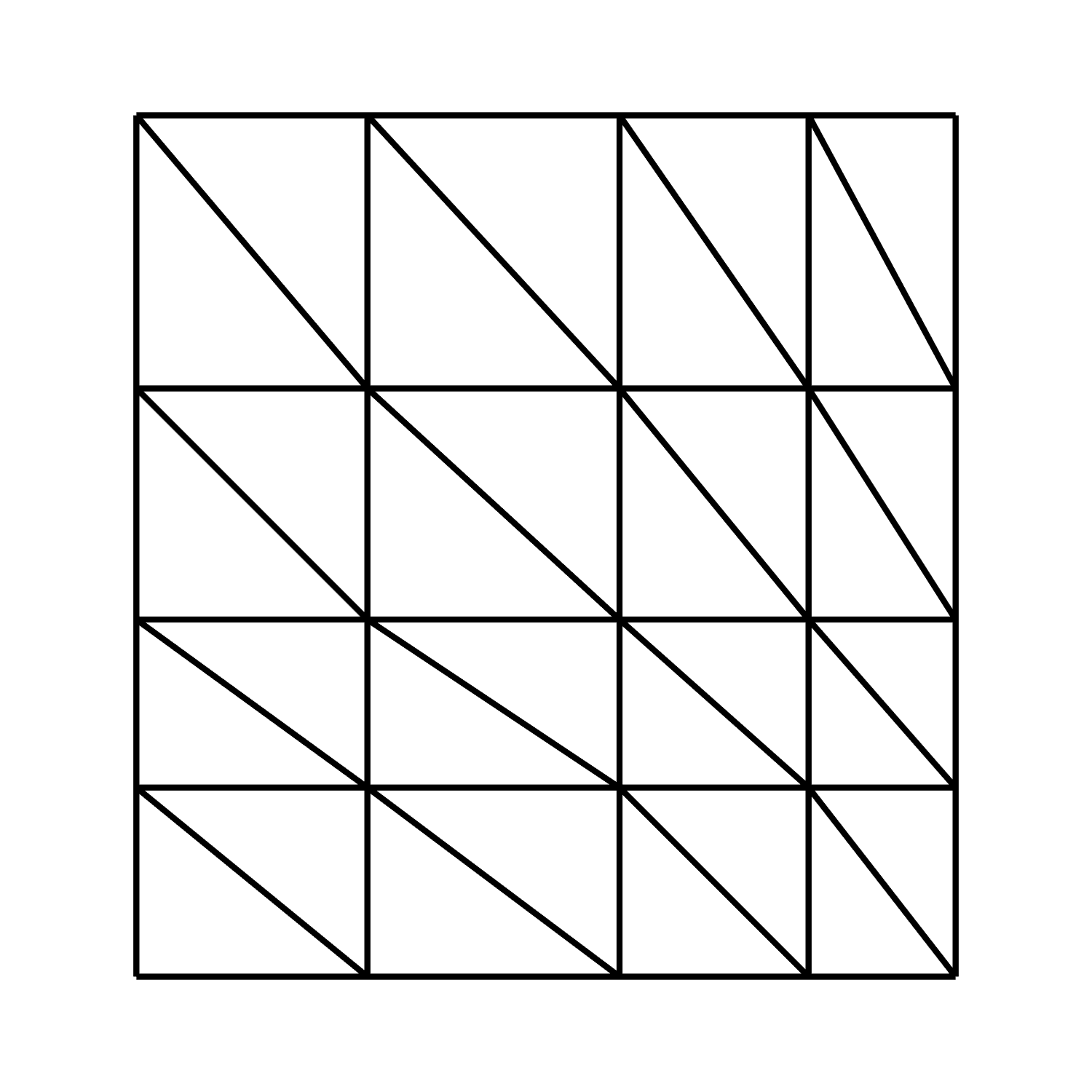}
    \caption{Triangulation of the Cartesian grid with uniformly oriented diagonals.}
    \label{fig: 2.1}
\end{figure}
Let $e_{i^+, j}$ denote the horizontal edge with midpoint $(x_{i^+}, y_j)$, $e_{i, j^+}$ denote the vertical edge with midpoint $(x_i, y_{j^+})$ and $e_{i^+, j^+}$ denote the diagonal edge with midpoint $(x_{i^+}, y_{j^+})$. Let $\mathcal{E}^H_h$, $\mathcal{E}^V_h$ and $\mathcal{E}^D_h$ denote the set of $e_{i^+, j}$, $e_{i, j^+}$ and $e_{i^+, j^+}$, respectively, and let $\Lambda^H_h = \{ (i^+, j) \}$, $\Lambda^V_h = \{ (i, j^+) \}$ and $\Lambda^D_h = \{ (i^+, j^+) \}$ denote the corresponding index sets. Let $\mathring{\mathcal{E}}^H_h$ and $\mathring{\mathcal{E}}^V_h$ denote the set of interior edges in $\mathcal{E}^H_h$ and $\mathcal{E}^V_h$, respectively.
Let $h_e$ denote the length of the edge $e$ and especially let $l_{i^+, j^+}$ denote the length of the diagonal edge $e_{i^+, j^+}$. Let $\mathbf{n} = (n^x, n^y)^{\mathsf{T}}$ and $\mathbf{t} = (t^x, t^y)^{\mathsf{T}}$ denote the normal and tangential unit vectors on edges. Define the jump and average operators on edges:
\begin{align*}
    \llbracket \cdot \rrbracket = (\cdot)^{-} - (\cdot)^{+}, \quad \{\!\!\{ \cdot \}\!\!\} = \frac{(\cdot)^{-} + (\cdot)^{+}}{2},
\end{align*}
where $(\cdot)^{\pm}$ represents the traces on the edge, with $(\cdot)^{+}$ taken from the side toward the normal vector and $(\cdot)^{-}$ from the opposite side.
Let $T^-_{i^+, j^+}$ and $T^+_{i^+, j^+}$ denote the triangles to the lower left and upper right of the diagonal edge $e_{i^+, j^+}$, respectively. Let $\mathcal{T}_h$ denote the set of all triangular elements. Define $T_{i^+, j}$, $T_{i, j^+}$ and $T_{i^+, j^+}$ as the union of triangles sharing the edge $e_{i^+, j}$, $e_{i, j^+}$ and $e_{i^+, j^+}$, respectively (see Figure \ref{fig: 2.2}). Let $\mathcal{T}^H_h$, $\mathcal{T}^V_h$ and $\mathcal{T}^D_h$ denote the set of $T_{i^+, j}$, $T_{i, j^+}$ and $T_{i^+, j^+}$, respectively. As illustrated in Figure \ref{fig: 2.3}, these three quadrilateral meshes collectively form a staggered-meshes system.
\begin{figure}[htbp]
    \centering
    \includegraphics{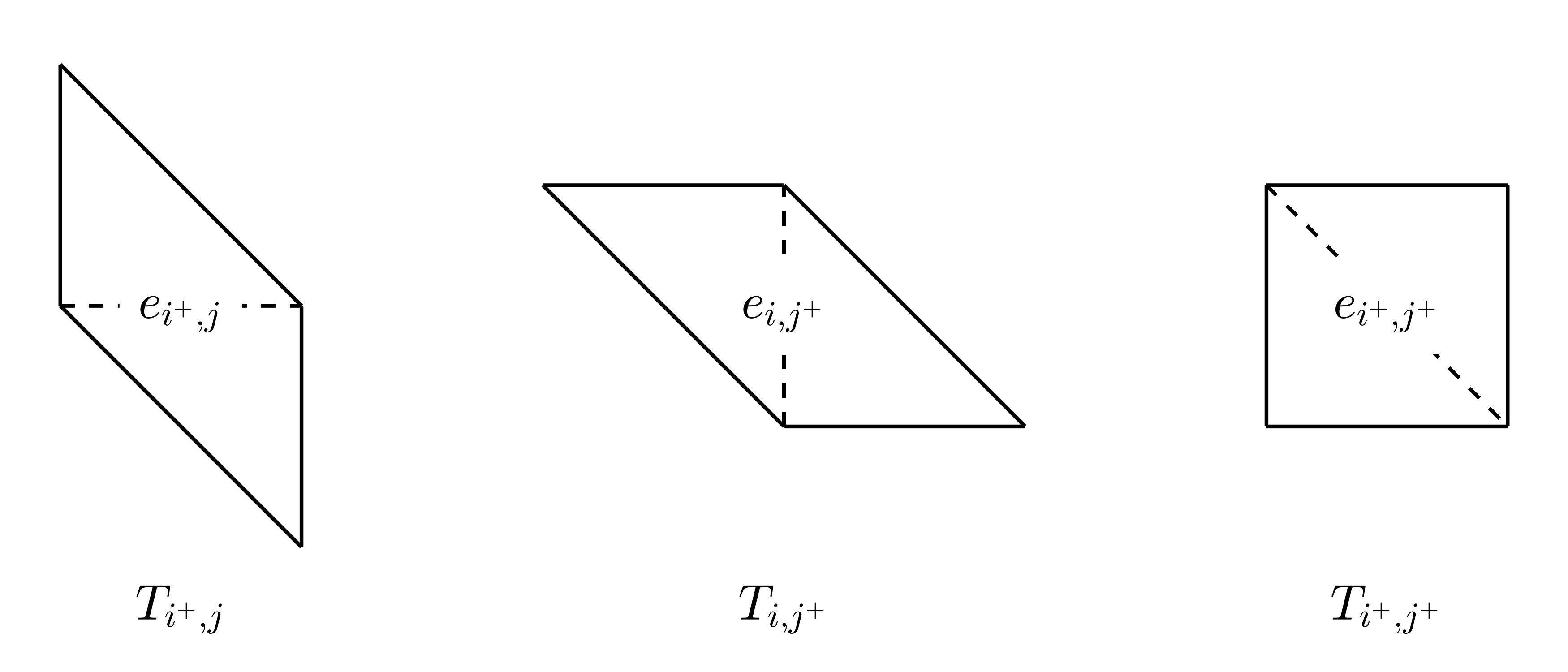}
    \caption{Gluing triangles along different types of edges.}
    \label{fig: 2.2}
\end{figure}
\begin{figure}[htbp]
    \centering
    \begin{subfigure}[h]{0.3\textwidth}
        \centering
        \includegraphics{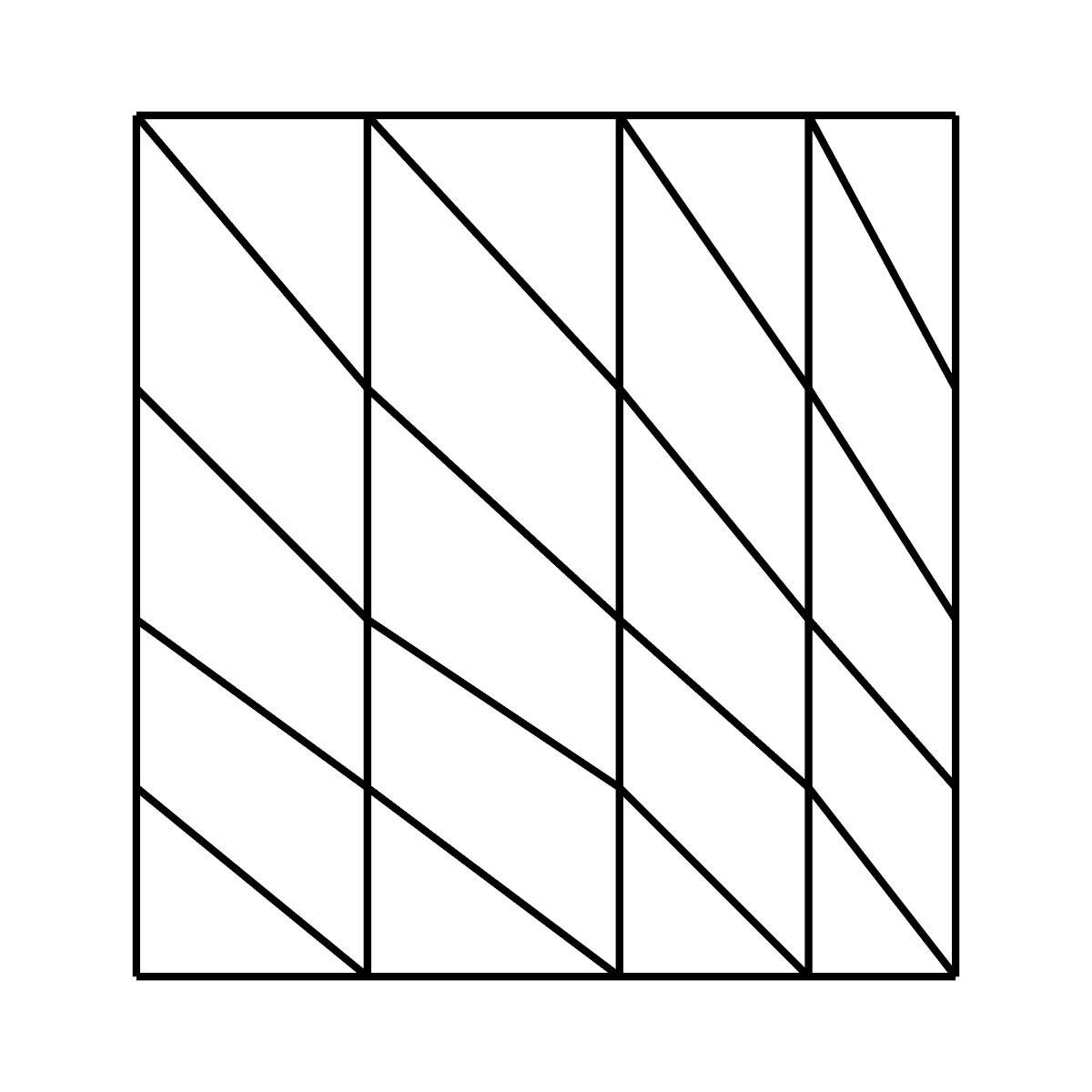}
        \caption{$\mathcal{T}^H_h$}
        \label{fig: 2.3.1}
    \end{subfigure}
    \begin{subfigure}[h]{0.3\textwidth}
        \centering
        \includegraphics{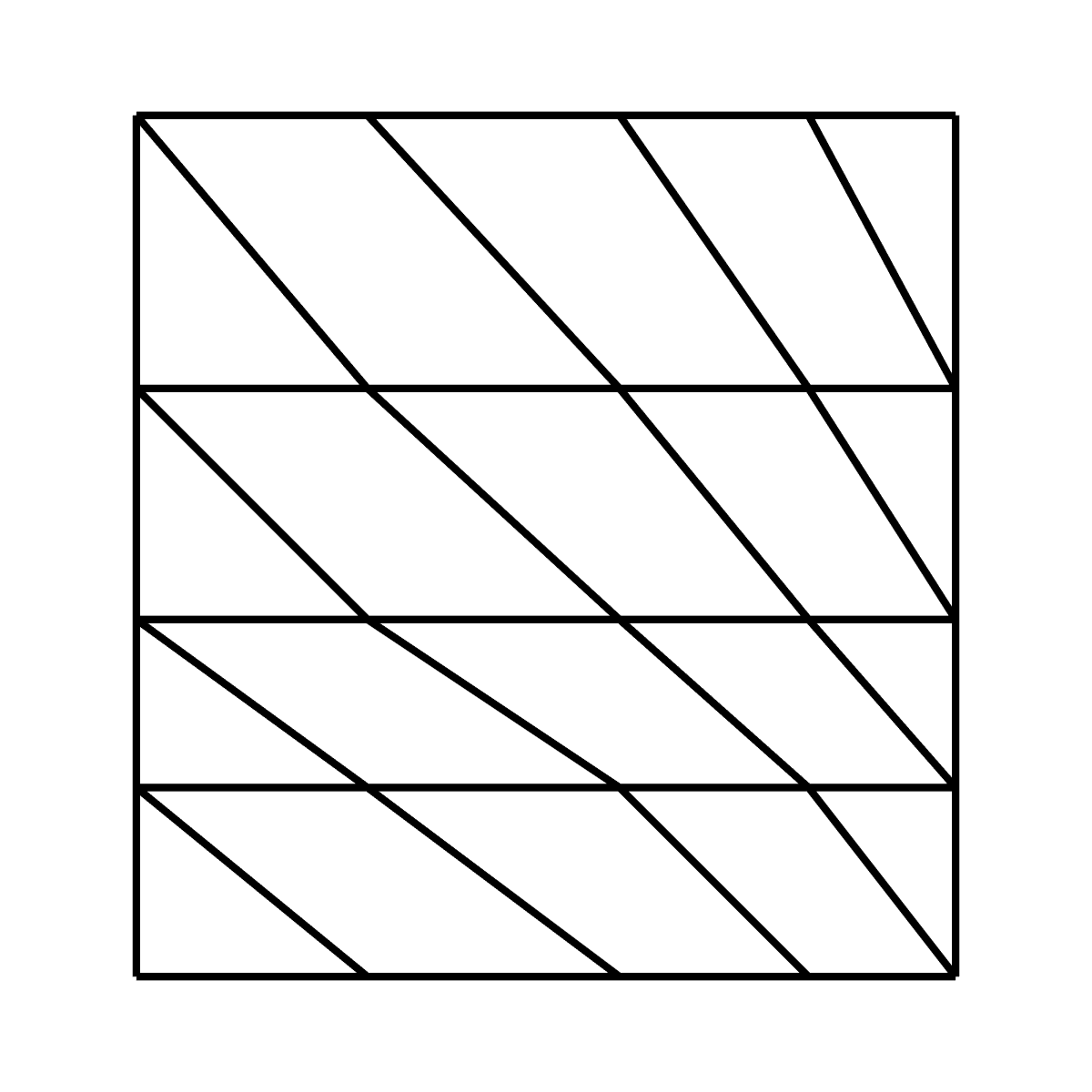}
        \caption{$\mathcal{T}^V_h$}
        \label{fig: 2.3.2}
    \end{subfigure}
    \begin{subfigure}[h]{0.3\textwidth}
        \centering
        \includegraphics{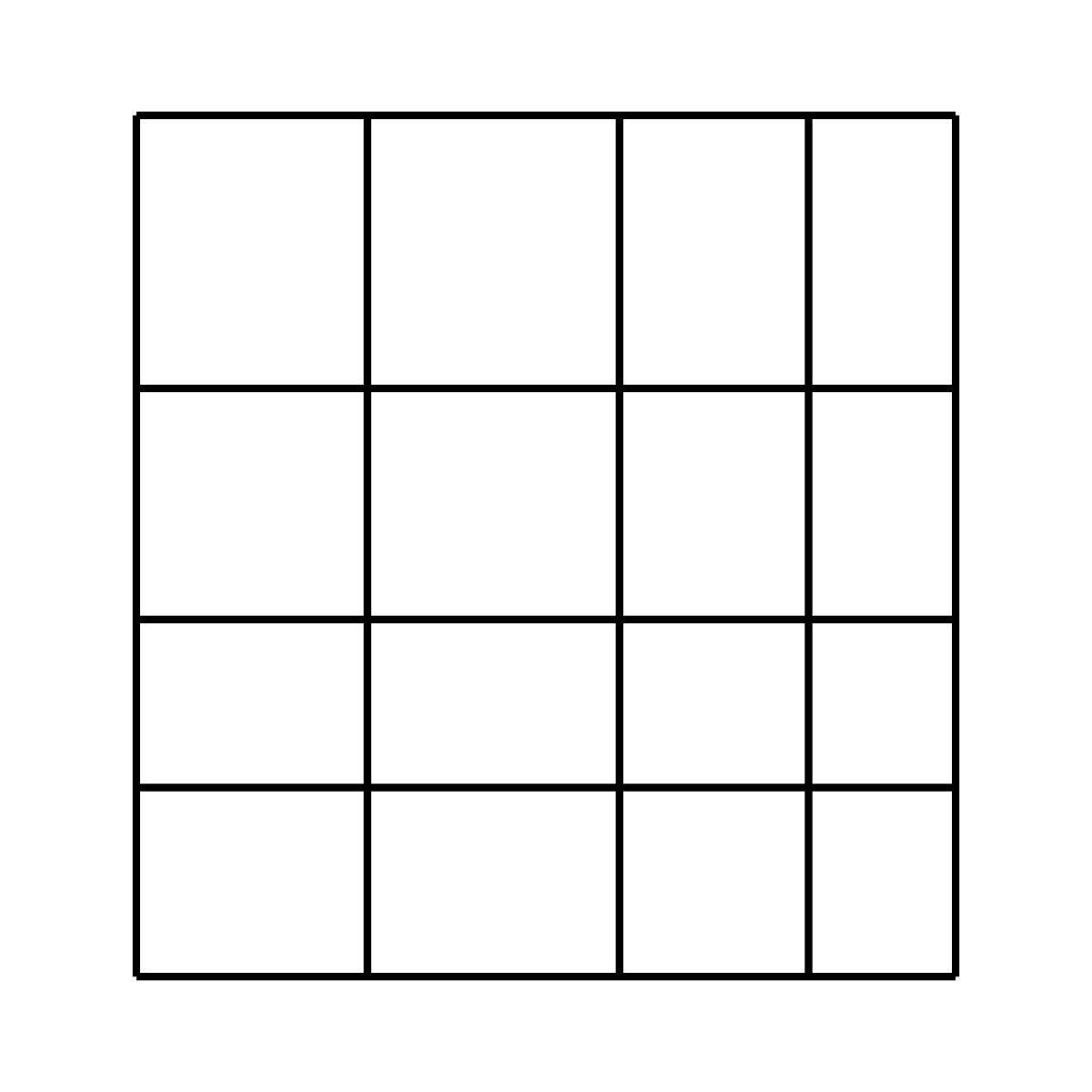}
        \caption{$\mathcal{T}^D_h$}
        \label{fig: 2.3.3}
    \end{subfigure}
    \caption{Staggered quadrilateral meshes.}
    \label{fig: 2.3}
\end{figure}

\begin{remark}
    We adopt uniformly oriented diagonals in the triangulation of the Cartesian grid. The analysis in Section \ref{sec 6} indicates that this specification is necessary to achieve the second-order consistency of the discrete convection term.
\end{remark}
\begin{remark} \label{rmk: mesh regularity}
    The Cartesian grid considered in this work is not restricted to be uniform. All theoretical developments and numerical experiments are carried out for non-uniform grids under the regularity assumption:
    \begin{align}
        C_1 \leq \frac{h^y_{j^+}}{h^x_{i^+}} \leq C_2, \label{eq: mesh regularity}
    \end{align}
    where $C_1$ and $C_2$ are positive constants independent of $h$.
\end{remark}
\begin{remark}
    This work is developed on two-dimensional Cartesian grids. Its extension to three dimensions is not a direct transfer and requires a new mesh design. A promising route is the Kuhn-Freudenthal subdivision, which splits each cube into six tetrahedra symmetrically along the three coordinate directions; a suitable grouping of these tetrahedra may yield a staggered structure analogous to the diagonal splitting used here. We leave this for future work.
\end{remark}

Throughout the paper, we write $A \lesssim B$ (resp. $A \gtrsim B$) to mean that $A \leq CB$ (resp. $A \geq CB$) for a positive constant $C$ independent of the mesh size $h$ and the viscosity $\nu$ but possibly depending on the grid-regularity constants $C_1,C_2$ in Remark~\ref{rmk: mesh regularity}. Besides, we write $A=O(h^r)$ to represent $A \lesssim h^r$.

\section{SDG$_0$ Method for Stokes Equations} \label{sec 3}

In this section, we develop a SDG$_0$ discretization for Stokes equations on the staggered meshes introduced in the previous section, and establish its fundamental properties.

Consider the Stokes equations:
\begin{align*}
    - \nu \Delta \mathbf{u} + \nabla p &= \mathbf{f}, \\
    \nabla \cdot \mathbf{u} &= 0,
\end{align*}
with conditions
\begin{align}
    \mathbf{u} |_{\partial \Omega} &= 0, \label{eq: Stokes condition 1} \\
    \int_{\Omega} p \, dx \, dy &= 0, \label{eq: Stokes condition 2}
\end{align}
where $\mathbf{u} = (u^x, u^y)^{\mathsf{T}}$ is the velocity vector, $p$ is the pressure, $\nu$ is the kinematic viscosity and $\mathbf{f}$ is the body force.
Let $\boldsymbol{\sigma}$ denote the velocity gradient tensor, i.e., $\boldsymbol{\sigma} = \begin{pmatrix} \sigma^{x, x} & \sigma^{x, y} \\ \sigma^{y, x} & \sigma^{y, y} \end{pmatrix} = \begin{pmatrix} \frac{\partial u^x}{\partial x} & \frac{\partial u^x}{\partial y} \\ \frac{\partial u^y}{\partial x} & \frac{\partial u^y}{\partial y} \end{pmatrix}$. Then we can rewrite the Stokes equations into the following mixed formulation:
\begin{align}
    \boldsymbol{\sigma} - \nabla \mathbf{u} &= 0, \label{eq: Stokes 2.1} \\
    - \nu \nabla \cdot \boldsymbol{\sigma} + \nabla p &= \mathbf{f}, \label{eq: Stokes 2.2} \\
    \nabla \cdot \mathbf{u} &= 0. \label{eq: Stokes 2.3}
\end{align}

Let $L^2 (\Omega)$, $H^m (\Omega)$ and $W^m_\infty (\Omega)$ denote the standard Lebesgue space, the Sobolev space of order $m$, and the Sobolev space with bounded derivatives up to order $m$, respectively. Their norms are denoted by $\| \cdot \|_{L^2}$, $\| \cdot \|_{H^m}$ and $\| \cdot \|_{W^m_\infty}$. Let $(\cdot, \cdot)$ denote the $L^2$ inner product. Throughout the paper, we assume that the exact solution $(\mathbf{u}, p)$ possesses sufficient regularity as measured by the aforementioned norms.

Define zero-order finite-element spaces with staggered continuity:
\begin{align*}
    \Sigma_h &= \{ \boldsymbol{\tau}_h : \boldsymbol{\tau}_h |_T \in \mathcal{P}_0 (T)^{2 \times 2}, \ \forall T \in \mathcal{T}_h; \ \llbracket \boldsymbol{\tau}_h \mathbf{n} \cdot \mathbf{t} \rrbracket_{|e} = 0, \ \forall e \in \mathring{\mathcal{E}}^H_h \cup \mathring{\mathcal{E}}^V_h; \ \llbracket \boldsymbol{\tau}_h \mathbf{n} \rrbracket_{|e} = 0, \ \forall e \in \mathcal{E}^D_h \}, \\
    U_h &= \{ \mathbf{v}_h : \mathbf{v}_h |_T \in \mathcal{P}_0 (T)^2, \ \forall T \in \mathcal{T}_h; \ \llbracket \mathbf{v}_h \cdot \mathbf{n} \rrbracket_{|e} = 0, \ \forall e \in \mathring{\mathcal{E}}^H_h \cup \mathring{\mathcal{E}}^V_h \}, \\
    P_h &= \{ q_h : q_h |_T \in \mathcal{P}_0 (T), \ \forall T \in \mathcal{T}_h; \ \llbracket q_h \rrbracket_{|e} = 0, \ \forall e \in \mathcal{E}^D_h \},
\end{align*}
where $\mathcal{P}_0 (T)$ is the space of constant functions on the element $T$. Impose the conditions \eqref{eq: Stokes condition 1}--\eqref{eq: Stokes condition 2} on the spaces:
\begin{align*}
    U_{h, 0} &= \{ \mathbf{v}_h \in U_h : \mathbf{v}_h \cdot \mathbf{n} = 0 \text{ on } \partial \Omega \}, \\
    P_{h, 0} &= \{ q_h \in P_h : \int_{\Omega} q_h \, dx \, dy = 0 \}.
\end{align*}
\textbf{SDG$_0$ Scheme}: find $\boldsymbol{\sigma}_h \in \Sigma_h$, $\mathbf{u}_h \in U_{h, 0}$ and $p_h \in P_{h, 0}$ such that
\begin{align}
    (\boldsymbol{\sigma}_h, \boldsymbol{\tau}_h) - B^*_h (\mathbf{u}_h, \boldsymbol{\tau}_h) &= 0, & &\forall \boldsymbol{\tau}_h \in \Sigma_h, \label{eq: SDG 1} \\
    \nu B_h (\boldsymbol{\sigma}_h, \mathbf{v}_h) - b^*_h (p_h, \mathbf{v}_h) &= (\mathbf{f}, \mathbf{v}_h), & &\forall \mathbf{v}_h \in U_{h, 0}, \label{eq: SDG 2} \\
    b_h (\mathbf{u}_h, q_h) &= 0, & &\forall q_h \in P_{h, 0}, \label{eq: SDG 3}
\end{align}
where
\begin{align*}
    B^*_h (\mathbf{v}_h, \boldsymbol{\tau}_h) &= \sum_{e \in \mathcal{E}^H_h \cup \mathcal{E}^V_h} \int_e \mathbf{v}_h \cdot \mathbf{n} \llbracket \boldsymbol{\tau}_h \mathbf{n} \cdot \mathbf{n} \rrbracket \, ds, \\
    B_h (\boldsymbol{\tau}_h, \mathbf{v}_h) &= - \sum_{e \in \mathcal{E}^H_h \cup \mathcal{E}^V_h} \int_e \boldsymbol{\tau}_h \mathbf{n} \cdot \mathbf{t} \llbracket \mathbf{v}_h \cdot \mathbf{t} \rrbracket \, ds - \sum_{e \in \mathcal{E}^D_h} \int_e \boldsymbol{\tau}_h \mathbf{n} \cdot \llbracket \mathbf{v}_h \rrbracket \, ds, \\
    b^*_h (q_h, \mathbf{v}_h) &= - \sum_{e \in \mathcal{E}^D_h} \int_e q_h \llbracket \mathbf{v}_h \cdot \mathbf{n} \rrbracket \, ds, \\
    b_h (\mathbf{v}_h, q_h) &= \sum_{e \in \mathcal{E}^H_h \cup \mathcal{E}^V_h} \int_e \mathbf{v}_h \cdot \mathbf{n} \llbracket q_h \rrbracket \, ds.
\end{align*}
\begin{remark}
    For the boundary condition \eqref{eq: Stokes condition 1}, we only impose the normal component of velocity to be zero on the boundary, while the tangential component is weakly enforced through the bilinear forms.
\end{remark}
\begin{remark}
    To highlight the novelty of our scheme, we compare it with the classical SDG method \cite{kim2013staggered, zhao2019staggered} at the lowest order. Figure~\ref{fig: 3.1} shows the distribution of the degrees of freedom on a single square block, from which two essential differences can be seen. First, regarding the mesh, the classical SDG method connects the centroid of each square to its four vertices resulting in a composite mesh of four sub-triangles, whereas our scheme uses a single diagonal splitting and thus produces only two triangles. Second, regarding the staggered continuity, the velocity, velocity gradient, and pressure are placed at different locations with different continuity configurations.
    \begin{figure}[htbp]
        \centering
        \begin{subfigure}[h]{0.48\textwidth}
            \centering
            \includegraphics{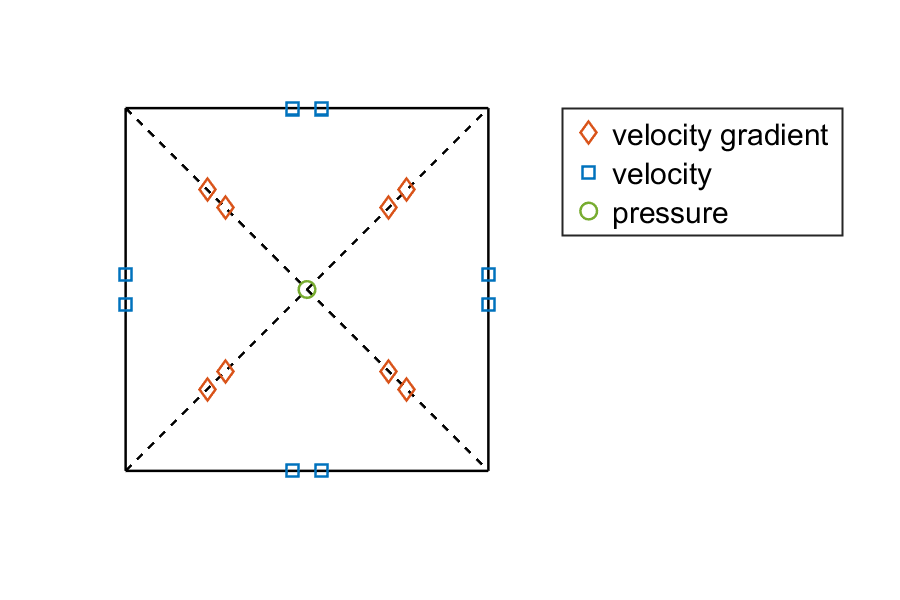}
            \caption{Classical SDG}
            \label{fig: 3.1.1}
        \end{subfigure}
        \begin{subfigure}[h]{0.48\textwidth}
            \centering
            \includegraphics{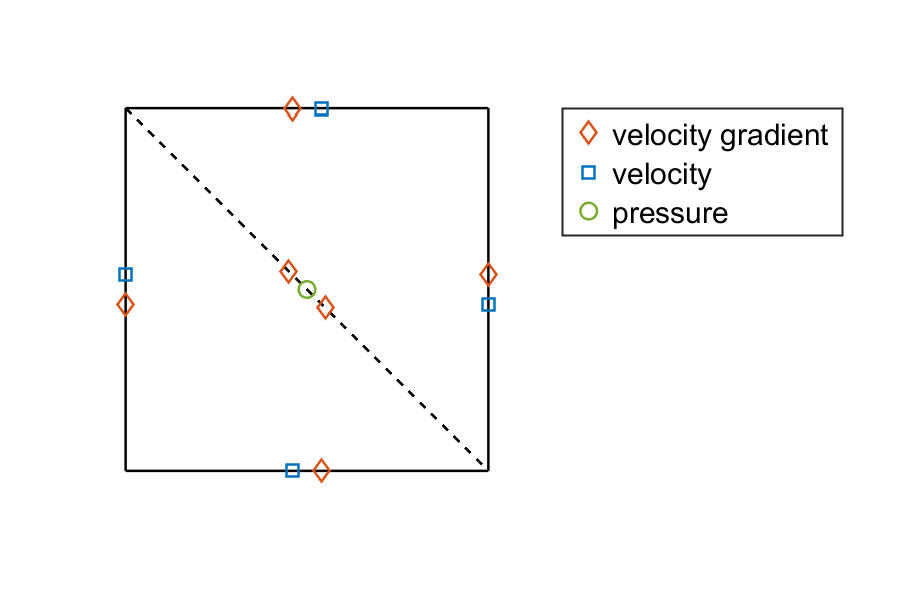}
            \caption{Present scheme}
            \label{fig: 3.1.2}
        \end{subfigure}
        \caption{Distribution of degrees of freedom}
        \label{fig: 3.1}
    \end{figure}
\end{remark}

By applying integration by parts, we can derive the bilinear forms are consistent:
\begin{align*}
    (\nabla \mathbf{u}, \boldsymbol{\tau}_h) &= B^*_h (\mathbf{u}, \boldsymbol{\tau}_h), & &\forall \boldsymbol{\tau}_h \in \Sigma_h, \\
    - (\nabla \cdot \boldsymbol{\sigma}, \mathbf{v}_h) &= B_h (\boldsymbol{\sigma}, \mathbf{v}_h), & &\forall \mathbf{v}_h \in U_{h, 0}, \\
    - (\nabla p, \mathbf{v}_h) &= b^*_h (p, \mathbf{v}_h), & &\forall \mathbf{v}_h \in U_{h, 0}, \\
    (\nabla \cdot \mathbf{u}, q_h) &= b_h (\mathbf{u}, q_h), & &\forall q_h \in P_h,
\end{align*}
and adjoint:
\begin{align*}
     B_h (\boldsymbol{\tau}_h, \mathbf{v}_h) &= B^*_h (\mathbf{v}_h, \boldsymbol{\tau}_h), & &\forall \boldsymbol{\tau}_h \in \Sigma_h, \ \forall \mathbf{v}_h \in U_h, \\
     b_h (\mathbf{v}_h, q_h) &= b^*_h (q_h, \mathbf{v}_h), & &\forall \mathbf{v}_h \in U_h, \ \forall q_h \in P_h.
\end{align*}
It follows the consistency of the $\mathrm{SDG_0}$ scheme, i.e.,
\begin{align*}
    (\boldsymbol{\sigma}, \boldsymbol{\tau}_h) - B^*_h (\mathbf{u}, \boldsymbol{\tau}_h) &= 0, & &\forall \boldsymbol{\tau}_h \in \Sigma_h, \\
    \nu B_h (\boldsymbol{\sigma}, \mathbf{v}_h) - b^*_h (p, \mathbf{v}_h) &= (\mathbf{f}, \mathbf{v}_h), & &\forall \mathbf{v}_h \in U_{h, 0}, \\
    b_h (\mathbf{u}, q_h) &= 0, & &\forall q_h \in P_{h, 0}.
\end{align*}

Define SDG norms: for $\boldsymbol{\tau}_h \in \Sigma_h$ and $\mathbf{v}_h \in U_h$,
\begin{align*}
    \| \boldsymbol{\tau}_h \|^2_{L^2, h} &= \sum_{e \in \mathcal{E}^H_h \cup \mathcal{E}^V_h} h_e \| \boldsymbol{\tau}_h \mathbf{n} \cdot \mathbf{t} \|^2_{L^2(e)} + \sum_{e \in \mathcal{E}^D_h} h_e \| \boldsymbol{\tau}_h \mathbf{n} \|^2_{L^2(e)}, \\
    \| \mathbf{v}_h \|^2_{H^1, h} &= \sum_{e \in \mathcal{E}^H_h \cup \mathcal{E}^V_h} h^{-1}_e \| \llbracket \mathbf{v}_h \cdot \mathbf{t} \rrbracket \|^2_{L^2(e)} + \sum_{e \in \mathcal{E}^D_h} h^{-1}_e \| \llbracket \mathbf{v}_h \rrbracket \|^2_{L^2(e)}.
\end{align*}
Then we have $B_h (\boldsymbol{\tau}_h, \mathbf{v}_h) \lesssim \| \boldsymbol{\tau}_h \|_{L^2, h} \| \mathbf{v}_h \|_{H^1, h}$. Scaling argument \cite{zhao2020newa} implies $\| \cdot \|_{L^2, h}$ is equivalent to $\| \cdot \|_{L^2}$, i.e.,
\begin{align}
    \| \boldsymbol{\tau}_h \|_{L^2} \lesssim \| \boldsymbol{\tau}_h \|_{L^2, h} \lesssim \| \boldsymbol{\tau}_h \|_{L^2}, \quad \forall \boldsymbol{\tau}_h \in \Sigma_h. \label{eq: norm equivalent}
\end{align}

Define the degrees of freedom (DoF) for $\Sigma_h$:
\begin{align*}
    \Phi^\Sigma_e (\boldsymbol{\tau}) &= \int_e \boldsymbol{\tau} \mathbf{n} \cdot \mathbf{t} \, ds, & &\forall e \in \mathcal{E}^H_h \cup \mathcal{E}^V_h, \\
    \Phi^{\Sigma, x}_e (\boldsymbol{\tau}) &= \int_e \tau^{n, x} \, ds, & &\forall e \in \mathcal{E}^D_h, \\
    \Phi^{\Sigma, y}_e (\boldsymbol{\tau}) &= \int_e \tau^{n, y} \, ds, & &\forall e \in \mathcal{E}^D_h,
\end{align*}
where $\tau^{n, x}$ and $\tau^{n, y}$ denote the components of $\boldsymbol{\tau}\mathbf{n}$ in the $x$- and $y$-directions, respectively, i.e.,
\begin{align*}
    \tau^{n, x} = n^x \tau^{x, x} + n^y \tau^{x, y}, \quad \tau^{n, y} = n^x \tau^{y, x} + n^y \tau^{y, y}.
\end{align*}
The degrees of freedom for $U_h$ and $P_h$ are defined as follows:
\begin{align*}
    \Phi^U_e (\mathbf{v}) &= \int_e \mathbf{v} \cdot \mathbf{n} \, ds, & &\forall e \in \mathcal{E}^H_h \cup \mathcal{E}^V_h, \\
    \Phi^P_e (q) &= \int_e q \, ds, & &\forall e \in \mathcal{E}^D_h.
\end{align*}
Let $\Pi^\Sigma_h$, $\Pi^U_h$ and $\Pi^P_h$ denote the canonical interpolation operators (associated with the above DoFs) onto $\Sigma_h$, $U_h$ and $P_h$, respectively. It is evident that the interpolation operators and the bilinear forms are compatible, i.e.,
\begin{align*}
    B^*_h (\mathbf{u} - \Pi^U_h \mathbf{u}, \mathbf{\tau}_h) &= 0, & &\forall \boldsymbol{\tau}_h \in \Sigma_h, \\
    B_h (\boldsymbol{\sigma} - \Pi^\Sigma_h \boldsymbol{\sigma}, \mathbf{v}_h) &= 0, & &\forall \mathbf{v}_h \in U_h, \\
    b^*_h (p - \Pi^P_h p, \mathbf{v}_h) &= 0, & &\forall \mathbf{v}_h \in U_h, \\
    b_h (\mathbf{u} - \Pi^U_h \mathbf{u}, q_h) &= 0, & &\forall q_h \in P_h.
\end{align*}
For any $\mathbf{v}$ satisfying the boundary condition \eqref{eq: Stokes condition 1}, we have $\Pi^U_h \mathbf{v} \in U_{h, 0}$.

The discrete inf-sup conditions are established in the following lemmas.
\begin{lemma} \label{lem: inf-sup condition 1}
    There exists a positive constant $\beta$ independent of mesh size $h$ such that
    \begin{align*}
        \sup_{\boldsymbol{\tau}_h \in \Sigma_h} \frac{B_h (\boldsymbol{\tau}_h, \mathbf{v}_h)}{\| \boldsymbol{\tau}_h \|_{L^2}} \geq \beta \| \mathbf{v}_h \|_{H^1, h}, & &\forall \mathbf{v}_h \in U_h.
    \end{align*}
\end{lemma}
\begin{proof}
    It suffices to prove that for any $\mathbf{v}_h \in U_h$, there exists $\boldsymbol{\tau}_h \in \Sigma_h$ such that $B_h (\boldsymbol{\tau}_h, \mathbf{v}_h) = \| \mathbf{v}_h \|^2_{H^1, h}$ and $\| \boldsymbol{\tau}_h \|_{L^2} \lesssim \| \mathbf{v}_h \|_{H^1, h}$. According to the DoFs of $\Sigma_h$, we can construct $\boldsymbol{\tau}_h$ such that
    \begin{align*}
        \int_e \boldsymbol{\tau}_h \mathbf{n} \cdot \mathbf{t} \, ds &= -h^{-1}_e \int_e \llbracket \mathbf{v}_h \cdot \mathbf{t} \rrbracket \, ds, & &\forall e \in \mathcal{E}^H_h \cup \mathcal{E}^V_h, \\
        \int_e \boldsymbol{\tau}_h \mathbf{n} \, ds &= -h^{-1}_e \int_e \llbracket \mathbf{v}_h \rrbracket \, ds, & &\forall e \in \mathcal{E}^D_h.
    \end{align*}
    It directly follows that $B_h (\boldsymbol{\tau}_h, \mathbf{v}_h) = \| \mathbf{v}_h \|^2_{H^1, h}$ and $\| \boldsymbol{\tau}_h \|_{L^2, h} = \| \mathbf{v}_h \|_{H^1, h}$. Then applying the norm equivalence \eqref{eq: norm equivalent} completes the proof.
\end{proof}
\begin{lemma} \label{lem: inf-sup condition 2}
    There exists a positive constant $\beta$ independent of mesh size $h$ such that
    \begin{align*}
        \sup_{\mathbf{v}_h \in U_{h, 0}} \frac{b_h (\mathbf{v}_h, q_h)}{\| \mathbf{v}_h \|_{H^1, h}} \geq \beta \| q_h \|_{L^2}, & &\forall q_h \in P_{h, 0}.
    \end{align*}
\end{lemma}
\begin{proof}
    It is well-known that for the continuous spaces $H^1_0 (\Omega)^2 = \{ \mathbf{v} \in H^1 (\Omega)^2 : \mathbf{v} |_{\partial \Omega} = 0 \}$ and $L^2_0 (\Omega) = \{ q \in L^2(\Omega) : \int_{\Omega} q \, dx \, dy = 0 \}$, it holds that
    \begin{align*}
        \sup_{\mathbf{v} \in H^1_0 (\Omega)^2} \frac{(\nabla \cdot \mathbf{v}, q)}{\| \mathbf{v} \|_{H^1}} \geq \beta' \| q \|_{L^2}, \quad \forall q \in L^2_0 (\Omega),
    \end{align*}
    where $\beta'$ is a $h$-independent positive constant. Given $q_h \in P_{h, 0} \subset L^2_0 (\Omega)$, for any $\mathbf{v} \in H^1_0 (\Omega)^2$, we take $\mathbf{v}_h = \Pi^U_h \mathbf{v} \in U_{h, 0}$. By the properties of the projection operator,
    \begin{align*}
        b_h (\mathbf{v}_h, q_h) = b_h (\Pi^U_h \mathbf{v}, q_h) = b_h (\mathbf{v}, q_h) = (\nabla \cdot \mathbf{v}, q_h).
    \end{align*}
    To establish the inf-sup condition, we only need to bound $\| \mathbf{v}_h \|_{H^1, h}$ in terms of $\| \mathbf{v} \|_{H^1}$. By Lemma \ref{lem: inf-sup condition 1},
    \begin{align*}
        \| \mathbf{v}_h \|_{H^1, h} \lesssim \sup_{\boldsymbol{\tau}_h \in \Sigma_h} \frac{B_h (\boldsymbol{\tau}_h, \mathbf{v}_h)}{\| \boldsymbol{\tau}_h \|_{L^2}}.
    \end{align*}
    By the adjoint and consistent property of the bilinear form,
    \begin{align*}
        B_h (\boldsymbol{\tau}_h, \mathbf{v}_h) = B_h (\boldsymbol{\tau}_h, \Pi^U_h \mathbf{v}) = B^*_h (\Pi^U_h \mathbf{v}, \boldsymbol{\tau}_h) = B^*_h (\mathbf{v}, \boldsymbol{\tau}_h) = (\nabla \mathbf{v}, \boldsymbol{\tau}_h).
    \end{align*}
    Therefore,
    \begin{align*}
        \| \mathbf{v}_h \|_{H^1, h} \lesssim \sup_{\boldsymbol{\tau}_h \in \Sigma_h} \frac{(\nabla \mathbf{v}, \boldsymbol{\tau}_h)}{\| \boldsymbol{\tau}_h \|_{L^2}} \lesssim \| \mathbf{v} \|_{H^1},
    \end{align*}
    which completes the proof.
\end{proof}

\section{Pointwise Formulation and Local Static Condensation} \label{sec 4}

In this section, we derive a explicit pointwise formulation of the SDG$_0$ scheme and apply local static condensation to eliminate the diagonal entries of the velocity gradient.

Specify the direction of the normal and tangential unit vectors on edges as follows:
\begin{align*}
    \mathbf{n} &= (0, 1)^{\mathsf{T}}, \quad \mathbf{t} = (1, 0)^{\mathsf{T}}, & &\text{on } e_{i^+, j}, \\
    \mathbf{n} &= (1, 0)^{\mathsf{T}}, \quad \mathbf{t} = (0, 1)^{\mathsf{T}}, & &\text{on } e_{i, j^+}, \\
    \mathbf{n} &= (n^x_{i^+, j^+}, n^y_{i^+, j^+})^{\mathsf{T}}, & &\text{on } e_{i^+, j^+},
\end{align*}
where
\begin{align*}
    n^x_{i^+, j^+} = \frac{h^y_{j^+}}{l_{i^+, j^+}}, \quad 
    n^y_{i^+, j^+} = \frac{h^x_{i^+}}{l_{i^+, j^+}}.
\end{align*}

The finite element spaces can be explicitly characterized as follows:
\begin{align*}
    \Sigma_h &= \{ \boldsymbol{\tau}_h : \tau^{x, y}_h |_T \in \mathcal{P}_0 (T), \ \forall T \in \mathcal{T}^H_h; \ \tau^{y, x}_h |_T \in \mathcal{P}_0 (T), \ \forall T \in \mathcal{T}^V_h; \\ &\quad \ \tau^{n, x}_h |_T, \tau^{n, y}_h |_T \in \mathcal{P}_0 (T), \ \forall T \in \mathcal{T}^D_h \}, \\
    U_h &= \{ \mathbf{v}_h : v^x_h |_T \in \mathcal{P}_0 (T), \ \forall T \in \mathcal{T}^V_h; \ v^y_h |_T \in \mathcal{P}_0 (T), \ \forall T \in \mathcal{T}^H_h \}, \\
    P_h &= \{ q_h : q_h |_T \in \mathcal{P}_0 (T), \ \forall T \in \mathcal{T}^D_h \},
\end{align*}
where we specify $\mathbf{n} = (n^x_{i^+, j^+}, n^y_{i^+, j^+})^{\mathsf{T}}$ for $\tau^{n, x}_h, \tau^{n, y}_h$ in $T_{i^+, j^+}$.
For $\boldsymbol{\tau}_h \in \Sigma_h$, we denote its constant components as follows:
\begin{align*}
    \tau^{x, y}_h \equiv \tau^{x, y}_{i^+, j} \text{ in } T_{i^+, j}, \quad
    \tau^{n, x}_h \equiv \tau^{n, x}_{i^+, j^+} \text{ in } T_{i^+, j^+}, \quad
    \tau^{y, x}_h \equiv \tau^{y, x}_{i, j^+} \text{ in } T_{i, j^+}, \quad
    \tau^{n, y}_h \equiv \tau^{n, y}_{i^+, j^+} \text{ in } T_{i^+, j^+}.
\end{align*}
Then, $\tau^{x, x}_h$ and $\tau^{y, y}_h$ can be expressed as
\begin{align} \label{eq: velocity gradient relation}
    \tau^{x, x}_h =
    \begin{cases}
        \frac{\tau^{n, x}_{i^+, j^+} - n^y_{i^+, j^+} \tau^{x, y}_{i^+, j}}{n^x_{i^+, j^+}} & \text{in } T^-_{i^+, j^+}, \\
        \frac{\tau^{n, x}_{i^+, j^+} - n^y_{i^+, j^+} \tau^{x, y}_{i^+, j + 1}}{n^x_{i^+, j^+}} & \text{in } T^+_{i^+, j^+},
    \end{cases} \quad
    \tau^{y, y}_h =
    \begin{cases}
        \frac{\tau^{n, y}_{i^+, j^+} - n^x_{i^+, j^+} \tau^{y, x}_{i, j^+}}{n^y_{i^+, j^+}} & \text{in } T^-_{i^+, j^+}, \\
        \frac{\tau^{n, y}_{i^+, j^+} - n^x_{i^+, j^+} \tau^{y, x}_{i + 1, j^+}}{n^y_{i^+, j^+}} & \text{in } T^+_{i^+, j^+}.
    \end{cases}
\end{align}
The constant components of $\mathbf{v}_h \in U_h$ and $q_h \in P_h$ are denoted as follows:
\begin{gather*}
    v^x_h \equiv v^x_{i, j^+} \text{ in } T_{i, j^+}, \quad
    v^y_h \equiv v^y_{i^+, j} \text{ in } T_{i^+, j}, \\
    q_h \equiv q_{i^+, j^+} \text{ in } T_{i^+, j^+}.
\end{gather*}

Let $\chi_T$ denote the characteristic function on the element $T$. Take basis functions associated with DoF $\Phi^\Sigma_e$ as follows:
\begin{align*}
    \boldsymbol{\phi}^\Sigma_{e_{i^+, j}} &=
    \begin{pmatrix}
        - \frac{h^x_{i^+}}{h^y_{j^+}} \chi_{T^-_{i^+, j^+}} \!\! - \frac{h^x_{i^+}}{h^y_{j^-}} \chi_{T^+_{i^+, j^-}} & \chi_{T_{i^+, j}} \\
        0 & 0
    \end{pmatrix}, \\
    \boldsymbol{\phi}^\Sigma_{e_{i, j^+}} &=
    \begin{pmatrix}
        0 & 0 \\
        \chi_{T_{i, j^+}} & - \frac{h^y_{j^+}}{h^x_{i^+}} \chi_{T^-_{i^+, j^+}} \!\! - \frac{h^y_{j^+}}{h^x_{i^-}} \chi_{T^+_{i^-, j^+}}
    \end{pmatrix},
\end{align*}
where the definition outside the domain $\Omega$ is specified to be ignored. Take basis functions associated with DoF $\Phi^{\Sigma, x}_e$ and $\Phi^{\Sigma, y}_e$ as follows:
\begin{align*}
    \boldsymbol{\phi}^{\Sigma, x}_{e_{i^+, j^+}} =
    \begin{pmatrix}
        \frac{\chi_{T_{i^+, j^+}}}{n^x_{i^+, j^+}} & 0 \\
        0 & 0
    \end{pmatrix}, \quad
    \boldsymbol{\phi}^{\Sigma, y}_{e_{i^+, j^+}} =
    \begin{pmatrix}
        0 & 0 \\
        0 & \frac{\chi_{T_{i^+, j^+}}}{n^y_{i^+, j^+}}
    \end{pmatrix}.
\end{align*}
Basis functions associated with DoF $\Phi^U_e$ and $\Phi^P_e$ are
\begin{gather*}
    \boldsymbol{\phi}^U_{e_{i, j^+}} = (\chi_{T_{i, j^+}}, 0)^{\mathsf{T}}, \quad
    \boldsymbol{\phi}^U_{e_{i^+, j}} = (0, \chi_{T_{i^+, j}})^{\mathsf{T}}, \\
    \phi^P_{e_{i^+, j^+}} = \chi_{T_{i^+, j^+}}.
\end{gather*}
It is easy to check that each basis function yields a non-zero value when evaluated against its associated DoF and zero when evaluated against all other DoF.
We choose these basis functions so that the expansion coefficients of the finite element function coincide with its values on the associated locations, i.e., for $\boldsymbol{\tau}_h \in \Sigma_h$, $\mathbf{v}_h \in U_h$ and $q_h \in P_h$,
\begin{align*}
    \boldsymbol{\tau}_h &= \sum_{(i^+, j) \in \Lambda^H_h} \tau^{x, y}_{i^+, j} \boldsymbol{\phi}^\Sigma_{e_{i^+, j}} + \sum_{(i, j^+) \in \Lambda^V_h} \tau^{y, x}_{i, j^+} \boldsymbol{\phi}^\Sigma_{e_{i, j^+}} + \sum_{(i^+, j^+) \in \Lambda^D_h} (\tau^{n, x}_{i^+, j^+} \boldsymbol{\phi}^{\Sigma, x}_{e_{i^+, j^+}} + \tau^{n, y}_{i^+, j^+} \boldsymbol{\phi}^{\Sigma, y}_{e_{i^+, j^+}}), \\
    \mathbf{v}_h &= \sum_{(i, j^+) \in \Lambda^V_h} v^x_{i, j^+} \boldsymbol{\phi}^U_{e_{i, j^+}} + \sum_{(i^+, j) \in \Lambda^H_h} v^y_{i^+, j} \boldsymbol{\phi}^U_{e_{i^+, j}}, \\
    q_h &= \sum_{(i^+, j^+) \in \Lambda^D_h} q_{i^+, j^+} \phi^P_{e_{i^+, j^+}}.
\end{align*}

Split $\Sigma_h$ into diagonal and off-diagonal parts:
\begin{align*}
    \Sigma^D_h = \mathrm{span} \{ \boldsymbol{\phi}^{\Sigma, x}_{e_{i^+, j^+}}, \boldsymbol{\phi}^{\Sigma, y}_{e_{i^+, j^+}} \}, \quad
    \Sigma^O_h = \mathrm{span} \{ \boldsymbol{\phi}^\Sigma_{e_{i^+, j}}, \boldsymbol{\phi}^\Sigma_{e_{i, j^+}} \}.
\end{align*}
For any $\boldsymbol{\tau}_h \in \Sigma_h$, we write $\boldsymbol{\tau}_h = \boldsymbol{\tau}^D_h + \boldsymbol{\tau}^O_h$, where $\boldsymbol{\tau}^D_h \in \Sigma^D_h$, $\boldsymbol{\tau}^O_h \in \Sigma^O_h$. 
Let $\mathsf{A}$, $\mathsf{B}$ and $\mathsf{b}$ be the matrices corresponding to the bilinear forms $(\boldsymbol{\sigma}_h, \boldsymbol{\tau}_h)$, $B_h (\boldsymbol{\sigma}_h, \mathbf{v}_h)$ and $b_h (\mathbf{u}_h, q_h)$, respectively. Under the above space decomposition, $\mathsf{A}$ and $\mathsf{B}$ admit the following block structures:
\begin{align*}
    \mathsf{A} &= \begin{pmatrix}
        \mathsf{A}_{DD} & \mathsf{A}_{DO} \\
        \mathsf{A}_{OD} & \mathsf{A}_{OO}
    \end{pmatrix}, \quad
    \mathsf{B} = (\mathsf{B}_D, \mathsf{B}_O).
\end{align*}
Let $\mathsf{F}$ be the vector corresponding to the linear functional $(\mathbf{f}, \mathbf{v}_h)$. Let $\mathsf{\sigma}_D$, $\mathsf{\sigma}_O$, $\mathsf{u}$ and $\mathsf{p}$ be the coefficient vectors of $\boldsymbol{\sigma}^D_h$, $\boldsymbol{\sigma}^O_h$, $\mathbf{u}_h$ and $p_h$ with respect to the chosen basis functions. Then the algebraic system of the SDG$_0$ scheme can be expressed as
\begin{align*}
    \begin{pmatrix}
        \mathsf{A}_{DD} & \mathsf{A}_{DO} & - \mathsf{B}^{\mathsf{T}}_D & 0 \\
        \mathsf{A}_{OD} & \mathsf{A}_{OO} & - \mathsf{B}^{\mathsf{T}}_O & 0 \\
        \nu \mathsf{B}_D & \nu \mathsf{B}_O & 0 & - \mathsf{b}^{\mathsf{T}} \\
        0 & 0 & \mathsf{b} & 0
    \end{pmatrix}
    \begin{pmatrix}
        \mathsf{\sigma}_D \\
        \mathsf{\sigma}_O \\
        \mathsf{u} \\
        \mathsf{p}
    \end{pmatrix}
    =
    \begin{pmatrix}
        0 \\
        0 \\
        \mathsf{F} \\
        0
    \end{pmatrix}.
\end{align*}
Noting that $\mathsf{A}_{DD}$ is diagonal and can be eliminated directly, we obtain a reduced system:
\begin{align*}
    \begin{pmatrix}
        \tilde{\mathsf{A}} & - \tilde{\mathsf{B}}^{\mathsf{T}} & 0 \\
        \nu \tilde{\mathsf{B}} & \nu \tilde{\mathsf{a}} & - \mathsf{b}^{\mathsf{T}} \\
        0 & \mathsf{b} & 0
    \end{pmatrix}
    \begin{pmatrix}
        \mathsf{\sigma}_O \\
        \mathsf{u} \\
        \mathsf{p}
    \end{pmatrix}
    =
    \begin{pmatrix}
        0 \\
        \mathsf{F} \\
        0
    \end{pmatrix},
\end{align*}
where
\begin{align*}
    \tilde{\mathsf{A}} &= \mathsf{A}_{OO} - \mathsf{A}_{OD} \mathsf{A}^{-1}_{DD} \mathsf{A}_{DO}, \\
    \tilde{\mathsf{B}} &= \mathsf{B}_O - \mathsf{B}_D \mathsf{A}^{-1}_{DD} \mathsf{A}_{DO}, \\
    \tilde{\mathsf{a}} &= \mathsf{B}_D \mathsf{A}^{-1}_{DD} \mathsf{B}^{\mathsf{T}}_D.
\end{align*}
After some algebra (see Appendix~\ref{sec: appendix} for details), the condensed SDG$_0$ scheme can be written explicitly as follows:
\begin{empheq}[left=\empheqlbrace]{align}
    &\begin{multlined} \label{eq: condensed SDG 1}
        [h^x_{i^+} h^y_j + \frac{(h^x_{i^+})^3}{4 h^y_{j^+}} + \frac{(h^x_{i^+})^3}{4 h^y_{j^-}}] \sigma^{x, y}_{i^+, j} - \frac{(h^x_{i^+})^3}{4 h^y_{j^+}} \sigma^{x, y}_{i^+, j + 1} - \frac{(h^x_{i^+})^3}{4 h^y_{j^-}} \sigma^{x, y}_{i^+, j - 1} \\ - \frac{1}{2} h^x_{i^+} u^x_{i, j^+} + \frac{1}{2} h^x_{i^+} u^x_{i, j^-} - \frac{1}{2} h^x_{i^+} u^x_{i + 1, j^+} + \frac{1}{2} h^x_{i^+} u^x_{i + 1, j^-} = 0, \quad \text{for } (i^+, j) \in \Lambda^H_h,
    \end{multlined} \\
    &\begin{multlined} \label{eq: condensed SDG 2}
        (h^x_i h^y_{j^+} + \frac{(h^y_{j^+})^3}{4 h^x_{i^+}} + \frac{(h^y_{j^+})^3}{4 h^x_{i^-}}) \sigma^{y, x}_{i, j^+} - \frac{(h^y_{j^+})^3}{4 h^x_{i^+}} \sigma^{y, x}_{i + 1, j^+} - \frac{(h^y_{j^+})^3}{4 h^x_{i^-}} \sigma^{y, x}_{i - 1, j^+} \\ - \frac{1}{2} h^y_{j^+} u^y_{i^+, j} + \frac{1}{2} h^y_{j^+} u^y_{i^-, j} - \frac{1}{2} h^y_{j^+} u^y_{i^+, j + 1} + \frac{1}{2} h^y_{j^+} u^y_{i^-, j + 1} = 0, \quad \text{for } (i, j^+) \in \Lambda^V_h,
    \end{multlined} \\
    &\begin{multlined} \label{eq: condensed SDG 3}
        -\nu [\frac{h^x_{i^+}}{2} \sigma^{x, y}_{i^+, j + 1} - \frac{h^x_{i^+}}{2} \sigma^{x, y}_{i^+, j} + \frac{h^x_{i^-}}{2} \sigma^{x, y}_{i^-, j + 1} - \frac{h^x_{i^-}}{2} \sigma^{x, y}_{i^-, j} + \frac{h^y_{j^+}}{h^x_{i^+}} u^x_{i + 1, j^+} - \frac{2 h^x_i h^y_{j^+}}{h^x_{i^+} h^x_{i^-}} u^x_{i, j^+} \\ + \frac{h^y_{j^+}}{h^x_{i^-}} u^x_{i - 1, j^+}] + h^y_{j^+} p_{i^+, j^+} - h^y_{j^+} p_{i^-, j^+} = \int_{T_{i, j^+}} f^x \, dx \, dy, \quad \text{for } (i, j^+) \in \mathring{\Lambda}^V_h,
    \end{multlined} \\
    &\begin{multlined} \label{eq: condensed SDG 4}
        -\nu [\frac{h^y_{j^+}}{2} \sigma^{y, x}_{i + 1, j^+} - \frac{h^y_{j^+}}{2} \sigma^{y, x}_{i, j^+} + \frac{h^y_{j^-}}{2} \sigma^{y, x}_{i + 1, j^-} - \frac{h^y_{j^-}}{2} \sigma^{y, x}_{i, j^-} + \frac{h^x_{i^+}}{h^y_{j^+}} u^y_{i^+, j + 1} - \frac{2 h^y_j h^x_{i^+}}{h^y_{j^+} h^y_{j^-}} u^y_{i^+, j} \\ + \frac{h^x_{i^+}}{h^y_{j^-}} u^y_{i^+, j - 1}] + h^x_{i^+} p_{i^+, j^+} - h^x_{i^+} p_{i^+, j^-} = \int_{T_{i^+, j}} f^y \, dx \, dy, \quad \text{for } (i^+, j) \in \mathring{\Lambda}^H_h,
    \end{multlined} \\
    &h^y_{j^+} u^x_{i + 1, j^+} - h^y_{j^+} u^x_{i, j^+} + h^x_{i^+} u^y_{i^+, j + 1} - h^x_{i^+} u^y_{i^+, j} = 0, \quad \text{for } (i^+, j^+) \in \Lambda^D_h, \label{eq: condensed SDG 5}
\end{empheq}
where the terms outside the domain $\Omega$ are specified to be ignored.
\begin{remark} \label{rmk: mass lumping}
    The off-diagonal components of the velocity gradient that remain after the static condensation can be further eliminated by standard mass lumping without loss of accuracy, as confirmed by numerical experiments in Section~\ref{sec 7}.
\end{remark}

After static condensation, the diagonal entries $\sigma^{x,x}_h$ and $\sigma^{y,y}_h$ are no longer present in this system. Rather than recovering them from the condensation relation, we simply replace them by the central difference quotients of the velocity, i.e., define
\begin{align} \label{eq: reconstructed diagonal velocity gradient}
    \breve{\sigma}^{x, x}_h \equiv \breve{\sigma}^{x, x}_{i^+, j^+} = d_x u^x_{i^+, j^+} \text{ in } T_{i^+, j^+}, \quad
    \breve{\sigma}^{y, y}_h \equiv \breve{\sigma}^{y, y}_{i^+, j^+} = d_y u^y_{i^+, j^+} \text{ in } T_{i^+, j^+},
\end{align}
where
\begin{align*}
    d_x u^x_{i^+, j^+} = \frac{u^x_{i + 1, j^+} - u^x_{i, j^+}}{h^x_{i^+}}, \quad
    d_y u^y_{i^+, j^+} = \frac{u^y_{i^+, j + 1} - u^y_{i^+, j}}{h^y_{j^+}}.
\end{align*}
Let $\breve{\boldsymbol{\sigma}}_h$ denote the reconstructed velocity gradient tensor:
\begin{align} \label{eq: reconstructed velocity gradient}
    \breve{\boldsymbol{\sigma}}_h =
    \begin{pmatrix}
        \breve{\sigma}^{x, x}_h & \sigma^{x, y}_h \\
        \sigma^{y, x}_h & \breve{\sigma}^{y, y}_h
    \end{pmatrix}.
\end{align}
\begin{remark}
    The reconstruction \eqref{eq: reconstructed diagonal velocity gradient} is motivated not only by the second-order accuracy of the central difference quotients, but also by an exact averaging property: $d_x u^x_{i^+, j^+}$ and $d_y u^y_{i^+, j^+}$ coincide exactly with the cell average of $\sigma^{x, x}_h$ and $\sigma^{y, y}_h$, respectively (see Appendix~\ref{sec: appendix} for details).
\end{remark}

In the subsequent analysis, we interpret $\breve{\sigma}^{x, x}_{i^+, j^+}$, $\sigma^{x, y}_{i^+, j}$, $\sigma^{y, x}_{i, j^+}$, $\breve{\sigma}^{y, y}_{i^+, j^+}$, $u^x_{i, j^+}$, $u^y_{i^+, j}$ and $p_{i^+, j^+}$ as the approximations of $\sigma^{x, x} (x_{i^+}, y_{j^+})$, $\sigma^{x, y} (x_{i^+}, y_j)$, $\sigma^{y, x} (x_i, y_{j^+})$, $\sigma^{y, y} (x_{i^+}, y_{j^+})$, $u^x (x_i, y_{j^+})$, $u^y (x_{i^+}, y_j)$ and $p (x_{i^+}, y_{j^+})$, respectively. Accordingly, we define the pointwise norms:
\begin{align*}
    \| \breve{\boldsymbol{\sigma}}_h \|^2_{l^2} &= \sum_{(i^+, j^+) \in \Lambda^D_h} h^x_{i^+} h^y_{j^+} (\breve{\sigma}^{x, x}_{i^+, j^+})^2 + \sum_{(i^+, j) \in \Lambda^H_h} h^x_{i^+} h^y_j (\sigma^{x, y}_{i^+, j})^2 \\ &\quad + \sum_{(i, j^+) \in \Lambda^V_h} h^x_i h^y_{j^+} (\sigma^{y, x}_{i, j^+})^2 + \sum_{(i^+, j^+) \in \Lambda^D_h} h^x_{i^+} h^y_{j^+} (\breve{\sigma}^{y, y}_{i^+, j^+})^2, \\
    \| \mathbf{u}_h \|^2_{l^2} &= \sum_{(i, j^+) \in \Lambda^V_h} h^x_i h^y_{j^+} (u^x_{i, j^+})^2 + \sum_{(i^+, j) \in \Lambda^H_h} h^x_{i^+} h^y_j (u^y_{i^+, j})^2, \\
    \| p_h \|^2_{l^2} &= \sum_{(i^+, j^+) \in \Lambda^D_h} h^x_{i^+} h^y_{j^+} (p_{i^+, j^+})^2,
\end{align*}
where $\| \breve{\boldsymbol{\sigma}}_h \|^2_{l^2}$ can be decomposed as diagonal and off-diagonal seminorms:
\begin{align*}
    | \boldsymbol{\sigma}_h |^2_{l^2, off} &= \sum_{(i^+, j) \in \Lambda^H_h} h^x_{i^+} h^y_j (\sigma^{x, y}_{i^+, j})^2 + \sum_{(i, j^+) \in \Lambda^V_h} h^x_i h^y_{j^+} (\sigma^{y, x}_{i, j^+})^2, \\
    | \mathbf{u}_h |^2_{h^1, diag} &= \sum_{(i^+, j^+) \in \Lambda^D_h} h^x_{i^+} h^y_{j^+} [ (d_x u^x_{i^+, j^+})^2 + (d_y u^y_{i^+, j^+})^2 ].
\end{align*}
The relationship between the pointwise norms and the SDG norms is summarized in the following lemma.
\begin{lemma} \label{lem: norm relation}
    For $\boldsymbol{\tau}_h \in \Sigma_h$, $\mathbf{v}_h \in U_h$ and $q_h \in P_h$, it holds that
    \begin{align}
        | \boldsymbol{\tau}_h |_{l^2, off} &\lesssim \| \boldsymbol{\tau}_h \|_{L^2}, \label{eq: norm relation 1} \\
        | \mathbf{v}_h |_{h^1, diag} &\lesssim \| \mathbf{v}_h \|_{H^1, h}, \label{eq: norm relation 2} \\
        \| \mathbf{v}_h \|_{l^2} &= \| \mathbf{v}_h \|_{L^2}, \label{eq: norm relation 3} \\
        \| q_h \|_{l^2} &= \| q_h \|_{L^2}. \label{eq: norm relation 4}
    \end{align}
    For $\mathbf{v}_h \in U_{h, 0}$, it holds that
    \begin{align} \label{eq: norm relation 5}
        \| \mathbf{v}_h \|_{l^2} \lesssim | \mathbf{v}_h |_{h^1, diag},
    \end{align}
    and consequently,
    \begin{align*}
        \| \mathbf{v}_h \|_{L^2} \lesssim \| \mathbf{v}_h \|_{H^1, h}.
    \end{align*}
\end{lemma}
\begin{proof}
    \eqref{eq: norm relation 1}, \eqref{eq: norm relation 3} and \eqref{eq: norm relation 4} are evident from the definition of norms. As for \eqref{eq: norm relation 2},
    \begin{align*}
        \| \mathbf{v}_h \|^2_{H^1, h} &\geq \sum_{e \in \mathcal{E}^D_h} h^{-1}_e \| \llbracket v^x_h \rrbracket \|^2_{L^2(e)} + \sum_{e \in \mathcal{E}^D_h} h^{-1}_e \| \llbracket v^y_h \rrbracket \|^2_{L^2(e)} \\
        &= \sum_{(i^+, j^+) \in \Lambda^D_h} [ (v^x_{i + 1, j^+} - v^x_{i, j^+})^2 + (v^y_{i^+, j + 1} - v^y_{i^+, j})^2 ] \\
        &\gtrsim | \mathbf{v}_h |_{h^1, diag}.
    \end{align*}
    Now we prove \eqref{eq: norm relation 5}. For $\mathbf{v}_h \in U_{h, 0}$,
    \begin{align*}
        v^x_{l, j^+} = \sum_{i = 0}^{l - 1} h^x_{i^+} d_x v^x_{i^+, j^+} \lesssim [ \sum_{i = 0}^{n_x - 1} h^x_{i^+} (d_x v^x_{i^+, j^+})^2 ]^{1/2}.
    \end{align*}
    Then,
    \begin{align*}
        \sum_{l = 1}^{n_x - 1} \sum_{j = 0}^{n_y - 1} h^x_l h^y_{j^+} (v^x_{l, j^+})^2 \lesssim \sum_{i = 0}^{n_x - 1} \sum_{j = 0}^{n_y - 1} h^x_{i^+} h^y_{j^+} (d_x v^x_{i^+, j^+})^2 \leq | \mathbf{v}_h |_{h^1, diag}.
    \end{align*}
    Similarly,
    \begin{align*}
        \sum_{i = 0}^{n_x - 1} \sum_{m = 1}^{n_y - 1} h^x_{i^+} h^y_m (v^y_{i^+, m})^2 \lesssim | \mathbf{v}_h |_{h^1, diag}.
    \end{align*}
    Adding the above two inequalities yields \eqref{eq: norm relation 5}.
\end{proof}
\begin{remark}
    The pointwise formulation \eqref{eq: condensed SDG 1}--\eqref{eq: condensed SDG 5} is reminiscent of the classical MAC scheme for the Stokes equations: the velocity is discretized at edge midpoints, whereas the pressure is located at cell centers. This formulation also yields a compact stencil that couples only nearest-neighbor unknowns, as illustrated in Figure~\ref{fig: 4.1}. The essential difference lies in the treatment of the off-diagonal velocity gradient. The MAC scheme approximates it by difference quotients of the velocity at the grid vertices, which are only first-order accurate on non-uniform grids \cite{rui2017stability}. Our scheme instead treats these components as independent unknowns at the edge midpoints, which are second-order superconvergent even on non-uniform grids (see Section~\ref{sec 5}). This benefit extends to the Navier-Stokes equations, where the velocity gradient enters the convection term, and its accurate approximation helps achieve overall second-order accuracy (see Section~\ref{sec 6}).
    \begin{figure}[htbp]
        \centering
        \includegraphics[width=\textwidth]{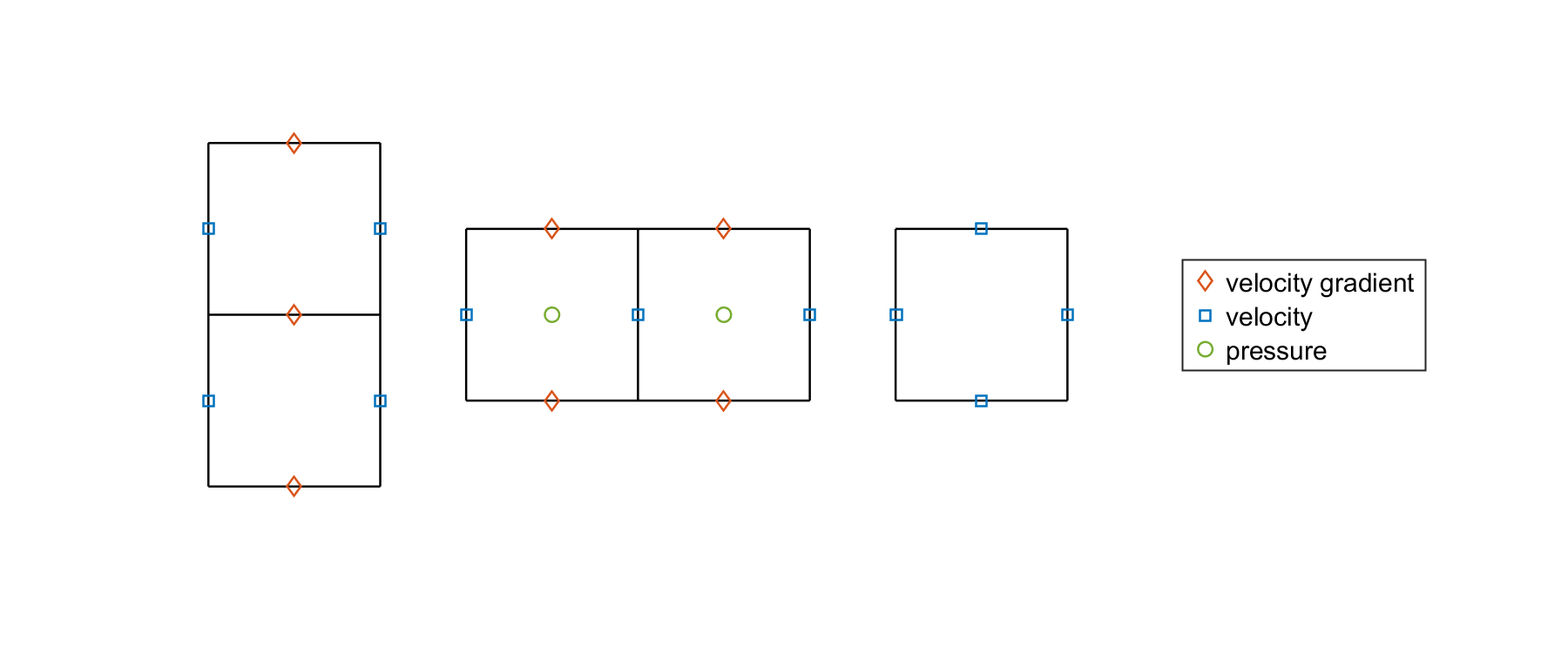}
        \caption{Stencils of the pointwise formulation: \eqref{eq: condensed SDG 1}, \eqref{eq: condensed SDG 3}, and \eqref{eq: condensed SDG 5} from left to right.}
        \label{fig: 4.1}
    \end{figure}
\end{remark}

\section{Error Analysis} \label{sec 5}

In this section, we develop a detailed pointwise error analysis of the SDG$_0$ scheme, and rigorously prove its pressure robustness and second-order superconvergence.
Throughout, we assume that the exact solution is sufficiently smooth, namely
\begin{align}
    \mathbf{u} \in [W^3_\infty (\Omega)]^2, \qquad p \in W^2_\infty (\Omega). \label{eq: solution regularity}
\end{align}
\begin{remark} \label{rmk: solution regularity}
    The regularity assumed above is stronger than that required in the standard SDG error analysis, due to the pointwise nature of our argument. Such regularity assumption is common in superconvergence analyses on Cartesian grids \cite{li2015superconvergence, rui2017stability}. Under the low regularity, the method then reverts to the conventional SDG framework and achieves the usual convergence rate. Since this analysis is standard, we do not repeat it here; refer to \cite{kim2013staggered, zhao2019staggered, zhao2020newa} for details.
\end{remark}

Define auxiliary variables:
\begin{align}
    \tilde{\boldsymbol{\sigma}}_h &= \Pi^\Sigma_h \boldsymbol{\sigma}, \label{eq: auxiliary variables 1} \\
    \tilde{\mathbf{u}}_h &= \Pi^U_h \mathbf{u} - \boldsymbol{\delta}_h, \label{eq: auxiliary variables 2} \\
    \tilde{p}_h &= \Pi^P_h p - \int_{\Omega} \Pi^P_h p \, dx \, dy, \label{eq: auxiliary variables 3}
\end{align}
where the correction term $\boldsymbol{\delta}_h \in U_h$ has the following components:
\begin{align*}
    \delta^x_{i, j^+} &= \frac{(h^y_{j^+})^2}{6} \cdot \frac{\partial^2 u^x}{\partial y^2} (x_i, y_{j^+}), \\
    \delta^y_{i^+, j} &= \frac{(h^x_{i^+})^2}{6} \cdot \frac{\partial^2 u^y}{\partial x^2} (x_{i^+}, y_j).
\end{align*}
It is clear that $\tilde{\mathbf{u}}_h \in U_{h, 0}$ and $\tilde{p}_h \in P_{h, 0}$.
The auxiliary variables are 2-nd order approximations of the exact solution at specific points, as stated in the following lemma.
\begin{lemma} \label{lem: projection error}
    For $\tilde{\boldsymbol{\sigma}}_h$, $\tilde{\mathbf{u}}_h$ and $\tilde{p}_h$ defined in \eqref{eq: auxiliary variables 1}--\eqref{eq: auxiliary variables 3}, it holds that
    \begin{align*}
        \tilde{\sigma}^{x, y}_{i^+, j} &= \sigma^{x, y} (x_{i^+}, y_j) + O(h^2) \| \mathbf{u} \|_{W^3_\infty}, &
        \tilde{\sigma}^{n, x}_{i^+, j^+} &= \sigma^{n, x} (x_{i^+}, y_{j^+}) + O(h^2) \| \mathbf{u} \|_{W^3_\infty}, \\
        \tilde{\sigma}^{y, x}_{i, j^+} &= \sigma^{y, x} (x_i, y_{j^+}) + O(h^2) \| \mathbf{u} \|_{W^3_\infty}, &
        \tilde{\sigma}^{n, y}_{i^+, j^+} &= \sigma^{n, y} (x_{i^+}, y_{j^+}) + O(h^2) \| \mathbf{u} \|_{W^3_\infty}, \\
        \tilde{u}^x_{i, j^+} &= u^x (x_i, y_{j^+}) + O(h^2) \| \mathbf{u} \|_{W^2_\infty}, &
        \tilde{u}^y_{i^+, j} &= u^y (x_{i^+}, y_j) + O(h^2) \| \mathbf{u} \|_{W^2_\infty}, \\
        \tilde{p}_{i^+, j^+} &= p (x_{i^+}, y_{j^+}) + O(h^2) \| p \|_{W^2_\infty}.
    \end{align*}
\end{lemma}
\begin{proof}
    The first six equations are direct consequences from the definition of the projection operator. As for the last equation, we write $\tilde{p}_h$ in detail:
    \begin{align*}
        \tilde{p}_{i^+, j^+} &= \frac{1}{l_{i^+, j^+}} \int_{e_{i^+, j^+}} p \, ds - \sum_{(i^+, j^+) \in \Lambda^D_h} \int_{T_{i^+, j^+}} \left( \frac{1}{l_{i^+, j^+}} \int_{e_{i^+, j^+}} p \, ds \right) \, dx \, dy.
    \end{align*}
    Since $p$ is zero mean,
    \begin{align*}
        \tilde{p}_{i^+, j^+} - p (x_{i^+, j^+})
        &= \frac{1}{l_{i^+, j^+}} \int_{e_{i^+, j^+}} p \, ds - p (x_{i^+, j^+}) \\ 
        &\quad - \sum_{(i^+, j^+) \in \Lambda^D_h} \int_{T_{i^+, j^+}} \left[ \frac{1}{l_{i^+, j^+}} \int_{e_{i^+, j^+}} p \, ds - p (x_{i^+, j^+}) \right] \, dx \, dy \\ 
        &\quad - \sum_{(i^+, j^+) \in \Lambda^D_h} \int_{T_{i^+, j^+}} [ p (x_{i^+, j^+}) - p ] \, dx \, dy \\
        &= O(h^2) \| p \|_{W^2_\infty}.
    \end{align*}
\end{proof}

Now we analyze the truncation errors by substituting the auxiliary variables into the SDG$_0$ scheme.
\begin{lemma} \label{lem: truncation error 1}
    Let $\tilde{\boldsymbol{\sigma}}_h$ and $\tilde{\mathbf{u}}_h$ be defined in \eqref{eq: auxiliary variables 1}--\eqref{eq: auxiliary variables 2}. For basis functions $\boldsymbol{\phi}^{\Sigma, x}_e$ and $\boldsymbol{\phi}^{\Sigma, y}_e$, it holds that
    \begin{align}
        (\tilde{\boldsymbol{\sigma}}_h, \boldsymbol{\phi}^{\Sigma, x}_{e_{i^+, j^+}}) - B^*_h (\tilde{\mathbf{u}}_h, \boldsymbol{\phi}^{\Sigma, x}_{e_{i^+, j^+}}) &= (\mathbf{R}^{\Sigma, x}_{e_{i^+, j^+}}, \boldsymbol{\phi}^{\Sigma, x}_{e_{i^+, j^+}}), \label{eq: truncation error 1.1} \\
        (\tilde{\boldsymbol{\sigma}}_h, \boldsymbol{\phi}^{\Sigma, y}_{e_{i^+, j^+}}) - B^*_h (\tilde{\mathbf{u}}_h, \boldsymbol{\phi}^{\Sigma, y}_{e_{i^+, j^+}}) &= (\mathbf{R}^{\Sigma, y}_{e_{i^+, j^+}}, \boldsymbol{\phi}^{\Sigma, y}_{e_{i^+, j^+}}), \label{eq: truncation error 1.2}
    \end{align}
    where
    \begin{align*}
        \mathbf{R}^{\Sigma, x}_{e_{i^+, j^+}} =
        \begin{pmatrix}
            R^{\Sigma, x}_{i^+, j^+} \chi_{T_{i^+, j^+}} & 0 \\
            0 & 0
        \end{pmatrix}, \quad
        \mathbf{R}^{\Sigma, y}_{e_{i^+, j^+}} =
        \begin{pmatrix}
            0 & 0 \\
            0 & R^{\Sigma, y}_{i^+, j^+} \chi_{T_{i^+, j^+}}
        \end{pmatrix},
    \end{align*}
    with $R^{\Sigma, x}_{i^+, j^+} = O(h^2) \| \mathbf{u} \|_{W^3_\infty}$ and $R^{\Sigma, y}_{i^+, j^+} = O(h^2) \| \mathbf{u} \|_{W^3_\infty}$.
\end{lemma}
\begin{proof}
    We only show the proof of \eqref{eq: truncation error 1.1}, and \eqref{eq: truncation error 1.2} can be proven similarly. For brevity, we scale the basis functions:
    \begin{align*}
        \boldsymbol{\phi}^{\Sigma, x}_{e_{i^+, j^+}} =
        \begin{pmatrix}
            \chi_{T_{i^+, j^+}} & 0 \\
            0 & 0
        \end{pmatrix}, \quad
        \boldsymbol{\phi}^{\Sigma, y}_{e_{i^+, j^+}} =
        \begin{pmatrix}
            0 & 0 \\
            0 & \chi_{T_{i^+, j^+}}
        \end{pmatrix}.
    \end{align*}
    By consistency of the scheme, the left side of \eqref{eq: truncation error 1.1} can be rewritten as
    \begin{align*}
        &(\tilde{\boldsymbol{\sigma}}_h, \boldsymbol{\phi}^{\Sigma, x}_{e_{i^+, j^+}})- B^*_h (\tilde{\mathbf{u}}_h, \boldsymbol{\phi}^{\Sigma, x}_{e_{i^+, j^+}}) \\
        &= (\tilde{\boldsymbol{\sigma}}_h -\boldsymbol{\sigma}, \boldsymbol{\phi}^{\Sigma, x}_{e_{i^+, j^+}}) - B^*_h (\tilde{\mathbf{u}}_h - \mathbf{u}, \boldsymbol{\phi}^{\Sigma, x}_{e_{i^+, j^+}}) \\
        &= \int_{T_{i^+, j^+}} (\tilde{\sigma}^{x, x}_h - \sigma^{x, x}) \, dx \, dy + \int_{e_{i, j^+}} (\tilde{u}^x_h - u^x) \, ds - \int_{e_{i + 1, j^+}} (\tilde{u}^x_h - u^x) \, ds \\
        &= \int_{T_{i^+, j^+}} \frac{(\tilde{\sigma}^{n, x}_h - \sigma^{n, x}) - n^y_{i^+, j^+} (\tilde{\sigma}^{x, y}_h - \sigma^{x, y})}{n^x_{i^+, j^+}} \, dx \, dy \\ &\quad + \int_{e_{i, j^+}} (\tilde{u}^x_h - u^x) \, ds - \int_{e_{i + 1, j^+}} (\tilde{u}^x_h - u^x) \, ds.
    \end{align*}
    Denote the components in the truncation error as follows:
    \begin{align*}
        R^{\Sigma, x}_{i^+, j^+} (\sigma^{n, x}) &= \int_{T_{i^+, j^+}} (\tilde{\sigma}^{n, x}_h - \sigma^{n, x}) \, dx \, dy, \\
        R^{\Sigma, x}_{i^+, j^+} (\sigma^{x, y}) &= \int_{T_{i^+, j^+}} (\tilde{\sigma}^{x, y}_h - \sigma^{x, y}) \, dx \, dy, \\
        R^{\Sigma, x}_{i^+, j^+} (u^x) &= \int_{e_{i, j^+}} (\tilde{u}^x_h - u^x) \, ds - \int_{e_{i + 1, j^+}} (\tilde{u}^x_h - u^x) \, ds.
    \end{align*}
    For the first component,
    \begin{align*}
        R^{\Sigma, x}_{i^+, j^+} (\sigma^{n, x}) &= h^x_{i^+} h^y_{j^+} [ \tilde{\sigma}^{n, x}_{i^+, j^+} - \sigma^{n, x} (x_{i^+}, y_{j^+}) ] + \int_{T_{i^+, j^+}} [ \sigma^{n, x} (x_{i^+}, y_{j^+}) - \sigma^{n, x} ] \, dx \, dy \\
        &= O(h^4) \| \mathbf{u} \|_{W^3_\infty}.
    \end{align*}
    For the second component,
    \begin{align*}
        R^{\Sigma, x}_{i^+, j^+} (\sigma^{x, y}) &= h^x_{i^+} h^y_{j^+} \left[ \frac{1}{2} (\tilde{\sigma}^{x, y}_{i^+, j} - \sigma^{x, y} (x_{i^+}, y_j)) + \frac{1}{2} (\tilde{\sigma}^{x, y}_{i^+, j + 1} - \sigma^{x, y} (x_{i^+}, y_{j + 1})) \right. \\ &\quad + \left. \frac{1}{2} \sigma^{x, y} (x_{i^+}, y_j) + \frac{1}{2} \sigma^{x, y} (x_{i^+}, y_{j + 1}) - \sigma^{x, y} (x_{i^+}, y_{j^+}) \right] \\ &\quad + \int_{T_{i^+, j^+}} [\sigma^{x, y} (x_{i^+}, y_{j^+}) - \sigma^{x, y}] \, dx \, dy \\
        &= O(h^4) \| \mathbf{u} \|_{W^3_\infty}.
    \end{align*}
    For the third component,
    \begin{align*}
        R^{\Sigma, x}_{i^+, j^+} (u^x) = \frac{(h^y_{j^+})^3}{6} [\frac{\partial^2 u^x}{\partial y^2} (x_{i + 1}, y_{j^+}) - \frac{\partial^2 u^x}{\partial y^2} (x_i, y_{j^+})] = O(h^4) \| \mathbf{u} \|_{W^3_\infty}.
    \end{align*}
    The right side of \eqref{eq: truncation error 1.1} equals $h^x_{i^+} h^y_{j^+} R^{\Sigma, x}_{i^+, j^+}$. Comparing the both sides of \eqref{eq: truncation error 1.1} yields
    \begin{align*}
        R^{\Sigma, x}_{i^+, j^+} &= \frac{1}{h^x_{i^+} h^y_{j^+}} \left[ \frac{1}{n^x_{i^+, j^+}} R^{\Sigma, x}_{i^+, j^+} (\sigma^{n, x}) - \frac{n^y_{i^+, j^+}}{n^x_{i^+, j^+}} R^{\Sigma, x}_{i^+, j^+} (\sigma^{x, y}) + R^{\Sigma, x}_{i^+, j^+} (u^x) \right] \\
        &= O(h^2) \| \mathbf{u} \|_{W^3_\infty},
    \end{align*}
    which completes the proof.
\end{proof}

\begin{lemma} \label{lem: truncation error 2}
    Let $\tilde{\boldsymbol{\sigma}}_h$ and $\tilde{\mathbf{u}}_h$ be defined in \eqref{eq: auxiliary variables 1}--\eqref{eq: auxiliary variables 2}. For basis function $\boldsymbol{\phi}^{\Sigma}_e$, it holds that
    \begin{align}
        (\tilde{\boldsymbol{\sigma}}_h, \boldsymbol{\phi}^{\Sigma}_{e_{i^+, j}}) - B^*_h (\tilde{\mathbf{u}}_h, \boldsymbol{\phi}^{\Sigma}_{e_{i^+, j}}) &= (\mathbf{R}^{\Sigma}_{e_{i^+, j}}, \boldsymbol{\phi}^{\Sigma}_{e_{i^+, j}}), \label{eq: truncation error 2.1} \\
        (\tilde{\boldsymbol{\sigma}}_h, \boldsymbol{\phi}^{\Sigma}_{e_{i, j^+}}) - B^*_h (\tilde{\mathbf{u}}_h, \boldsymbol{\phi}^{\Sigma}_{e_{i, j^+}}) &= (\mathbf{R}^{\Sigma}_{e_{i, j^+}}, \boldsymbol{\phi}^{\Sigma}_{e_{i, j^+}}), \label{eq: truncation error 2.2}
    \end{align}
    where
    \begin{align*}
        \mathbf{R}^{\Sigma}_{e_{i^+, j}} =
        \begin{pmatrix}
            0 & R^{\Sigma}_{i^+, j} \chi_{T_{i^+, j}} \\
            0 & 0
        \end{pmatrix}, \quad
        \mathbf{R}^{\Sigma}_{e_{i, j^+}} =
        \begin{pmatrix}
            0 & 0 \\
            R^{\Sigma}_{i, j^+} \chi_{T_{i, j^+}} & 0
        \end{pmatrix},
    \end{align*}
    with $R^{\Sigma}_{i^+, j} = O(h^2) \| \mathbf{u} \|_{W^3_\infty}$ and $R^{\Sigma}_{i, j^+} = O(h^2) \| \mathbf{u} \|_{W^3_\infty}$.
\end{lemma}
\begin{proof}
    We only show the proof of \eqref{eq: truncation error 2.1} for $e_{i^+, j} \in \mathring{\mathcal{E}}^H_h$, and the remaining can be proven similarly. By consistency of the scheme, the left side of \eqref{eq: truncation error 2.1} can be rewritten as
    \begin{align*}
        &(\tilde{\boldsymbol{\sigma}}_h, \boldsymbol{\phi}^{\Sigma}_{e_{i^+, j}}) - B^*_h (\tilde{\mathbf{u}}_h, \boldsymbol{\phi}^{\Sigma}_{e_{i^+, j}}) \\
        &= (\tilde{\boldsymbol{\sigma}}_h - \boldsymbol{\sigma}, \boldsymbol{\phi}^{\Sigma}_{e_{i^+, j}}) - B^*_h (\tilde{\mathbf{u}}_h - \mathbf{u}, \boldsymbol{\phi}^{\Sigma}_{e_{i^+, j}}) \\
        &= \int_{T^-_{i^+, j^+}} [ - \frac{h^x_{i^+}}{h^y_{j^+}} (\tilde{\sigma}^{x, x}_h - \sigma^{x, x}) + \tilde{\sigma}^{x, y}_h - \sigma^{x, y} ] \, dx \, dy \\
        &\quad + \int_{T^+_{i^+, j^-}} [ - \frac{h^x_{i^+}}{h^y_{j^-}} (\tilde{\sigma}^{x, x}_h - \sigma^{x, x}) + \tilde{\sigma}^{x, y}_h - \sigma^{x, y} ] \, dx \, dy \\
        &\quad - \int_{e_{i, j^+}} \frac{h^x_{i^+}}{h^y_{j^+}} (\tilde{u}^x_h - u^x) \, ds + \int_{e_{i + 1, j^-}} \frac{h^x_{i^+}}{h^y_{j^-}} (\tilde{u}^x_h - u^x) \, ds \\
        &= \int_{T^-_{i^+, j^+}} \left[ - \frac{h^x_{i^+}}{h^y_{j^+}} \cdot \frac{(\tilde{\sigma}^{n, x}_h - \sigma^{n, x}) - n^y_{i^+, j^+} (\tilde{\sigma}^{x, y}_h - \sigma^{x, y})}{n^x_{i^+, j^+}} + \tilde{\sigma}^{x, y}_h - \sigma^{x, y} \right] \, dx \, dy \\
        &\quad + \int_{T^+_{i^+, j^-}} \left[ - \frac{h^x_{i^+}}{h^y_{j^-}} \cdot \frac{(\tilde{\sigma}^{n, x}_h - \sigma^{n, x}) - n^y_{i^+, j^-} (\tilde{\sigma}^{x, y}_h - \sigma^{x, y})}{n^x_{i^+, j^-}} + \tilde{\sigma}^{x, y}_h - \sigma^{x, y} \right] \, dx \, dy \\
        &\quad - \int_{e_{i, j^+}} \frac{h^x_{i^+}}{h^y_{j^+}} (\tilde{u}^x_h - u^x) \, ds + \int_{e_{i + 1, j^-}} \frac{h^x_{i^+}}{h^y_{j^-}} (\tilde{u}^x_h - u^x) \, ds.
    \end{align*}
    Denote the components in the truncation error as follows:
    \begin{align*}
        R^{\Sigma, +}_{i^+, j} (\sigma^{n, x}) &= \int_{T^-_{i^+, j^+}} (\tilde{\sigma}^{n, x}_h - \sigma^{n, x}) \, dx \, dy, \quad
        R^{\Sigma, -}_{i^+, j} (\sigma^{n, x}) = \int_{T^+_{i^+, j^-}} (\tilde{\sigma}^{n, x}_h - \sigma^{n, x}) \, dx \, dy, \\
        R^{\Sigma, +}_{i^+, j} (\sigma^{x, y}) &= \int_{T^-_{i^+, j^+}} (\tilde{\sigma}^{x, y}_h - \sigma^{x, y}) \, dx \, dy, \quad
        R^{\Sigma, -}_{i^+, j} (\sigma^{x, y}) = \int_{T^+_{i^+, j^-}} (\tilde{\sigma}^{x, y}_h - \sigma^{x, y}) \, dx \, dy, \\
        R^{\Sigma}_{i^+, j} (u^x) &= - \int_{e_{i, j^+}} \frac{h^x_{i^+}}{h^y_{j^+}} (\tilde{u}^x_h - u^x) \, ds + \int_{e_{i + 1, j^-}} \frac{h^x_{i^+}}{h^y_{j^-}} (\tilde{u}^x_h - u^x) \, ds.
    \end{align*}
    For the first component,
    \begin{align*}
        R^{\Sigma, +}_{i^+, j} (\sigma^{n, x}) &= \frac{h^x_{i^+} h^y_{j^+}}{2} [ \tilde{\sigma}^{n, x}_{i^+, j^+} - \sigma^{n, x} (x_{i^+}, y_{j^+}) ] + \int_{T^-_{i^+, j^+}} [ \sigma^{n, x} (x_{i^+}, y_{j^+}) - \sigma^{n, x} ] \, dx \, dy \\
        &= \frac{(h^x_{i^+})^2 h^y_{j^+}}{12} \cdot \frac{\partial \sigma^{n, x}}{\partial x} (x_{i^+}, y_{j^+}) + \frac{h^x_{i^+} (h^y_{j^+})^2}{12} \cdot \frac{\partial \sigma^{n, x}}{\partial y} (x_{i^+}, y_{j^+}) + O(h^4) \| \mathbf{u} \|_{W^3_\infty} \\
        &= \frac{(h^x_{i^+})^2 h^y_{j^+}}{12} [ n^x_{i^+, j^+} \frac{\partial^2 u^x}{\partial x^2} (x_{i^+}, y_{j^+}) + n^y_{i^+, j^+} \frac{\partial^2 u^x}{\partial x \partial y} (x_{i^+}, y_{j^+}) ] \\
        & \quad + \frac{h^x_{i^+} (h^y_{j^+})^2}{12} [ n^x_{i^+,j^+} \frac{\partial^2 u^x}{\partial y \partial x} (x_{i^+}, y_{j^+}) + n^y_{i^+, j^+} \frac{\partial^2 u^x}{\partial y^2} (x_{i^+}, y_{j^+}) ] + O(h^4) \| \mathbf{u} \|_{W^3_\infty}.
    \end{align*}
    Similarly, for the second component,
    \begin{align*}
        R^{\Sigma, -}_{i^+, j} (\sigma^{n, x}) &= - \frac{(h^x_{i^+})^2 h^y_{j^-}}{12} [ n^x_{i^+, j^-} \frac{\partial^2 u^x}{\partial x^2} (x_{i^+}, y_{j^-}) + n^y_{i^+, j^-} \frac{\partial^2 u^x}{\partial x \partial y} (x_{i^+}, y_{j^-}) ] \\
        &\quad - \frac{h^x_{i^+} (h^y_{j^-})^2}{12} [ n^x_{i^+, j^-} \frac{\partial^2 u^x}{\partial y \partial x} (x_{i^+}, y_{j^-}) + n^y_{i^+, j^-} \frac{\partial^2 u^x}{\partial y^2} (x_{i^+}, y_{j^-}) ] + O(h^4) \| \mathbf{u} \|_{W^3_\infty}.
    \end{align*}
    For the third component,
    \begin{align*}
        R^{\Sigma, +}_{i^+, j} (\sigma^{x, y}) &= \frac{h^x_{i^+} h^y_{j^+}}{2} [ \tilde{\sigma}^{x, y}_{i^+, j} - \sigma^{x, y} (x_{i^+}, y_j) ] + \int_{T^-_{i^+, j^+}} [ \sigma^{x, y} (x_{i^+}, y_j) - \sigma^{x, y} ] \, dx \, dy \\
        &= \frac{(h^x_{i^+})^2 h^y_{j^+}}{12} \cdot \frac{\partial^2 u^x}{\partial x \partial y} (x_{i^+}, y_j) - \frac{h^x_{i^+} (h^y_{j^+})^2}{6} \cdot \frac{\partial^2 u^x}{\partial y^2} (x_{i^+}, y_j) + O(h^4) \| \mathbf{u} \|_{W^3_\infty}.
    \end{align*}
    Similarly, for the fourth component,
    \begin{align*}
        R^{\Sigma, -}_{i^+, j} (\sigma^{x, y}) = - \frac{(h^x_{i^+})^2 h^y_{j^-}}{12} \cdot \frac{\partial^2 u^x}{\partial x \partial y} (x_{i^+}, y_{j^-}) + \frac{h^x_{i^+} (h^y_{j^-})^2}{6} \cdot \frac{\partial^2 u^x}{\partial y^2} (x_{i^+}, y_{j^-}) + O(h^4) \| \mathbf{u} \|_{W^3_\infty}.
    \end{align*}
    For the fifth component,
    \begin{align*}
        R^{\Sigma}_{i^+, j} (u^x) = \frac{h^x_{i^+} (h^y_{j^+})^2}{6} \cdot \frac{\partial^2 u^x}{\partial y^2} (x_i, y_{j^+}) - \frac{h^x_{i^+} (h^y_{j^-})^2}{6} \cdot \frac{\partial^2 u^x}{\partial y^2} (x_{i + 1}, y_{j^-}).
    \end{align*}
    The right side of \eqref{eq: truncation error 2.1} equals $h^x_{i^+} h^y_j R^{\Sigma}_{i^+, j}$. Comparing the both sides of \eqref{eq: truncation error 2.1} yields
    \begin{align*}
        R^{\Sigma}_{i^+, j} &= \frac{1}{h^x_{i^+} h^y_j} \left[ - \frac{h^x_{i^+}}{h^y_{j^+} n^x_{i^+, j^+}} R^{\Sigma, +}_{i^+, j} (\sigma^{n, x}) + (\frac{h^x_{i^+} n^y_{i^+, j^+}}{h^y_{j^+} n^x_{i^+, j^+}} + 1) R^{\Sigma, +}_{i^+, j} (\sigma^{x, y}) \right. \\
        &\quad \left. - \frac{h^x_{i^+}}{h^y_{j^-} n^x_{i^+, j^-}} R^{\Sigma, -}_{i^+, j} (\sigma^{n, x}) + (\frac{h^x_{i^+} n^y_{i^+, j^-}}{h^y_{j^-} n^x_{i^+, j^-}} + 1) R^{\Sigma, -}_{i^+, j} (\sigma^{x, y}) + R^{\Sigma}_{i^+, j} (u^x) \right] \\
        & = O(h^2) \| \mathbf{u} \|_{W^3_\infty},
    \end{align*}
    which completes the proof.
\end{proof}

\begin{lemma} \label{lem: truncation error 3}
    Let $\tilde{\mathbf{u}}_h$ be defined in \eqref{eq: auxiliary variables 1}. For basis function $\phi^P_e$, it holds that
    \begin{align}
        b_h (\tilde{\mathbf{u}}_h, \phi^P_{e_{i^+, j^+}}) = (R^P_{e_{i^+, j^+}}, \phi^P_{e_{i^+, j^+}}), \label{eq: truncation error 3}
    \end{align}
    where
    \begin{align*}
        R^P_{e_{i^+, j^+}} = R^P_{i^+, j^+} \chi_{T_{i^+, j^+}},
    \end{align*}
    with $R^P_{i^+, j^+} = O(h^2) \| \mathbf{u} \|_{W^3_\infty}$.
\end{lemma}
\begin{proof}
    Estimate the left side of \eqref{eq: truncation error 3}:
    \begin{align*}
        &b_h (\tilde{\mathbf{u}}_h, \phi^P_{e_{i^+, j^+}}) = b_h (\tilde{\mathbf{u}}_h - \mathbf{u}, \phi^P_{e_{i^+, j^+}}) \\
        &= \int_{e_{i + 1, j^+}} (\tilde{u}^x_h - u^x) \, ds - \int_{e_{i, j^+}} (\tilde{u}^x_h - u^x) \, ds + \int_{e_{i^+, j + 1}} (\tilde{u}^y_h - u^y) \, ds - \int_{e_{i^+, j}} (\tilde{u}^y_h - u^y) \, ds \\
        &= - \frac{(h^y_{j^+})^3}{6} [ \frac{\partial^2 u^x}{\partial y^2} (x_{i + 1, j^+}) - \frac{\partial^2 u^x}{\partial y^2} (x_{i, j^+}) ] - \frac{(h^x_{i^+})^3}{6} [ \frac{\partial^2 u^y}{\partial x^2} (x_{i^+, j + 1}) - \frac{\partial^2 u^y}{\partial x^2} (x_{i^+, j}) ] \\
        &= O(h^4) \| \mathbf{u} \|_{W^3_\infty}.
    \end{align*}
    The right side of \eqref{eq: truncation error 3} equals $h^x_{i^+} h^y_{j^+} R^P_{e_{i^+, j^+}}$, and thus we have $R^P_{i^+, j^+} = O(h^2) \| \mathbf{u} \|_{W^3_\infty}$.
\end{proof}

By Lemma \ref{lem: truncation error 1}, Lemma \ref{lem: truncation error 2} and Lemma \ref{lem: truncation error 3}, the auxiliary variables satisfy the following equations:
\begin{align}
    (\tilde{\boldsymbol{\sigma}}_h, \boldsymbol{\tau}_h) - B^*_h (\tilde{\mathbf{u}}_h, \boldsymbol{\tau}_h) &= (\mathbf{R}^\Sigma_h, \boldsymbol{\tau}_h), & &\forall \boldsymbol{\tau}_h \in \Sigma_h, \label{eq: auxiliary equation 1} \\
    \nu B_h (\tilde{\boldsymbol{\sigma}}_h, \mathbf{v}_h) - b^*_h (\tilde{p}_h, \mathbf{v}_h) &= (\mathbf{f}, \mathbf{v}_h), & &\forall \mathbf{v}_h \in U_{h, 0}, \label{eq: auxiliary equation 2} \\
    b_h (\tilde{\mathbf{u}}_h, q_h) &= (R^P_h, q_h), & &\forall q_h \in P_{h, 0}, \label{eq: auxiliary equation 3}
\end{align}
where
\begin{align*}
    \mathbf{R}^\Sigma_h &= \sum_{e \in \mathcal{E}^H_h \cup \mathcal{E}^V_h} R^\Sigma_e + \sum_{e \in \mathcal{E}^D_h} (R^{\Sigma, x}_e + R^{\Sigma, y}_e) \\
    &=
    \begin{cases}
        \begin{pmatrix}
            R^{\Sigma, x}_{i^+, j^+} & R^{\Sigma}_{i^+, j} \\
            R^{\Sigma}_{i, j^+} & R^{\Sigma, y}_{i^+, j^+}
        \end{pmatrix}
        & \text{in } T^-_{i^+, j^+}, \\
        \begin{pmatrix}
            R^{\Sigma, x}_{i^+, j^+} & R^{\Sigma}_{i^+, j + 1} \\
            R^{\Sigma}_{i + 1, j^+} & R^{\Sigma, y}_{i^+, j^+}
        \end{pmatrix}
        & \text{in } T^+_{i^+, j^+},
    \end{cases}
\end{align*}
and
\begin{align*}
    R^P_h &= \sum_{e \in \mathcal{E}^D_h} R^P_e = R^P_{i^+, j^+} \quad \text{in } T_{i^+, j^+}.
\end{align*}
It is evident that $\| \mathbf{R}^\Sigma_h \|_{L^2} = O(h^2) \| \mathbf{u} \|_{W^3_\infty}$ and $\| R^P_h \|_{L^2} = O(h^2) \| \mathbf{u} \|_{W^3_\infty}$. Then, we have the following error estimates between the auxiliary variables and the numerical solutions.
\begin{lemma} \label{lem: auxiliary error}
    Let $(\boldsymbol{\sigma}_h, \mathbf{u}_h, p_h)$ be the solution of \eqref{eq: SDG 1}--\eqref{eq: SDG 3} and $(\tilde{\boldsymbol{\sigma}}_h, \tilde{\mathbf{u}}_h, \tilde{p}_h)$ be defined in \eqref{eq: auxiliary variables 1}--\eqref{eq: auxiliary variables 3}. Then, the following estimates hold:
    \begin{align}
        \| \tilde{\boldsymbol{\sigma}}_h - \boldsymbol{\sigma}_h \|_{L^2} &\lesssim h^2 \| \mathbf{u} \|_{W^3_\infty}, \label{eq: auxiliary error 1} \\
        \| \tilde{\mathbf{u}}_h - \mathbf{u}_h \|_{H^1, h} &\lesssim h^2 \| \mathbf{u} \|_{W^3_\infty}, \label{eq: auxiliary error 2} \\
        \| \tilde{p}_h - p_h \|_{L^2} &\lesssim \nu h^2 \| \mathbf{u} \|_{W^3_\infty}. \label{eq: auxiliary error 3}
    \end{align}
\end{lemma}
\begin{proof}
    Subtracting \eqref{eq: auxiliary equation 1}-\eqref{eq: auxiliary equation 3} from \eqref{eq: SDG 1}-\eqref{eq: SDG 3} gives
    \begin{align*}
        (\tilde{\boldsymbol{\sigma}}_h - \boldsymbol{\sigma}_h, \boldsymbol{\tau}_h) - B^*_h (\tilde{\mathbf{u}}_h - \mathbf{u}_h, \boldsymbol{\tau}_h) &= (\mathbf{R}^\Sigma_h, \boldsymbol{\tau}_h), & &\forall \boldsymbol{\tau}_h \in \Sigma_h, \\
        \nu B_h (\tilde{\boldsymbol{\sigma}}_h - \boldsymbol{\sigma}_h, \mathbf{v}_h) - b^*_h (\tilde{p}_h - p_h, \mathbf{v}_h) &= 0, & &\forall \mathbf{v}_h \in U_{h, 0}, \\
        b_h (\tilde{\mathbf{u}}_h - \mathbf{u}_h, q_h) &= (R^P_h, q_h), & &\forall q_h \in P_{h, 0}.
    \end{align*}
    Taking $\boldsymbol{\tau}_h = \nu (\tilde{\boldsymbol{\sigma}}_h - \boldsymbol{\sigma}_h)$, $\mathbf{v}_h = \tilde{\mathbf{u}}_h - \mathbf{u}_h$, $q_h = \tilde{p}_h - p_h$ and summing the above equations yields
    \begin{align} \label{proof: auxiliary error 1}
        \nu \| \tilde{\boldsymbol{\sigma}}_h - \boldsymbol{\sigma}_h \|_{L^2}^2 = \nu (\mathbf{R}^\Sigma_h, \tilde{\boldsymbol{\sigma}}_h - \boldsymbol{\sigma}_h) + (R^P_h, \tilde{p}_h - p_h).
    \end{align}
    By Lemma \ref{lem: inf-sup condition 2},
    \begin{equation} \label{proof: auxiliary error 2}
        \begin{aligned}
            \| \tilde{p}_h - p_h \|_{L^2} &\lesssim \sup_{\mathbf{v}_h \in U_{h, 0}} \frac{b_h (\mathbf{v}_h, \tilde{p}_h - p_h)}{\| \mathbf{v}_h \|_{H^1, h}} = \sup_{\mathbf{v}_h \in U_{h, 0}} \frac{b^*_h (\tilde{p}_h - p_h, \mathbf{v}_h)}{\| \mathbf{v}_h \|_{H^1, h}} \\
            &= \sup_{\mathbf{v}_h \in U_{h, 0}} \frac{\nu B_h (\tilde{\boldsymbol{\sigma}}_h - \boldsymbol{\sigma}_h, \mathbf{v}_h)}{\| \mathbf{v}_h \|_{H^1, h}} \lesssim \nu \| \tilde{\boldsymbol{\sigma}}_h - \boldsymbol{\sigma}_h \|_{L^2}.
        \end{aligned}
    \end{equation}
    Substituting \eqref{proof: auxiliary error 2} into \eqref{proof: auxiliary error 1} gives
    \begin{align*}
        \| \tilde{\boldsymbol{\sigma}}_h - \boldsymbol{\sigma}_h \|_{L^2} &\lesssim \| \mathbf{R}^\Sigma_h \|_{L^2} + \| R^P_h \|_{L^2} \lesssim h^2 \| \mathbf{u} \|_{W^3_\infty}.
    \end{align*}
    Using \eqref{proof: auxiliary error 2} again derives \eqref{eq: auxiliary error 3}. By Lemma \ref{lem: inf-sup condition 1},
    \begin{align*}
        \| \tilde{\mathbf{u}}_h - \mathbf{u}_h \|_{H^1, h} &\lesssim \sup_{\boldsymbol{\tau}_h \in \Sigma_h} \frac{B_h (\boldsymbol{\tau}_h, \tilde{\boldsymbol{\sigma}}_h - \boldsymbol{\sigma}_h)}{\| \boldsymbol{\tau}_h \|_{L^2}} = \sup_{\boldsymbol{\tau}_h \in \Sigma_h} \frac{B^*_h (\tilde{\boldsymbol{\sigma}}_h - \boldsymbol{\sigma}_h, \boldsymbol{\tau}_h)}{\| \boldsymbol{\tau}_h \|_{L^2}} \\
        &= \sup_{\boldsymbol{\tau}_h \in \Sigma_h} \frac{(\tilde{\boldsymbol{\sigma}}_h - \boldsymbol{\sigma}_h, \boldsymbol{\tau}_h) - (\mathbf{R}^\Sigma_h, \boldsymbol{\tau}_h)}{\| \boldsymbol{\tau}_h \|_{L^2}} \lesssim \| \tilde{\boldsymbol{\sigma}}_h - \boldsymbol{\sigma}_h \|_{L^2} + \| \mathbf{R}^\Sigma_h \|_{L^2},
    \end{align*}
    which implies \eqref{eq: auxiliary error 2}.
\end{proof}

Now we are ready to present the main error estimates.
\begin{theorem}
    Let $(\boldsymbol{\sigma}, \mathbf{u}, p)$ be the solution of \eqref{eq: Stokes 2.1}--\eqref{eq: Stokes 2.3}, $(\boldsymbol{\sigma}_h, \mathbf{u}_h, p_h)$ be the solution of \eqref{eq: SDG 1}--\eqref{eq: SDG 3}, and $\breve{\boldsymbol{\sigma}}_h$ be defined in \eqref{eq: reconstructed diagonal velocity gradient}--\eqref{eq: reconstructed velocity gradient}. Then, the following error estimates hold:
    \begin{align*}
        \| \boldsymbol{\sigma} - \breve{\boldsymbol{\sigma}}_h \|_{l^2} &\lesssim h^2 \| \mathbf{u} \|_{W^3_\infty}, \\
        \| \mathbf{u} - \mathbf{u}_h \|_{l^2} &\lesssim h^2 \| \mathbf{u} \|_{W^3_\infty}, \\
        \| p - p_h \|_{l^2} &\lesssim \nu h^2 \| \mathbf{u} \|_{W^3_\infty} + h^2 \| \mathbf{p} \|_{W^2_\infty}.
    \end{align*}
\end{theorem}
\begin{proof}
    Combining Lemma \ref{lem: auxiliary error} and Lemma \ref{lem: norm relation} gives
    \begin{align*}
        | \tilde{\boldsymbol{\sigma}}_h - \boldsymbol{\sigma}_h |_{l^2, off} &\lesssim h^2 \| \mathbf{u} \|_{W^3_\infty}, \\
        | \tilde{\mathbf{u}}_h - \mathbf{u}_h |_{h^1, diag} &\lesssim h^2 \| \mathbf{u} \|_{W^3_\infty}, \\
        \| \tilde{\mathbf{u}}_h - \mathbf{u}_h \|_{l^2} &\lesssim h^2 \| \mathbf{u} \|_{W^3_\infty}, \\
        \| \tilde{p}_h - p_h \|_{l^2} &\lesssim \nu h^2 \| \mathbf{u} \|_{W^3_\infty}.
    \end{align*}
    By Lemma \ref{lem: projection error},
    \begin{align*}
        | \boldsymbol{\sigma} - \tilde{\boldsymbol{\sigma}}_h |_{l^2, off} &\lesssim h^2 \| \mathbf{u} \|_{W^3_\infty}, \\
        \| \mathbf{u} - \tilde{\mathbf{u}}_h \|_{l^2} &\lesssim h^2 \| \mathbf{u} \|_{W^2_\infty}, \\
        \| p - \tilde{p}_h \|_{l^2} &\lesssim h^2 \| \mathbf{p} \|_{W^2_\infty}.
    \end{align*}
    By the definition of $\breve{\boldsymbol{\sigma}}_h$, we decompose $\| \boldsymbol{\sigma} - \breve{\boldsymbol{\sigma}}_h \|_{l^2}$ as follows:
    \begin{align*}
        \| \boldsymbol{\sigma} - \breve{\boldsymbol{\sigma}}_h \|^2_{l^2} 
        &= | \boldsymbol{\sigma} - \boldsymbol{\sigma}_h |^2_{l^2, off} + \sum_{(i^+, j^+) \in \Lambda^D_h} h^x_{i^+} h^y_{j^+} \left[ (\sigma^{x, x} (x_{i^+}, y_{j^+}) - d_x u^x_{i^+, j^+})^2 \right. \\ &\quad \left. + (\sigma^{y, y} (x_{i^+}, y_{j^+}) - d_y u^y_{i^+, j^+})^2 \right] \\
        &\lesssim | \boldsymbol{\sigma} - \tilde{\boldsymbol{\sigma}}_h |^2_{l^2, off} + | \tilde{\mathbf{u}}_h - \mathbf{u}_h |^2_{h^1, diag} + \sum_{(i^+, j^+) \in \Lambda^D_h} h^x_{i^+} h^y_{j^+} \left[ (\sigma^{x, x} (x_{i^+}, y_{j^+}) - d_x \tilde{u}^x_{i^+, j^+})^2 \right. \\ &\quad \left. + (\sigma^{y, y} (x_{i^+}, y_{j^+}) - d_y \tilde{u}^y_{i^+, j^+})^2 \right]
    \end{align*}
    By the definition of $\tilde{\mathbf{u}}_h$, it is easy to see that
    \begin{align*}
         d_x \tilde{u}^x_{i^+, j^+} = \sigma^{x, x} (x_{i^+}, y_{j^+}) + O(h^2) \| \mathbf{u} \|_{W^3_\infty}, \quad d_y \tilde{u}^y_{i^+, j^+} = \sigma^{y, y} (x_{i^+}, y_{j^+}) + O(h^2) \| \mathbf{u} \|_{W^3_\infty}.
    \end{align*}
    Combining the above inequalities completes the proof.
\end{proof}
\begin{remark}
    The error estimates demonstrate that our method possesses second-order superconvergence for all variables, and that the velocity (gradient) error is independent of pressure and viscosity (i.e., pressure robustness).
\end{remark}

\section{Extension to Navier-Stokes Equations} \label{sec 6}

In this section, we extend the proposed SDG$_0$ method to the incompressible Navier-Stokes equations:
\begin{align*}
    \boldsymbol{\sigma} - \nabla \mathbf{u} &= 0, \\
    \frac{\partial \mathbf{u}}{\partial t} + \mathbf{u} \cdot \nabla \mathbf{u} - \nu \nabla \cdot \boldsymbol{\sigma} + \nabla p &= \mathbf{f}, \\
    \nabla \cdot \mathbf{u} &= 0.
\end{align*}

\subsection{Spatial Discretization}

We discretize the nonlinear convection term by the following form:
\begin{equation} \label{eq: discretized convection term}
\begin{aligned}
    c_h (\boldsymbol{\sigma}_h, \mathbf{u}_h, \mathbf{v}_h) &= - \sum_{e \in \mathcal{E}^D_h} \int_e n^x \{\!\!\{ u^x_h \}\!\!\} \llbracket u^x_h \rrbracket \{\!\!\{ v^x_h \}\!\!\} \, ds + (u^y_h \sigma^{x, y}_h, v^x_h) \\ &\quad - \sum_{e \in \mathcal{E}^D_h} \int_e n^y \{\!\!\{ u^y_h \}\!\!\} \llbracket u^y_h \rrbracket \{\!\!\{ v^y_h \}\!\!\} \, ds + (u^x_h \sigma^{y, x}_h, v^y_h).
\end{aligned}
\end{equation}
The consistency error of the discrete convection term is analyzed in the following lemma.
\begin{lemma} \label{lem: truncation error 4}
    Let $\tilde{\boldsymbol{\sigma}}_h$ and $\tilde{\mathbf{u}}_h$ be defined in \eqref{eq: auxiliary variables 1}--\eqref{eq: auxiliary variables 2}. For basis function $\boldsymbol{\phi}^U_e$, it holds that
    \begin{align}
        \begin{aligned} \label{eq: truncation error 4.1}
            &(\mathbf{u} \cdot \nabla \mathbf{u}, \boldsymbol{\phi}^U_{e_{i, j^+}}) - c_h (\tilde{\boldsymbol{\sigma}}_h, \tilde{\mathbf{u}}_h, \boldsymbol{\phi}^U_{e_{i, j^+}}) \\
            &= \int_{e_{i^+, j^+}} \mathbf{R}^{U, x}_{e_{i^+, j^+}} \cdot \llbracket \boldsymbol{\phi}^U_{e_{i, j^+}} \rrbracket \, ds + \int_{e_{i^-, j^+}} \mathbf{R}^{U, x}_{e_{i^-, j^+}} \cdot \llbracket \boldsymbol{\phi}^U_{e_{i, j^+}} \rrbracket \, ds \\ &\quad + (\mathbf{R}^U_{e_{i, j^+}}, \boldsymbol{\phi}^U_{e_{i, j^+}}), \quad \forall e_{i, j^+} \in \mathring{\mathcal{E}}^V_h,
        \end{aligned} \\
        \begin{aligned} \label{eq: truncation error 4.2}
            &(\mathbf{u} \cdot \nabla \mathbf{u}, \boldsymbol{\phi}^U_{e_{i^+, j}}) - c_h (\tilde{\boldsymbol{\sigma}}_h, \tilde{\mathbf{u}}_h, \boldsymbol{\phi}^U_{e_{i^+, j}}) \\
            &= \int_{e_{i^+, j^+}} \mathbf{R}^{U, y}_{e_{i^+, j^+}} \cdot \llbracket \boldsymbol{\phi}^U_{e_{i^+, j}} \rrbracket \, ds + \int_{e_{i^+, j^-}} \mathbf{R}^{U, y}_{e_{i^+, j^-}} \cdot \llbracket \boldsymbol{\phi}^U_{e_{i^+, j}} \rrbracket \, ds \\ &\quad + (\mathbf{R}^U_{e_{i^+, j}}, \boldsymbol{\phi}^U_{e_{i^+, j}}), \quad \forall e_{i^+, j} \in \mathring{\mathcal{E}}^H_h,
        \end{aligned}
    \end{align}
    where
    \begin{align*}
        \mathbf{R}^{U, x}_{e_{i^+, j^+}} &= (R^{U, x}_{i^+, j^+} \chi_{T_{i^+, j^+}}, 0)^{\mathsf{T}}, &
        \mathbf{R}^U_{e_{i, j^+}} &= (R^U_{i, j^+} \chi_{T_{i, j^+}}, 0)^{\mathsf{T}}, \\
        \mathbf{R}^{U, y}_{e_{i^+, j^+}} &= (0, R^{U, y}_{i^+, j^+}  \chi_{T_{i^+, j^+}})^{\mathsf{T}}, &
        \mathbf{R}^U_{e_{i^+, j}} &= (0, R^U_{i^+, j} \chi_{T_{i^+, j}})^{\mathsf{T}},
    \end{align*}
    with $R^{U, x}_{i^+, j^+} = O(h^2)$, $R^{U, y}_{i^+, j^+} = O(h^2)$, $R^U_{i, j^+} = O(h^2)$ and $R^U_{i^+, j} = O(h^2)$.
\end{lemma}
\begin{proof}
    We only show the proof of \eqref{eq: truncation error 4.1}, and \eqref{eq: truncation error 4.2} can be proven similarly. Write the left side of \eqref{eq: truncation error 4.1} in detail:
    \begin{equation} \label{prf: truncation error 4.1}
    \begin{aligned}
        &(\mathbf{u} \cdot \nabla \mathbf{u}, \boldsymbol{\phi}^U_{e_{i, j^+}}) - c_h (\tilde{\boldsymbol{\sigma}}_h, \tilde{\mathbf{u}}_h, \boldsymbol{\phi}^U_{e_{i, j^+}}) \\
        &= \int_{T_{i, j^+}} u^x \frac{\partial u^x}{\partial x} \, dx \, dy + \int_{T_{i, j^+}} u^y \frac{\partial u^x}{\partial y} \, dx \, dy \\
        &\quad - \frac{h^x_{i^+} h^y_{j^+}}{2} \cdot \frac{\tilde{u}^x_{i, j^+} + \tilde{u}^x_{i + 1, j^+}}{2} \cdot \frac{\tilde{u}^x_{i + 1, j^+} - \tilde{u}^x_{i, j^+}}{h^x_{i^+}} \\
        &\quad - \frac{h^x_{i^-} h^y_{j^+}}{2} \cdot \frac{\tilde{u}^x_{i - 1, j^+} + \tilde{u}^x_{i, j^+}}{2} \cdot \frac{\tilde{u}^x_{i, j^+} - \tilde{u}^x_{i - 1, j^+}}{h^x_{i^-}} \\
        &\quad - \frac{1}{2} h^x_{i^+} h^y_{j^+} \tilde{u}^y_{i^+, j} \tilde{\sigma}^{x, y}_{i^+, j} - \frac{1}{2} h^x_{i^-} h^y_{j^+} \tilde{u}^y_{i^-, j + 1} \tilde{\sigma}^{x, y}_{i^-, j + 1}.
    \end{aligned}
    \end{equation}
    For simplicity, we denote
    \begin{align*}
        g^{x, x} = u^x \frac{\partial u^x}{\partial x}, \quad g^{x, y} = u^y \frac{\partial u^x}{\partial y}.
    \end{align*}
    For the first term on the right side of \eqref{prf: truncation error 4.1},
    \begin{align*}
        &\int_{T_{i, j^+}} u^x \frac{\partial u^x}{\partial x} \, dx \, dy 
        = \int_{T^-_{i^+, j^+}} g^{x, x} \, dx \, dy + \int_{T^+_{i^-, j^+}} g^{x, x} \, dx \, dy \\
        &= \frac{1}{2} h^x_{i^+} h^y_{j^+} g^{x, x} (x_{i^+}, y_{j^+}) + \frac{1}{2} h^x_{i^-} h^y_{j^+} g^{x, x} (x_{i^-}, y_{j^+}) \\
        &\quad - \frac{(h^x_{i^+})^2 h^y_{j^+}}{12} \cdot \frac{\partial g^{x, x}}{\partial x} (x_{i^+}, y_{j^+}) - \frac{h^x_{i^+} (h^y_{j^+})^2}{12} \cdot \frac{\partial g^{x, x}}{\partial y} (x_{i^+}, y_{j^+}) \\
        &\quad + \frac{(h^x_{i^-})^2 h^y_{j^+}}{12} \cdot \frac{\partial g^{x, x}}{\partial x} (x_{i^-}, y_{j^+}) + \frac{h^x_{i^-} (h^y_{j^+})^2}{12} \cdot \frac{\partial g^{x, x} }{\partial y}(x_{i^-}, y_{j^+}) + O(h^4).
    \end{align*}
    For the second term on the right side of \eqref{prf: truncation error 4.1},
    \begin{align*}
        &\int_{T_{i, j^+}} u^y \frac{\partial u^x}{\partial y} \, dx \, dy 
        = \int_{T^-_{i^+, j^+}} g^{x, y} \, dx \, dy + \int_{T^+_{i^-, j^+}} g^{x, y} \, dx \, dy \\
        &= \frac{1}{2} h^x_{i^+} h^y_{j^+} g^{x, y} (x_{i^+}, y_j) + \frac{1}{2} h^x_{i^-} h^y_{j^+} g^{x, y} (x_{i^-}, y_{j + 1}) \\
        &\quad - \frac{(h^x_{i^+})^2 h^y_{j^+}}{12} \cdot \frac{\partial g^{x, y}}{\partial x} (x_{i^+}, y_{j^+}) + \frac{h^x_{i^+} (h^y_{j^+})^2}{6} \cdot \frac{\partial g^{x, y}}{\partial y} (x_{i^+}, y_{j^+}) \\
        &\quad + \frac{(h^x_{i^-})^2 h^y_{j^+}}{12} \cdot \frac{\partial g^{x, y}}{\partial x} (x_{i^-}, y_{j^+}) - \frac{h^x_{i^-} (h^y_{j^+})^2}{6} \cdot \frac{\partial g^{x, y}}{\partial y} (x_{i^-}, y_{j^+}) + O(h^4).
    \end{align*}
    The right side of \eqref{eq: truncation error 4.1} equals $l_{i^+, j^+} R^{U, x}_{i^+, j^+} - l_{i^-, j^+} R^{U, x}_{i^-, j^+} + h^x_i h^y_{j^+} R^U_{i, j^+}$. Comparing the both sides of \eqref{eq: truncation error 4.1} yields
    \begin{align*}
        R^{U, x}_{i^+, j^+} &= \frac{1}{l_{i^+, j^+}} \left[ - \frac{(h^x_{i^+})^2 h^y_{j^+}}{12} \cdot \frac{\partial g^{x, x}}{\partial x} (x_{i^+}, y_{j^+}) - \frac{h^x_{i^+} (h^y_{j^+})^2}{12} \cdot \frac{\partial g^{x, x}}{\partial y} (x_{i^+}, y_{j^+}) \right. \\
        &\quad \left. - \frac{(h^x_{i^+})^2 h^y_{j^+}}{12} \cdot \frac{\partial g^{x, y}}{\partial x} (x_{i^+}, y_{j^+}) + \frac{h^x_{i^+} (h^y_{j^+})^2}{6} \cdot \frac{\partial g^{x, y}}{\partial y} (x_{i^+}, y_{j^+}) \right] \\
        &= O (h^2),
    \end{align*}
    and
    \begin{align*}
        R^U_{i, j^+} &= \frac{1}{h^x_i h^y_{j^+}} \left[ \frac{1}{2} h^x_{i^+} h^y_{j^+} g^{x, x} (x_{i^+}, y_{j^+}) + \frac{1}{2} h^x_{i^-} h^y_{j^+} g^{x, x} (x_{i^-}, y_{j^+}) \right. \\
        &\quad - \frac{h^x_{i^+} h^y_{j^+}}{2} \cdot \frac{\tilde{u}^x_{i, j^+} + \tilde{u}^x_{i + 1, j^+}}{2} \cdot \frac{\tilde{u}^x_{i + 1, j^+} - \tilde{u}^x_{i, j^+}}{h^x_{i^+}} \\
        &\quad - \frac{h^x_{i^-} h^y_{j^+}}{2} \cdot \frac{\tilde{u}^x_{i - 1, j^+} + \tilde{u}^x_{i, j^+}}{2} \cdot \frac{\tilde{u}^x_{i, j^+} - \tilde{u}^x_{i - 1, j^+}}{h^x_{i^-}} \\
        &\quad + \frac{1}{2} h^x_{i^+} h^y_{j^+} g^{x, y} (x_{i^+}, y_j) + \frac{1}{2} h^x_{i^-} h^y_{j^+} g^{x, y} (x_{i^-}, y_{j + 1}) \\
        &\quad \left. - \frac{1}{2} h^x_{i^+} h^y_{j^+} \tilde{u}^y_{i^+, j} \tilde{\sigma}^{x, y}_{i^+, j} - \frac{1}{2} h^x_{i^-} h^y_{j^+} \tilde{u}^y_{i^-, j + 1} \tilde{\sigma}^{x, y}_{i^-, j + 1} + O(h^4) \right] \\
        &= O (h^2),
    \end{align*}
    which completes the proof.
\end{proof}
Consequently, $c_h (\boldsymbol{\sigma}_h, \mathbf{u}_h, \mathbf{v}_h)$ is a second-order approximation of the convection term, as stated in the following theorem.
\begin{theorem} \label{thm: truncation error 4}
    Let $\tilde{\boldsymbol{\sigma}}_h$ and $\tilde{\mathbf{u}}_h$ be defined in \eqref{eq: auxiliary variables 1}--\eqref{eq: auxiliary variables 2}. It holds that
    \begin{align*}
        (\mathbf{u} \cdot \nabla \mathbf{u}, \mathbf{v}_h) - c_h (\tilde{\boldsymbol{\sigma}}_h, \tilde{\mathbf{u}}_h, \mathbf{v}_h) \lesssim h^2 \| \mathbf{v}_h \|_{H^1, h}, \quad \forall \mathbf{v}_h \in U_{h, 0}.
    \end{align*}
\end{theorem}
\begin{proof}
    Define
    \begin{align*}
        \mathbf{R}^{U, D}_h &= \sum_{e \in \mathcal{E}^D_h} (\mathbf{R}^{U, x}_e + \mathbf{R}^{U, y}_e) = (\mathbf{R}^{U, x}_{i^+, j^+}, \mathbf{R}^{U, y}_{i^+, j^+})^{\mathsf{T}} \quad \text{in } T_{i^+, j^+}, \\
        \mathbf{R}^{U, O}_h &= \sum_{e \in \mathcal{E}^H_h \cup \mathcal{E}^V_h} \mathbf{R}^U_e =
        \begin{cases}
            (R^U_{i, j^+}, R^U_{i^+, j})^{\mathsf{T}} & \text{in } T^-_{i^+, j^+}, \\
            (R^U_{i + 1, j^+}, R^U_{i^+, j + 1})^{\mathsf{T}} & \text{in } T^+_{i^+, j^+},
        \end{cases}
    \end{align*}
    where we let $R^U_{0, j^+}$, $R^U_{n_x, j^+}$, $R^U_{i^+, 0}$ and $R^U_{i^+, n_y}$ all be zero. It is clear that $\| \mathbf{R}^{U, D}_h \|_{L^2} = O(h^2)$ and $\| \mathbf{R}^{U, O}_h \|_{L^2} = O(h^2)$. By Lemma \ref{lem: truncation error 4}, for any $\mathbf{v}_h \in U_{h, 0}$,
    \begin{align*}
        &(\mathbf{u} \cdot \nabla \mathbf{u}, \mathbf{v}_h) - c_h (\tilde{\boldsymbol{\sigma}}_h, \tilde{\mathbf{u}}_h, \mathbf{v}_h)
        = \sum_{e \in \mathcal{E}^D_h} \int_e \mathbf{R}^{U, D}_h \cdot \llbracket \mathbf{v}_h \rrbracket \, ds + (\mathbf{R}^{U, O}_h, \mathbf{v}_h) \\
        & \leq \| \mathbf{R}^{U, D}_h \|_{L^2} \| \mathbf{v}_h \|_{H^1, h} + \| \mathbf{R}^{U, O}_h \|_{L^2} \| \mathbf{v}_h \|_{L^2}
        \lesssim h^2 \| \mathbf{v}_h \|_{H^1, h}.
    \end{align*}
\end{proof}
\begin{remark}
    The discrete convection term $c_h(\boldsymbol{\sigma}_h, \mathbf{u}_h, \mathbf{v}_h)$ is constructed by a hybrid strategy that couples a mixed finite element treatment of the velocity gradient with an edge-based DG-style discretization of the advective flux. This operator is second-order consistent and, since it is independent of the diagonal entries of $\boldsymbol{\sigma}_h$, remains compatible with the local static condensation in Section~\ref{sec 4}. However, $c_h$ is not skew-symmetric and therefore does not, by itself, provide discrete energy stability. To address this issue, we adopt the SAV approach below, which guarantees unconditional energy stability. An upwind discretization could alternatively be employed to enforce stability at the spatial level, as in the SDG literature \cite{zhao2019priori, kim2021pressurerobust}, but this typically reduces the accuracy to first order. The present choice therefore achieves a balance between stability and second-order accuracy.
\end{remark}
\begin{remark}
    The term $(\mathbf{u}_h, \mathbf{v}_h)$ generated from the time discretization is also second-order consistent in the sense that
    \begin{align*}
        (\mathbf{u}, \mathbf{v}_h) - (\tilde{\mathbf{u}}_h, \mathbf{v}_h) \lesssim h^2 \| \mathbf{v}_h \|_{H^1, h}, \quad \forall \mathbf{v}_h \in U_{h, 0}.
    \end{align*}
    The proof is similar to that of Lemma \ref{lem: truncation error 4} and Theorem \ref{thm: truncation error 4} and is omitted here.
\end{remark}

\subsection{Temporal Discretization}

Temporal Discretization is performed using the scalar auxiliary variable approach and Crank-Nicolson scheme \cite{li2020error}.
Introduce the scalar auxiliary variable $s = \sqrt{E(\mathbf{u}) + \delta}$, where $E(\mathbf{u}) = \frac{1}{2} \| \mathbf{u} \|^2_{L^2 (\Omega)}$ is the kinetic energy and $\delta$ is a small positive constant. Then the Navier-Stokes equations can be reformulated as
\begin{align*}
    \boldsymbol{\sigma} - \nabla \mathbf{u} &= 0, \\
    \frac{\partial \mathbf{u}}{\partial t} + \frac{s}{\sqrt{E(\mathbf{u}) + \delta}} \mathbf{u} \cdot \nabla \mathbf{u} - \nu \nabla \cdot \boldsymbol{\sigma} + \nabla p &= \mathbf{f}, \\
    \nabla \cdot \mathbf{u} &= 0, \\
    \frac{\partial s}{\partial t} - \frac{1}{2 s} \int_{\Omega} \mathbf{u} \cdot \frac{\partial \mathbf{u}}{\partial t} \, dx \, dy - \frac{1}{2 \sqrt{E(\mathbf{u}) + \delta}} \int_{\Omega} (\mathbf{u} \cdot \nabla \mathbf{u}) \cdot \mathbf{u} \, dx \, dy &= 0.
\end{align*}
Take a uniform partition of the time interval $[0, T]$ with a time step size $\Delta t = T / N$ and denote $t_n = n \Delta t$ for $n = 0, 1, \ldots, N$. For any time-dependent quantity $\psi$, denote $\psi^n = \psi(\cdot, t_n)$. Define the midpoint interpolation $\psi^{n + 1/2} = (\psi^{n + 1} + \psi^n) / 2$ and the extrapolation $\bar{\psi}^{n + 1/2} = (3 \psi^n - \psi^{n - 1}) / 2$.

\textbf{CN-SAV-SDG$_0$ Scheme}: for $n = 1, 2, \ldots, N - 1$, find $\boldsymbol{\sigma}^{n + 1}_h \in \Sigma_h$, $\mathbf{u}^{n + 1}_h \in U_{h, 0}$ and $p^{n + 1}_h \in P_{h, 0}$ such that
\begin{empheq}[left=\empheqlbrace]{align}
    &(\boldsymbol{\sigma}^{n + 1}_h, \boldsymbol{\tau}_h) - B^*_h (\mathbf{u}^{n + 1}_h, \boldsymbol{\tau}_h) = 0, \quad \forall \boldsymbol{\tau}_h \in \Sigma_h, \label{eq: CN-SAV-SDG 1} \\
    &\begin{multlined} \label{eq: CN-SAV-SDG 2}
        (\frac{1}{\Delta t} (\mathbf{u}^{n + 1}_h - \mathbf{u}^n_h), \mathbf{v}_h) + \frac{s^{n + 1/2}}{\sqrt{E(\bar{\mathbf{u}}^{n + 1/2}_h) + \delta}} c_h (\bar{\boldsymbol{\sigma}}^{n + 1/2}_h, \bar{\mathbf{u}}^{n + 1/2}_h, \mathbf{v}_h) \\ + \nu B_h (\boldsymbol{\sigma}^{n + 1/2}_h, \mathbf{v}_h) - b^*_h (p^{n + 1/2}_h, \mathbf{v}_h) = (\mathbf{f}^{n + 1/2}, \mathbf{v}_h), \quad \forall \mathbf{v}_h \in U_{h, 0},
    \end{multlined} \\
    &b_h (\mathbf{u}^{n + 1}_h, q_h) = 0, \quad \forall q_h \in P_{h, 0}, \label{eq: CN-SAV-SDG 3} \\
    &\begin{multlined} \label{eq: CN-SAV-SDG 4}
        \frac{1}{\Delta t} (s^{n + 1} - s^n) - \frac{1}{2 s^{n + 1/2}} (\mathbf{u}^{n + 1/2}_h, \frac{1}{\Delta t} (\mathbf{u}^{n + 1}_h - \mathbf{u}^n_h)) \\ - \frac{1}{2 \sqrt{E(\bar{\mathbf{u}}^{n + 1/2}_h) + \delta}} c_h (\bar{\boldsymbol{\sigma}}^{n + 1/2}_h, \bar{\mathbf{u}}^{n + 1/2}_h, \mathbf{u}^{n + 1/2}_h) = 0.
    \end{multlined}
\end{empheq}
The first time step is not covered by the above recurrence due to the extrapolation, and it can be initialized by any consistent first-order method.

The CN-SAV-SDG$_0$ scheme can be efficiently implemented by splitting it into two linear Stokes-type problems and a scalar quadratic equation. Introduce intermediate variable:
\begin{align*}
    w^{n + 1} = \frac{s^{n + 1/2}}{\sqrt{E(\bar{\mathbf{u}}^{n + 1/2}_h) + \delta}}.
\end{align*}
Then the solution can be expressed as
\begin{align*}
    \boldsymbol{\sigma}^{n + 1}_h = \hat{\boldsymbol{\sigma}}^{n + 1}_h + w^{n + 1} \check{\boldsymbol{\sigma}}^{n + 1}_h, \quad \mathbf{u}^{n + 1}_h = \hat{\mathbf{u}}^{n + 1}_h + w^{n + 1} \check{\mathbf{u}}^{n + 1}_h, \quad p^{n + 1}_h = \hat{p}^{n + 1}_h + w^{n + 1} \check{p}^{n + 1}_h,
\end{align*}
where $(\hat{\boldsymbol{\sigma}}^{n + 1}_h, \hat{\mathbf{u}}^{n + 1}_h, \hat{p}^{n + 1}_h)$, $(\check{\boldsymbol{\sigma}}^{n + 1}_h, \check{\mathbf{u}}^{n + 1}_h, \check{p}^{n + 1}_h)$, and $w^{n + 1}$ are obtained by solving the following three sub-problems. \\
\textbf{Sub-problem 1}: Find $\hat{\boldsymbol{\sigma}}^{n + 1}_h \in \Sigma_h$, $\hat{\mathbf{u}}^{n + 1}_h \in U_{h, 0}$ and $\hat{p}^{n + 1}_h \in P_{h, 0}$ such that
\begin{empheq}[left=\empheqlbrace]{align*}
    &(\hat{\boldsymbol{\sigma}}^{n + 1}_h, \boldsymbol{\tau}_h) - B^*_h (\hat{\mathbf{u}}^{n + 1}_h, \boldsymbol{\tau}_h) = 0, \quad \forall \boldsymbol{\tau}_h \in \Sigma_h, \\
    &\begin{multlined}
        \frac{1}{\Delta t} (\hat{\mathbf{u}}^{n + 1}_h, \mathbf{v}_h) + \frac{\nu}{2} B_h (\hat{\boldsymbol{\sigma}}^{n + 1}_h, \mathbf{v}_h) - \frac{1}{2} b^*_h (\hat{p}^{n + 1}_h, \mathbf{v}_h) \\ = (\mathbf{f}^{n + 1/2}, \mathbf{v}_h) + \frac{1}{\Delta t} (\mathbf{u}^{n}_h, \mathbf{v}_h) - \frac{\nu}{2} B_h (\boldsymbol{\sigma}^{n}_h, \mathbf{v}_h) + \frac{1}{2} b^*_h (p^{n}_h, \mathbf{v}_h), \quad \forall \mathbf{v}_h \in U_{h, 0},
    \end{multlined} \\
    &b_h (\hat{\mathbf{u}}^{n + 1}_h, q_h) = 0, \quad \forall q_h \in P_{h, 0}.
\end{empheq}
\textbf{Sub-problem 2}: Find $\check{\boldsymbol{\sigma}}^{n + 1}_h \in \Sigma_h$, $\check{\mathbf{u}}^{n + 1}_h \in U_{h, 0}$ and $\check{p}^{n + 1}_h \in P_{h, 0}$ such that
\begin{empheq}[left=\empheqlbrace]{align*}
    &(\check{\boldsymbol{\sigma}}^{n + 1}_h, \boldsymbol{\tau}_h) - B^*_h (\check{\mathbf{u}}^{n + 1}_h, \boldsymbol{\tau}_h) = 0, \quad \forall \boldsymbol{\tau}_h \in \Sigma_h, \\
    &\begin{multlined}
        \frac{1}{\Delta t} (\check{\mathbf{u}}^{n + 1}_h, \mathbf{v}_h) + \frac{\nu}{2} B_h (\check{\boldsymbol{\sigma}}^{n + 1}_h, \mathbf{v}_h) - \frac{1}{2} b^*_h (\check{p}^{n + 1}_h, \mathbf{v}_h) = -c_h (\bar{\boldsymbol{\sigma}}^{n + 1/2}_h, \bar{\mathbf{u}}^{n + 1/2}_h, \mathbf{v}_h), \quad \forall \mathbf{v}_h \in U_{h, 0},
    \end{multlined} \\
    &b_h (\check{\mathbf{u}}^{n + 1}_h, q_h) = 0, \quad \forall q_h \in P_{h, 0}.
\end{empheq}
\textbf{Sub-problem 3}:
\begin{align*}
    \operatorname{argmin} \quad & |w^{n + 1} - 1|, \\
    \text{s.t.} \quad & \alpha^{n + 1}_0 + \alpha^{n + 1}_1 w^{n + 1} + \alpha^{n + 1}_2 (w^{n + 1})^2 = 0,
\end{align*}
where
\begin{align*}
    \alpha^{n + 1}_0 &= \frac{\nu}{4} B_h (\boldsymbol{\sigma}^n_h + \hat{\boldsymbol{\sigma}}^{n + 1}_h, \mathbf{u}^n_h + \hat{\mathbf{u}}^{n + 1}_h) - \frac{1}{2} (\mathbf{f}^{n + 1/2}, \mathbf{u}^n_h + \hat{\mathbf{u}}^{n + 1}_h), \\
    \alpha^{n + 1}_1 &= -\frac{4 s^n}{\Delta t} \sqrt{E(\bar{\mathbf{u}}^{n + 1/2}_h) + \delta} + \frac{\nu}{4} B_h (\boldsymbol{\sigma}^n + \hat{\boldsymbol{\sigma}}^{n + 1}_h, \check{\mathbf{u}}^{n + 1}_h) \\ &\quad + \frac{\nu}{4} B_h (\check{\boldsymbol{\sigma}}^{n + 1}_h, \mathbf{u}^n_h + \hat{\mathbf{u}}^{n + 1}_h) - \frac{1}{2} (\mathbf{f}^{n + 1/2}, \check{\mathbf{u}}^{n + 1}_h), \\
    \alpha^{n + 1}_2 &= \frac{4}{\Delta t} (E(\bar{\mathbf{u}}^{n + 1/2}_h) + \delta) + \frac{\nu}{4} B_h (\check{\boldsymbol{\sigma}}^{n + 1}_h, \check{\mathbf{u}}^{n + 1}_h).
\end{align*}
The proof of the equivalence between the CN-SAV-SDG$_0$ scheme and the above three sub-problems follows the same argument in \cite{li2020error}, and is thus omitted here.

The local static condensation presented in Section \ref{sec 4} can also be applied here. Let $\mathsf{a}$ be the matrix corresponding to the bilinear form $(\mathbf{u}_h, \mathbf{v}_h)$, and $\mathsf{c}^{n + 1/2}$ be the vector corresponding to the linear functional $c_h (\bar{\boldsymbol{\sigma}}^{n + 1/2}_h, \bar{\mathbf{u}}^{n + 1/2}_h, \mathbf{v}_h)$. Then the algebraic systems of the sub-problems 1 and 2 can be condensed to
\begin{align*}
    &\begin{pmatrix}
        \tilde{\mathsf{A}} & - \tilde{\mathsf{B}}^{\mathsf{T}} & 0 \\
        \frac{\nu}{2} \tilde{\mathsf{B}} & \frac{1}{\Delta t} \mathsf{a} + \frac{\nu}{2} \tilde{\mathsf{a}} & - \frac{1}{2} \mathsf{b}^{\mathsf{T}} \\
        0 & \mathsf{b} & 0
    \end{pmatrix}
    \begin{pmatrix}
        \hat{\mathsf{\sigma}}^{n + 1}_O \\
        \hat{\mathsf{u}}^{n + 1} \\
        \hat{\mathsf{p}}^{n + 1}
    \end{pmatrix} \\
    &=
    \begin{pmatrix}
        0 \\
        \mathsf{F}^{n + 1/2} - \frac{\nu}{2} \tilde{\mathsf{B}} \mathsf{\sigma}^n_O + (\frac{1}{\Delta t} \mathsf{a} - \frac{\nu}{2} \tilde{\mathsf{a}}) \mathsf{u}^n + \frac{1}{2} \mathsf{b}^{\mathsf{T}} \mathsf{p}^n \\
        0
    \end{pmatrix},
\end{align*}
and
\begin{align*}
    \begin{pmatrix}
        \tilde{\mathsf{A}} & - \tilde{\mathsf{B}}^{\mathsf{T}} & 0 \\
        \frac{\nu}{2} \tilde{\mathsf{B}} & \frac{1}{\Delta t} \mathsf{a} + \frac{\nu}{2} \tilde{\mathsf{a}} & - \frac{1}{2} \mathsf{b}^{\mathsf{T}} \\
        0 & \mathsf{b} & 0
    \end{pmatrix}
    \begin{pmatrix}
        \check{\mathsf{\sigma}}^{n + 1}_O \\
        \check{\mathsf{u}}^{n + 1} \\
        \check{\mathsf{p}}^{n + 1}
    \end{pmatrix}
    =
    \begin{pmatrix}
        0 \\
        - \mathsf{c}^{n + 1/2} \\
        0
    \end{pmatrix}.
\end{align*}
For sub-problem 3, we evaluate the $B_h$-dependent terms using the static condensation formula:
\begin{align*}
    B_h (\boldsymbol{\sigma}_h, \mathbf{u}_h) = \mathsf{u}^{\mathsf{T}} \mathsf{B} \mathsf{\sigma} = \mathsf{u}^{\mathsf{T}} \tilde{\mathsf{B}} \mathsf{\sigma}_O + \mathsf{u}^{\mathsf{T}} \tilde{\mathsf{a}} \mathsf{u}.
\end{align*}

\begin{remark}
    A rigorous analysis of the CN-SAV framework was presented in \cite{li2020error}, establishing unconditional energy stability and second-order accuracy in time. Combined with our analyses of the Stokes discretization and the discrete convection term, these results extend directly to the CN-SAV-SDG$_0$ scheme. For brevity, the proofs are omitted and the performance of the scheme is validated numerically in Section~\ref{sec 7}.
\end{remark}

\section{Numerical Experiments} \label{sec 7}

In this section, we present several numerical experiments to validate the accuracy and robustness of the SDG$_0$ and CN-SAV-SDG$_0$ schemes for Stokes and Navier-Stokes equations, respectively. We also examine the performance of mass lumping as mentioned in Remark \ref{rmk: mass lumping}. Unless otherwise stated, the computational domain is $\Omega = [0, 1] \times [0, 1]$, and all tests are conducted on non-uniform grids generated by perturbing uniform grids with zero-mean random noise of amplitude $10\%$ of the mesh size.

\subsection{Accuracy Test} \label{sec: 7.1}

We first assess the accuracy of the SDG$_0$ scheme for Stokes equations. Set $\nu = 1$ and adopt the manufactured solution:
\begin{align*}
    u^x &= \pi x^2 (1 - x)^2 \sin(2 \pi y), \\
    u^y &= - 2x (1 - x) (1 - 2x) \sin^2(\pi y), \\
    p &= \sin(x) \cos(y) + (\cos(1) - 1) \sin(1).
\end{align*}
Tables \ref{tab: 7.1} and \ref{tab: 7.2} report the errors and convergence rates for the original SDG$_0$ discretization and its mass-lumped variant, respectively. Both yield comparable error magnitudes and exhibit second-order convergence for all variables.

We then evaluate the CN-SAV-SDG$_0$ scheme for Navier-Stokes equations. The time-dependent solution is obtained by multiplying the fields above by $\exp(-t)$. We set $T = 0.25$ and $\Delta t = h$. Tables \ref{tab: 7.3} and \ref{tab: 7.4} present the corresponding errors and convergence rates for the original and mass-lumped variants, again confirming second-order accuracy for all variables.

\begin{table}[htbp]
    \centering
    \caption{Errors and convergence rates of SDG$_0$ for Stokes equations.}
    \begin{tabular}{ccccccc}
        \toprule
        $n_x \times n_y$ & $\| \boldsymbol{\sigma} - \breve{\boldsymbol{\sigma}}_h \|_{l^2}$ & Rate & $\| \mathbf{u} - \mathbf{u}_h \|_{l^2}$ & Rate & $\| p - p_h \|_{l^2}$ & Rate \\ 
        \midrule
        $8 \times 8$ & 1.65e-01 & -- & 2.54e-02 & -- & 9.39e-02 & -- \\ 
        $16 \times 16$ & 4.52e-02 & 1.87 & 6.77e-03 & 1.91 & 3.05e-02 & 1.62 \\ 
        $32 \times 32$ & 1.19e-02 & 1.93 & 1.74e-03 & 1.96 & 8.49e-03 & 1.85 \\ 
        $64 \times 64$ & 3.01e-03 & 1.98 & 4.37e-04 & 2.00 & 2.16e-03 & 1.98 \\ 
        $128 \times 128$ & 7.70e-04 & 1.97 & 1.10e-04 & 2.00 & 5.45e-04 & 1.99 \\ 
        \bottomrule
    \end{tabular}
    \label{tab: 7.1}
\end{table}
\begin{table}[htbp]
    \centering
    \caption{Errors and convergence rates of mass-lumped SDG$_0$ for Stokes equations.}
    \begin{tabular}{ccccccc}
        \toprule
        $n_x \times n_y$ & $\| \boldsymbol{\sigma} - \breve{\boldsymbol{\sigma}}_h \|_{l^2}$ & Rate & $\| \mathbf{u} - \mathbf{u}_h \|_{l^2}$ & Rate & $\| p - p_h \|_{l^2}$ & Rate \\ 
        \midrule
        $8 \times 8$ & 8.95e-02 & -- & 9.24e-03 & -- & 6.39e-02 & -- \\ 
        $16 \times 16$ & 2.31e-02 & 1.95 & 2.33e-03 & 1.99 & 1.88e-02 & 1.77 \\ 
        $32 \times 32$ & 5.85e-03 & 1.98 & 5.96e-04 & 1.97 & 4.87e-03 & 1.95 \\ 
        $64 \times 64$ & 1.46e-03 & 2.00 & 1.50e-04 & 1.99 & 1.23e-03 & 1.99 \\ 
        $128 \times 128$ & 3.66e-04 & 2.00 & 3.74e-05 & 2.00 & 3.07e-04 & 2.00 \\ 
        \bottomrule
    \end{tabular}
    \label{tab: 7.2}
\end{table}
\begin{table}[htbp]
    \centering
    \caption{Errors and convergence rates of CN-SAV-SDG$_0$ for Navier-Stokes equations.}
    \begin{tabular}{ccccccc}
        \toprule
        $n_x \times n_y$ & $\| \boldsymbol{\sigma} - \breve{\boldsymbol{\sigma}}_h \|_{l^2}$ & Rate & $\| \mathbf{u} - \mathbf{u}_h \|_{l^2}$ & Rate & $\| p - p_h \|_{l^2}$ & Rate \\ 
        \midrule
        $8 \times 8$ & 1.31e-01 & -- & 2.31e-02 & -- & 6.17e-02 & -- \\ 
        $16 \times 16$ & 3.56e-02 & 1.88 & 5.47e-03 & 2.08 & 1.99e-02 & 1.63 \\ 
        $32 \times 32$ & 9.35e-03 & 1.93 & 1.38e-03 & 1.99 & 5.21e-03 & 1.93 \\ 
        $64 \times 64$ & 2.37e-03 & 1.98 & 3.45e-04 & 2.00 & 1.24e-03 & 2.07 \\ 
        $128 \times 128$ & 6.01e-04 & 1.98 & 8.61e-05 & 2.00 & 2.92e-04 & 2.09 \\ 
        \bottomrule
    \end{tabular}
    \label{tab: 7.3}
\end{table}
\begin{table}
    \centering
    \caption{Errors and convergence rates of mass-lumped CN-SAV-SDG$_0$ for Navier-Stokes equations.}
    \begin{tabular}{ccccccc}
        \toprule
        $n_x \times n_y$ & $\| \boldsymbol{\sigma} - \breve{\boldsymbol{\sigma}}_h \|_{l^2}$ & Rate & $\| \mathbf{u} - \mathbf{u}_h \|_{l^2}$ & Rate & $\| p - p_h \|_{l^2}$ & Rate \\ 
        \midrule
        $8 \times 8$ & 7.12e-02 & -- & 8.11e-03 & -- & 4.42e-02 & -- \\ 
        $16 \times 16$ & 1.79e-02 & 1.99 & 1.87e-03 & 2.11 & 1.11e-02 & 1.99 \\ 
        $32 \times 32$ & 4.54e-03 & 1.98 & 4.66e-04 & 2.01 & 2.77e-03 & 2.01 \\ 
        $64 \times 64$ & 1.14e-03 & 1.99 & 1.17e-04 & 2.00 & 6.53e-04 & 2.09 \\ 
        $128 \times 128$ & 2.85e-04 & 2.00 & 2.92e-05 & 2.00 & 1.88e-04 & 1.80 \\ 
        \bottomrule
    \end{tabular}
    \label{tab: 7.4}
\end{table}

\subsection{No-Flow Problem}

To assess the pressure robustness of the SDG$_0$ scheme, we consider a no-flow problem with a large pressure magnitude and small viscosity:
\begin{align*}
    \mathbf{u} \equiv \mathbf{0}, \qquad p = \frac{1}{\nu} (-3y^2 + 6y - 2),
\end{align*}
where $\nu = 10^{-3}$. Pressure robustness means that the velocity error is independent of both pressure and viscosity. Accordingly, the discrete velocity should vanish up to numerical precision. Figure~\ref{fig: 7.1} displays the computed velocity fields on a $64 \times 64$ grid, where the velocities are $O(10^{-11})$ in magnitude, confirming that our method accurately resolves the no-flow state. We further examine errors of the SDG$_0$ scheme and its mass-lumped variant over a range of mesh sizes. As shown in Figure~\ref{fig: 7.2}, both the the velocity and velocity-gradient errors are around $O(10^{-10})$ across all mesh sizes. These results provide strong evidence for the pressure robustness of our method.

\begin{figure}[htbp]
    \centering
    \begin{subfigure}[h]{0.48\textwidth}
        \centering
        \includegraphics{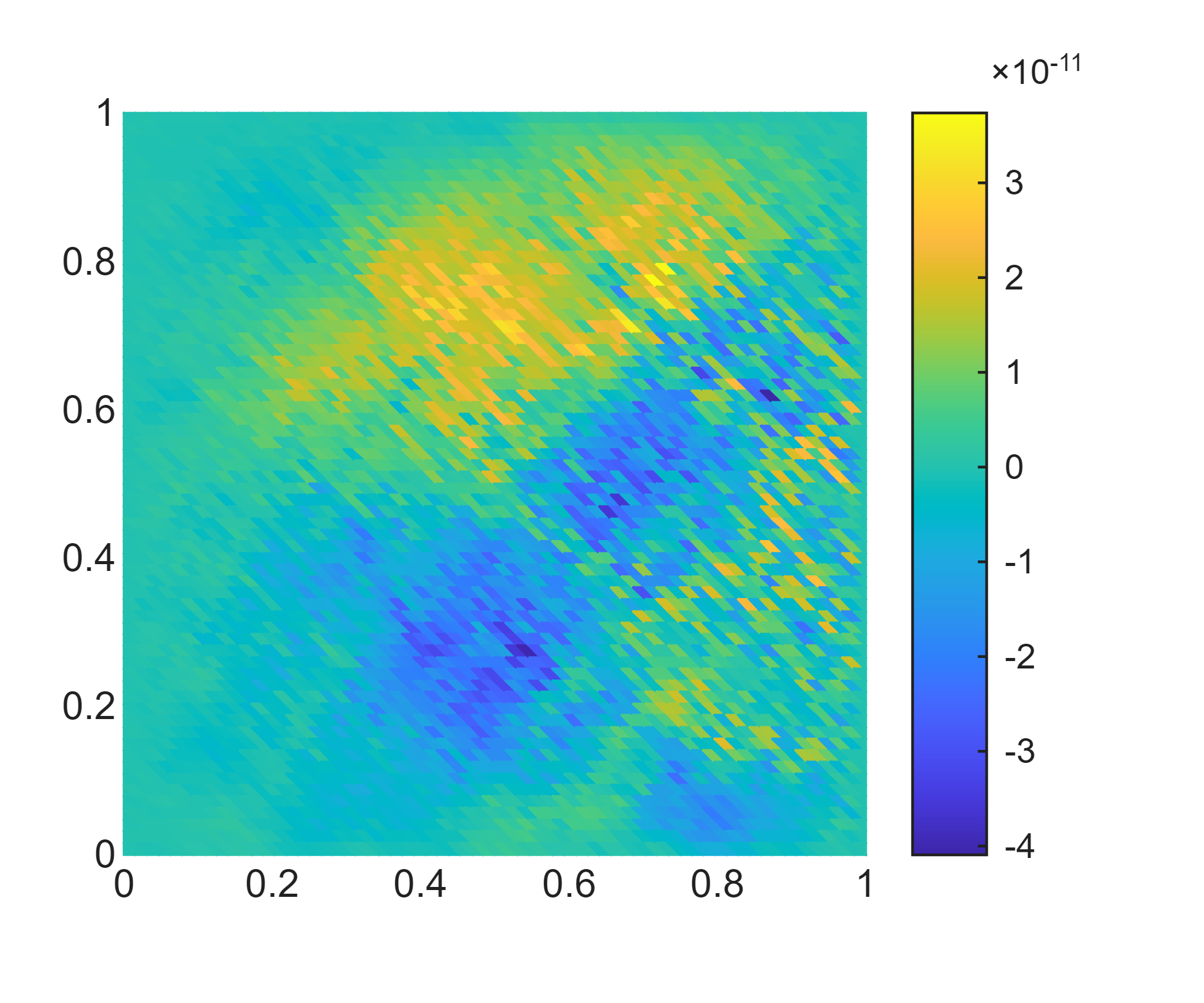}
        \caption{$u^x_h$.}
        \label{fig: 7.1.1}
    \end{subfigure}
    \begin{subfigure}[h]{0.48\textwidth}
        \centering
        \includegraphics{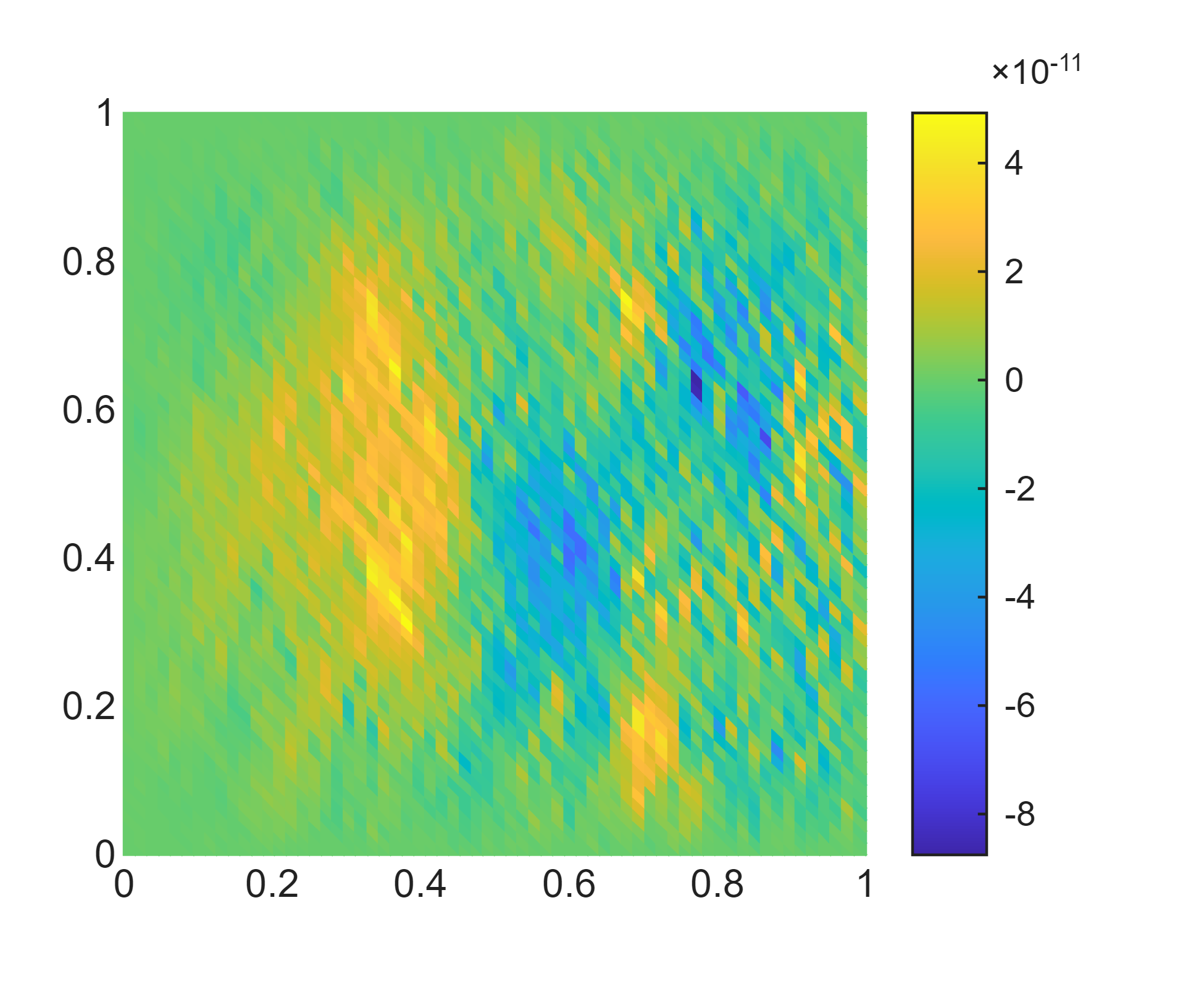}
        \caption{$u^y_h$.}
        \label{fig: 7.1.2}
    \end{subfigure}
    \caption{Numerical velocity fields for the no-flow problem.}
    \label{fig: 7.1}
\end{figure}
\begin{figure}[htbp]
    \centering
    \begin{subfigure}[h]{0.48\textwidth}
        \centering
        \includegraphics{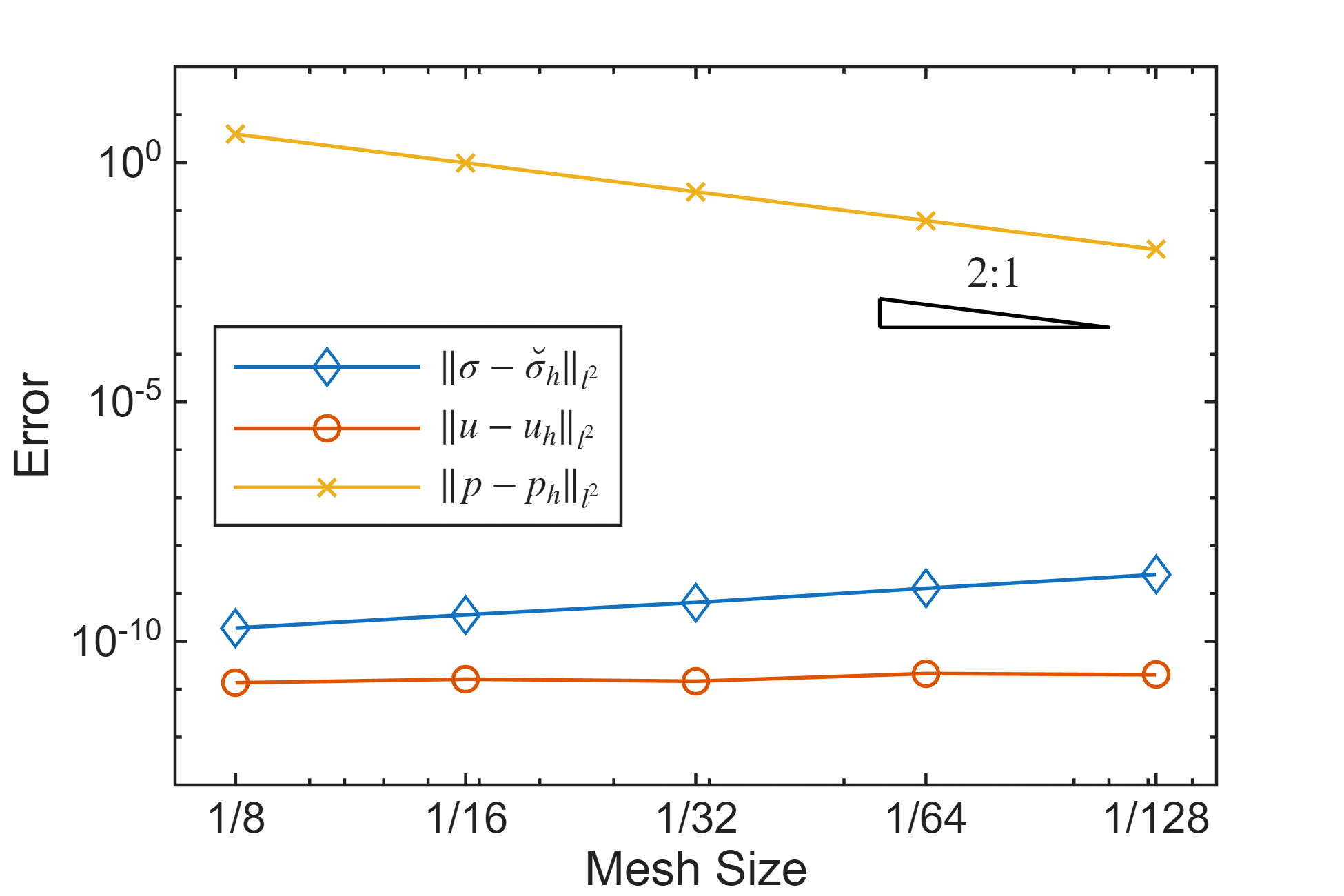}
        \caption{SDG$_0$.}
        \label{fig: 7.2.1}
    \end{subfigure}
    \begin{subfigure}[h]{0.48\textwidth}
        \centering
        \includegraphics{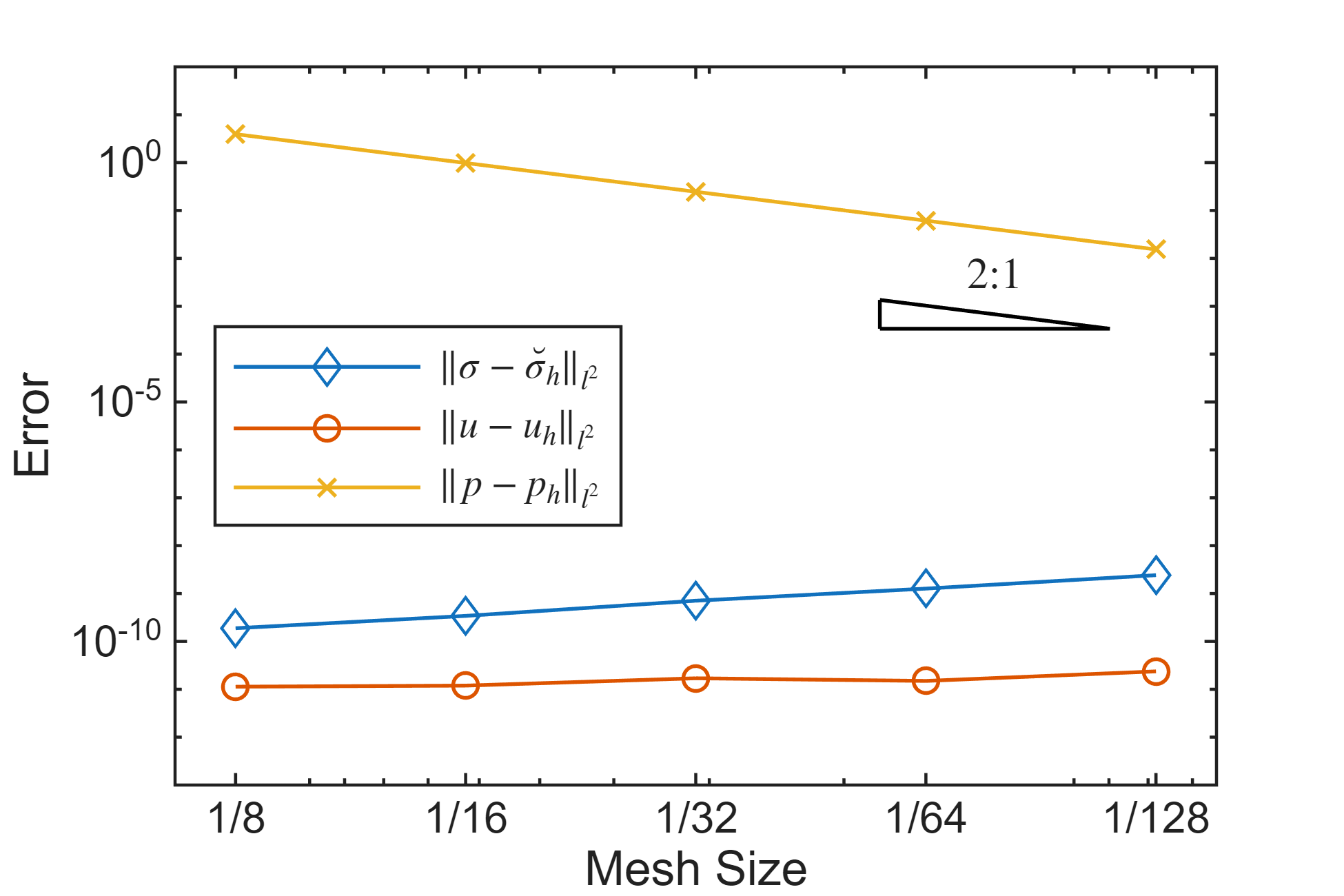}
        \caption{Mass-lumped SDG$_0$.}
        \label{fig: 7.2.2}
    \end{subfigure}
    \caption{Error convergence for the no-flow problem.}
    \label{fig: 7.2}
\end{figure}

\subsection{Taylor Vortex Flow}

We test the Taylor vortex flow to assess the accuracy of the CN-SAV-SDG$_0$ scheme for the Navier-Stokes equations with inhomogeneous boundary condition $\mathbf{u} |_{\partial \Omega} = \mathbf{g}$. The treatment of inhomogeneous boundary conditions follows a standard approach: seek $\mathbf{u}_h$ in
\begin{equation*}
    U_{h, g} = \{ \mathbf{v}_h \in U_h : \mathbf{v}_h \cdot \mathbf{n} = \mathbf{g} \cdot \mathbf{n} \text{ on } \partial \Omega \};
\end{equation*}
add the boundary contributions
\begin{align*}
    \sum_{e \in \partial \mathcal{E}_h} \int_e (\mathbf{g} \cdot \mathbf{t}) (\boldsymbol{\tau}_n \mathbf{n} \cdot \mathbf{t}) \, ds
\end{align*}
to the right side of \eqref{eq: CN-SAV-SDG 1}, and
\begin{align*}
    \frac{1}{2 \sqrt{E(\bar{\mathbf{u}}^{n + 1/2}_h) + \delta}} \int_{\partial \Omega} \frac{1}{2} | \mathbf{g}^{n + 1/2} |^2 \mathbf{g}^{n + 1/2} \cdot \mathbf{n} \, ds
\end{align*}
to the left side of \eqref{eq: CN-SAV-SDG 4}.

The Taylor vortex solution is
\begin{align*}
    u^x &= - \cos(\pi x) \sin(\pi y) \exp (-2 \pi^2 \nu t), \\
    u^y &= \sin(\pi x) \cos(\pi y) \exp (-2 \pi^2 \nu t), \\
    p &= - \frac{1}{4} (\cos(2 \pi x) + \cos(2 \pi y)) \exp (-4 \pi^2 \nu t).
\end{align*}
We set $\nu = 0.01$, $T = 1$ and $\Delta t = h$. Figure~\ref{fig: 7.3} shows the computed velocity and pressure fields on a $64 \times 64$ grid. Figure~\ref{fig: 7.4} reports the convergence histories of the CN-SAV-SDG$_0$ scheme and its mass-lumped variant, both of which exhibit second-order accuracy for all variables.

\begin{figure}[htbp]
    \centering
    \begin{subfigure}[h]{0.48\textwidth}
        \centering
        \includegraphics{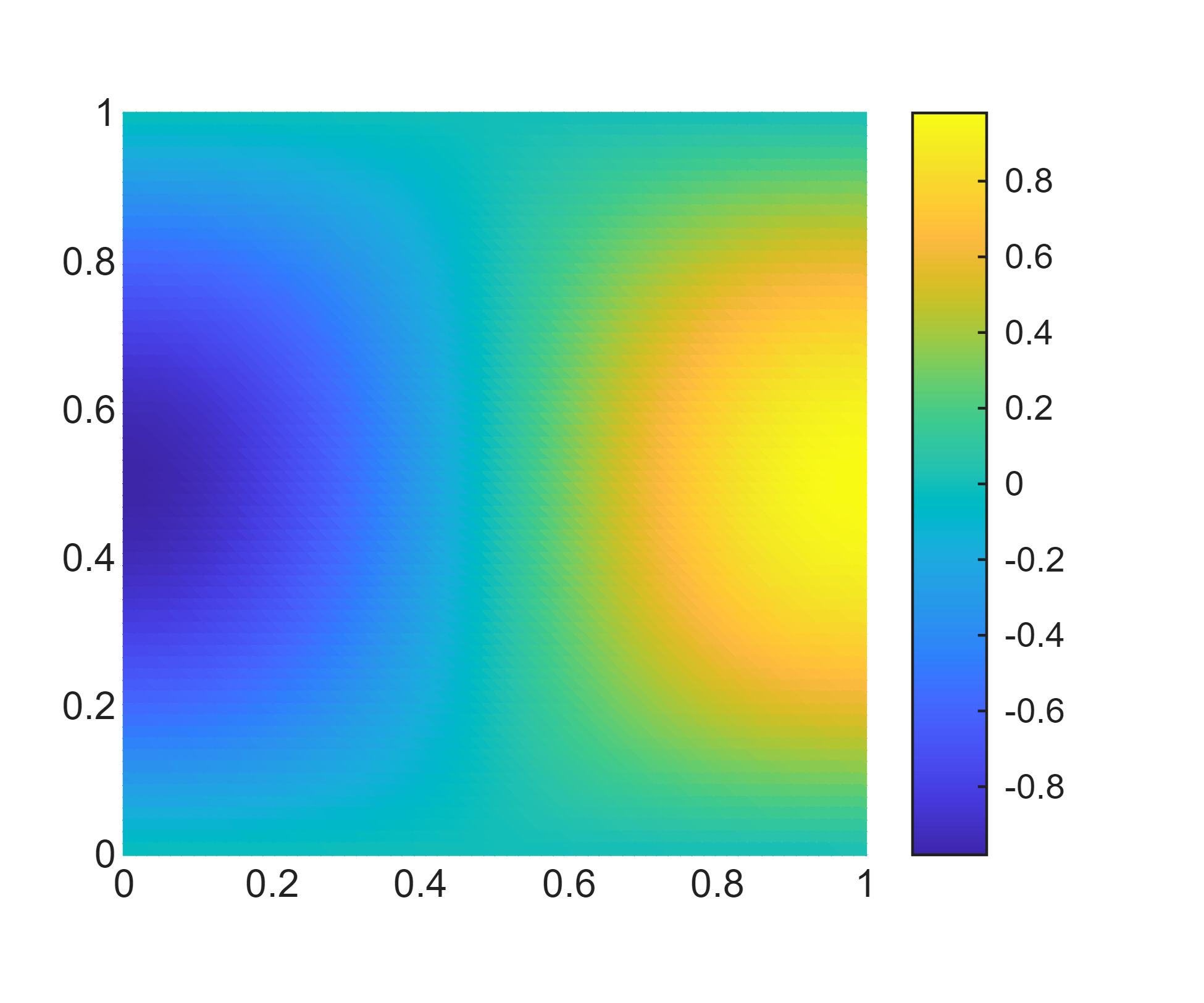}
        \caption{$u^x_h$.}
        \label{fig: 7.3.1}
    \end{subfigure}
    \begin{subfigure}[h]{0.48\textwidth}
        \centering
        \includegraphics{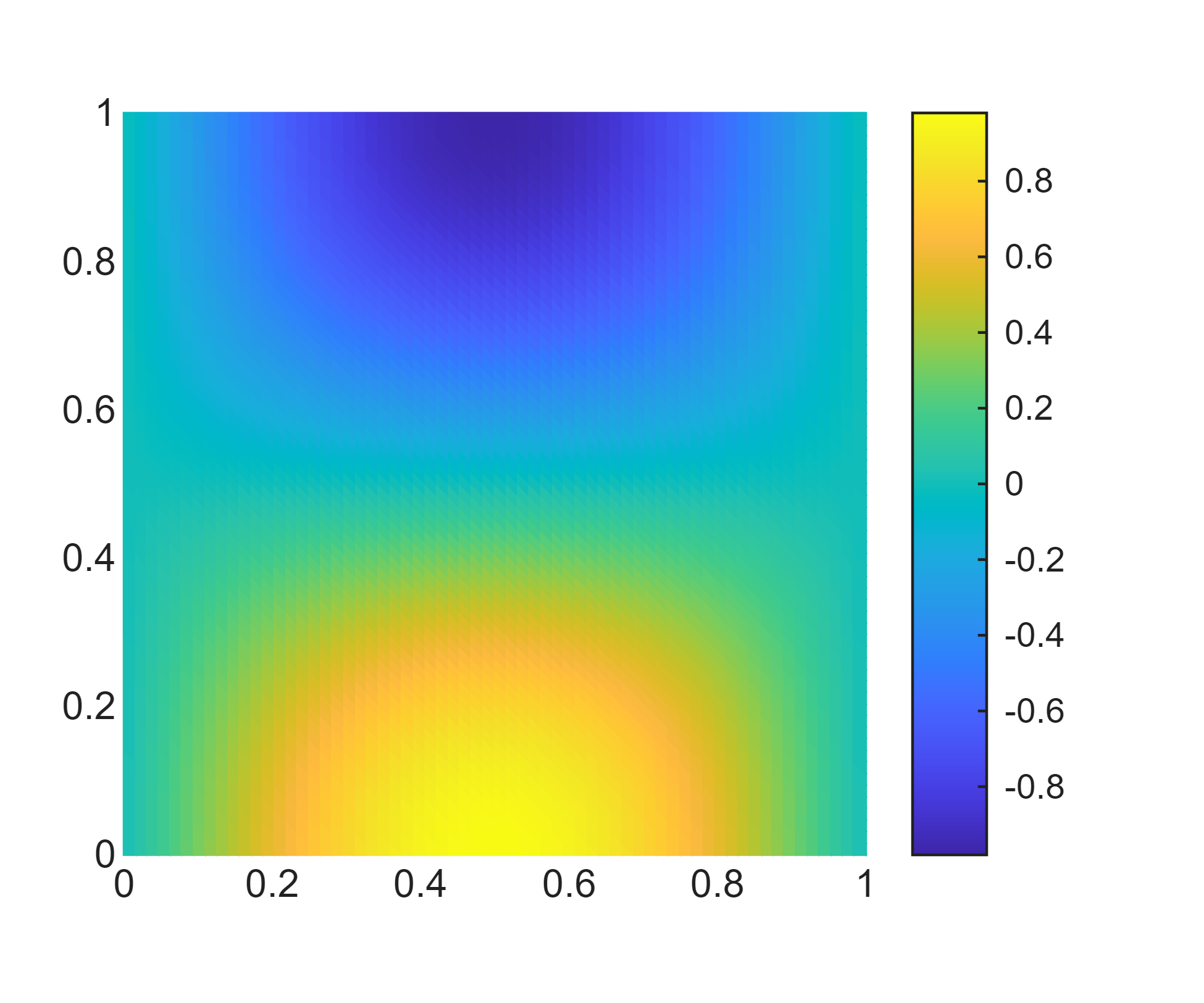}
        \caption{$u^y_h$.}
        \label{fig: 7.3.2}
    \end{subfigure}
    \begin{subfigure}[h]{0.48\textwidth}
        \centering
        \includegraphics{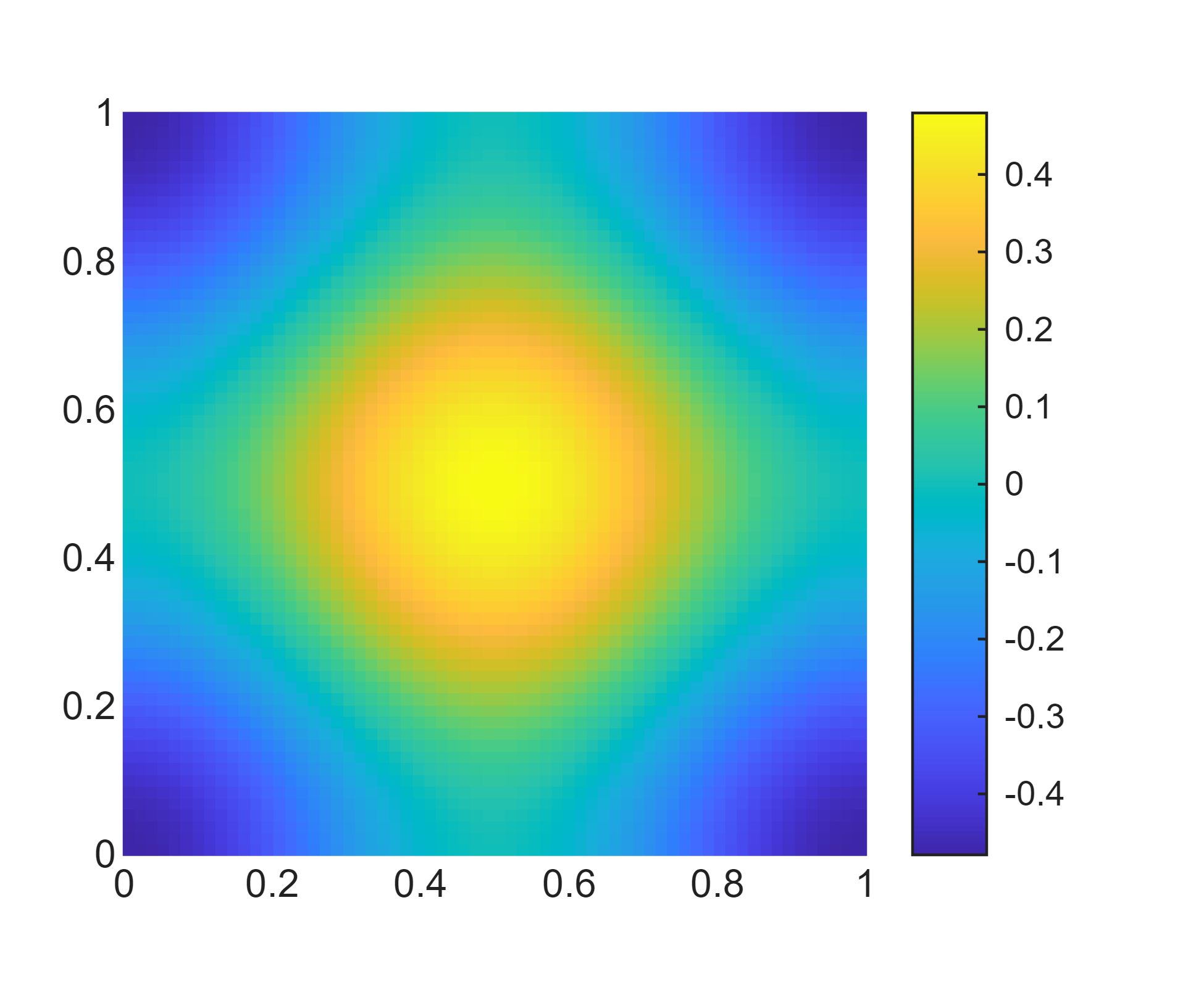}
        \caption{$p_h$.}
        \label{fig: 7.3.3}
    \end{subfigure}
    \caption{Numerical velocity and pressure fields for the Taylor vortex flow.}
    \label{fig: 7.3}
\end{figure}
\begin{figure}[htbp]
    \centering
    \begin{subfigure}[h]{0.48\textwidth}
        \centering
        \includegraphics{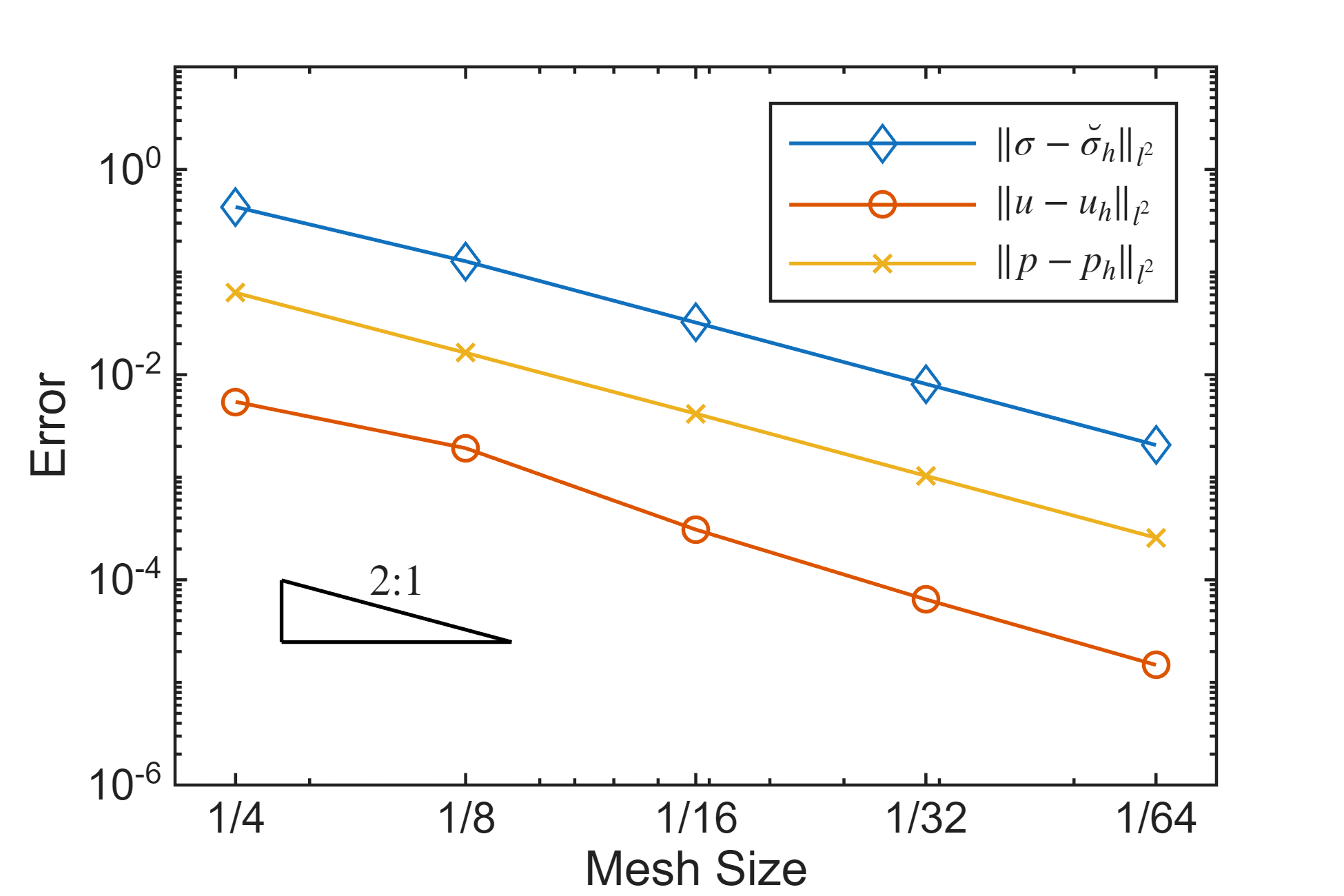}
        \caption{CN-SAV-SDG$_0$.}
        \label{fig: 7.4.1}
    \end{subfigure}
    \begin{subfigure}[h]{0.48\textwidth}
        \centering
        \includegraphics{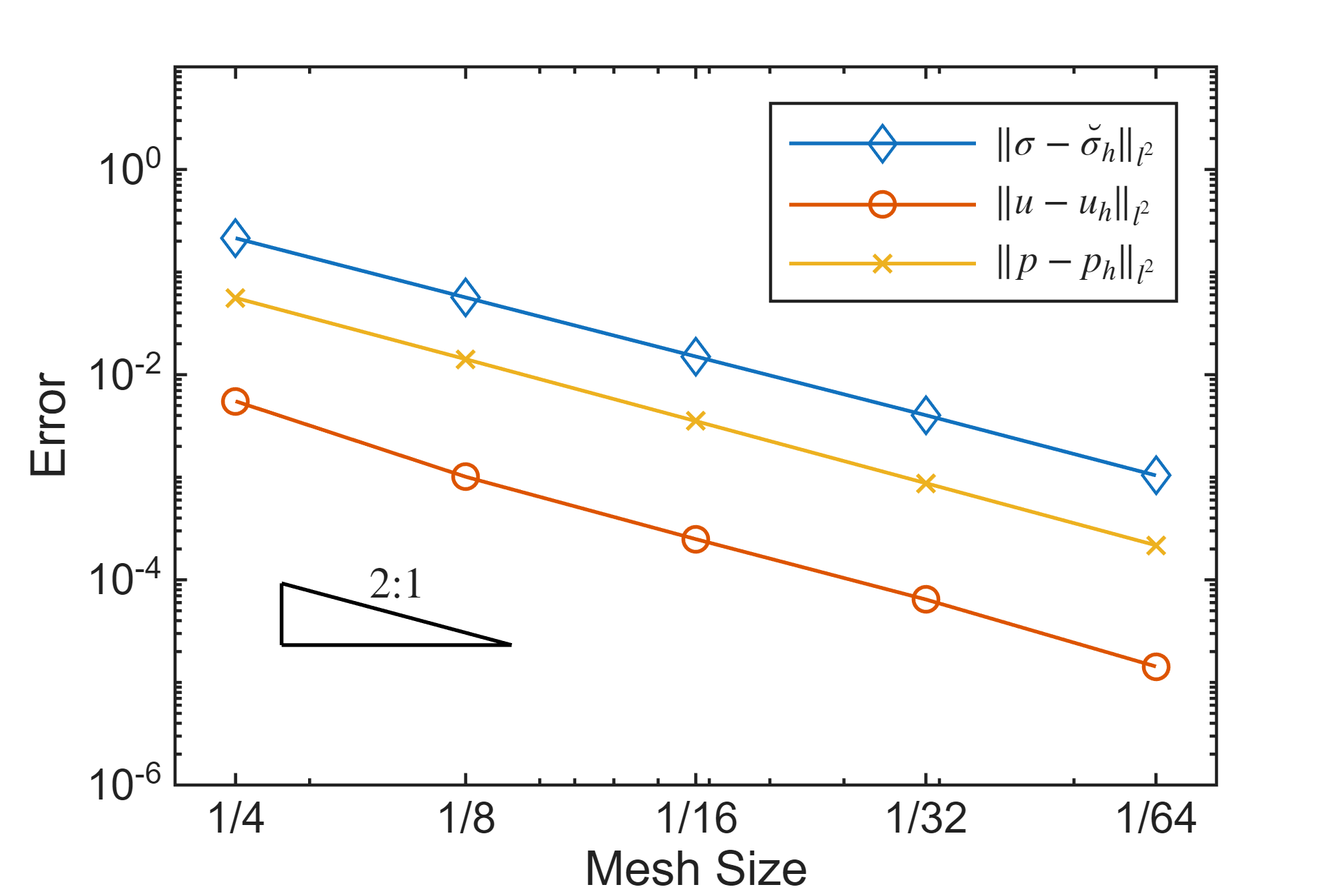}
        \caption{Mass-lumped CN-SAV-SDG$_0$.}
        \label{fig: 7.4.2}
    \end{subfigure}
    \caption{Error convergence for the Taylor vortex flow.}
    \label{fig: 7.4}
\end{figure}

\subsection{Lid-Driven Cavity Flow}

Finally, we consider a canonical benchmark: the lid-driven cavity flow \cite{ghia1982highre}. The fluid in the unit square is driven by a translating lid with unit speed along the top boundary, i.e., setting
\begin{align*}
    \mathbf{f} = \mathbf{0}, \quad
    u^x =
    \begin{cases}
        1, & y = 1, \\
        0, & \text{elsewhere on } \partial \Omega,
    \end{cases} \quad
    u^y = 0 \text{ on } \partial \Omega.
\end{align*}
Simulations are performed on a $64\times64$ grid and advanced in time until a steady state is reached. Figure~\ref{fig: 7.5} shows streamfunction contours for $\nu = 1/400$ and $\nu = 1/1000$, where the contour levels follow \cite{ghia1982highre}. The computed streamline patterns, including the primary vortex and secondary corner eddies, closely resemble the reference diagram in \cite{ghia1982highre}. Figures~\ref{fig: 7.6} and \ref{fig: 7.7} present the velocity profiles along the vertical and horizontal centerline, respectively, showing very good agreement with the benchmark data.

\begin{figure}[htbp]
    \centering
    \begin{subfigure}[h]{0.48\textwidth}
        \centering
        \includegraphics{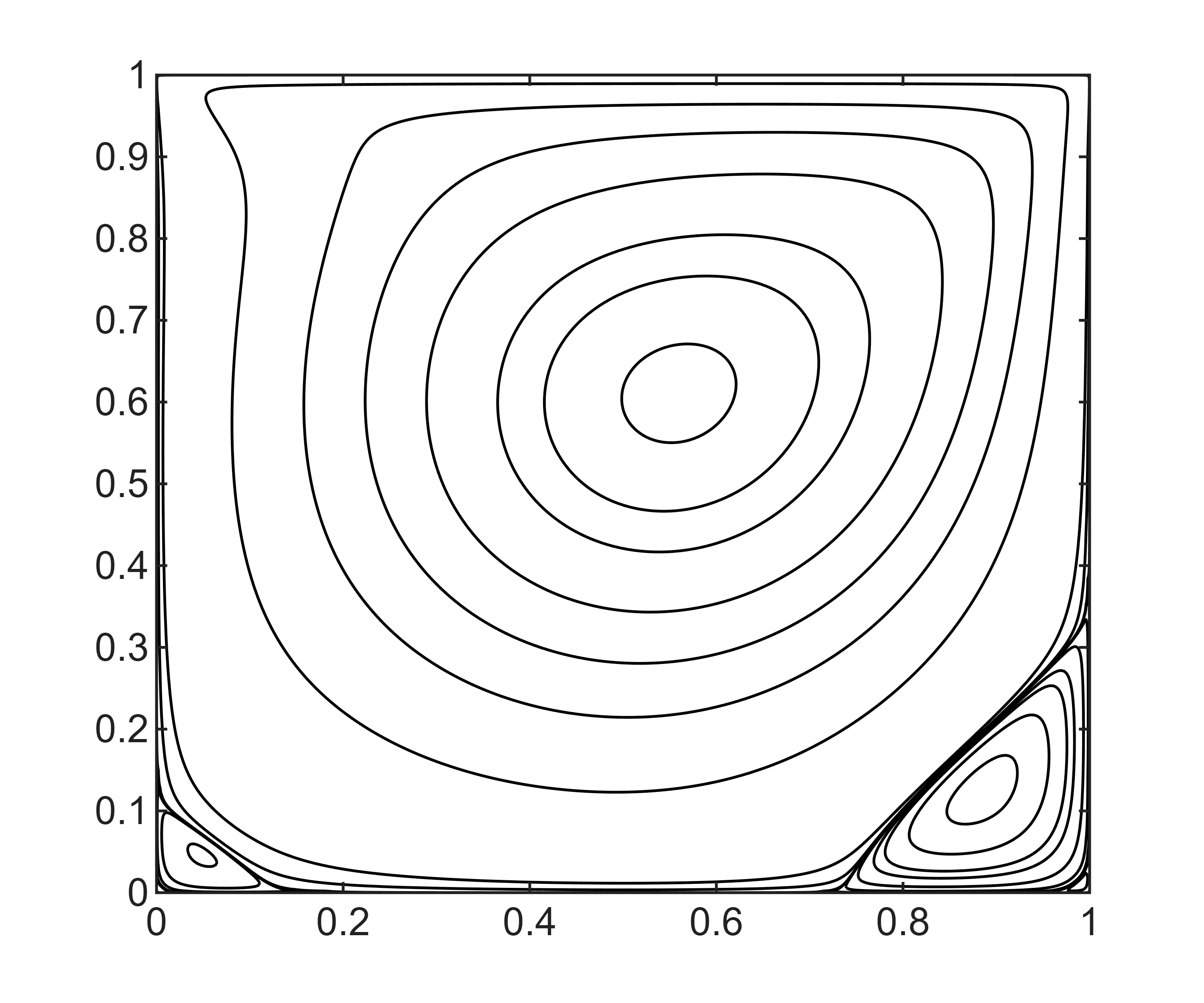}
        \caption{$\nu = 1/400$.}
        \label{fig: 7.5.1}
    \end{subfigure}
    \begin{subfigure}[h]{0.48\textwidth}
        \centering
        \includegraphics{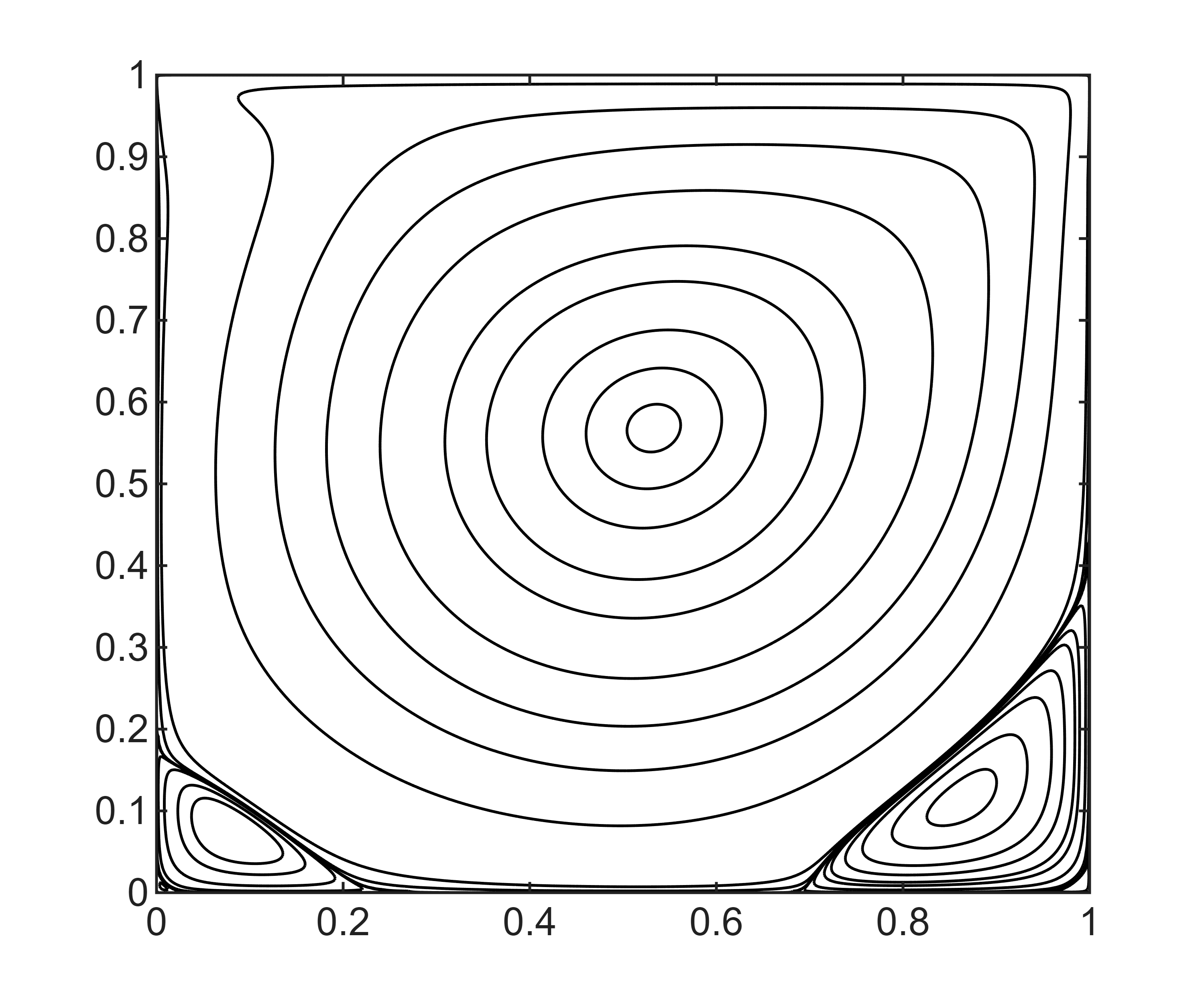}
        \caption{$\nu = 1/1000$.}
        \label{fig: 7.5.2}
    \end{subfigure}
    \caption{Streamline for the lid-driven cavity flow.}
    \label{fig: 7.5}
\end{figure}
\begin{figure}[htbp]
    \centering
    \begin{subfigure}[h]{0.48\textwidth}
        \centering
        \includegraphics{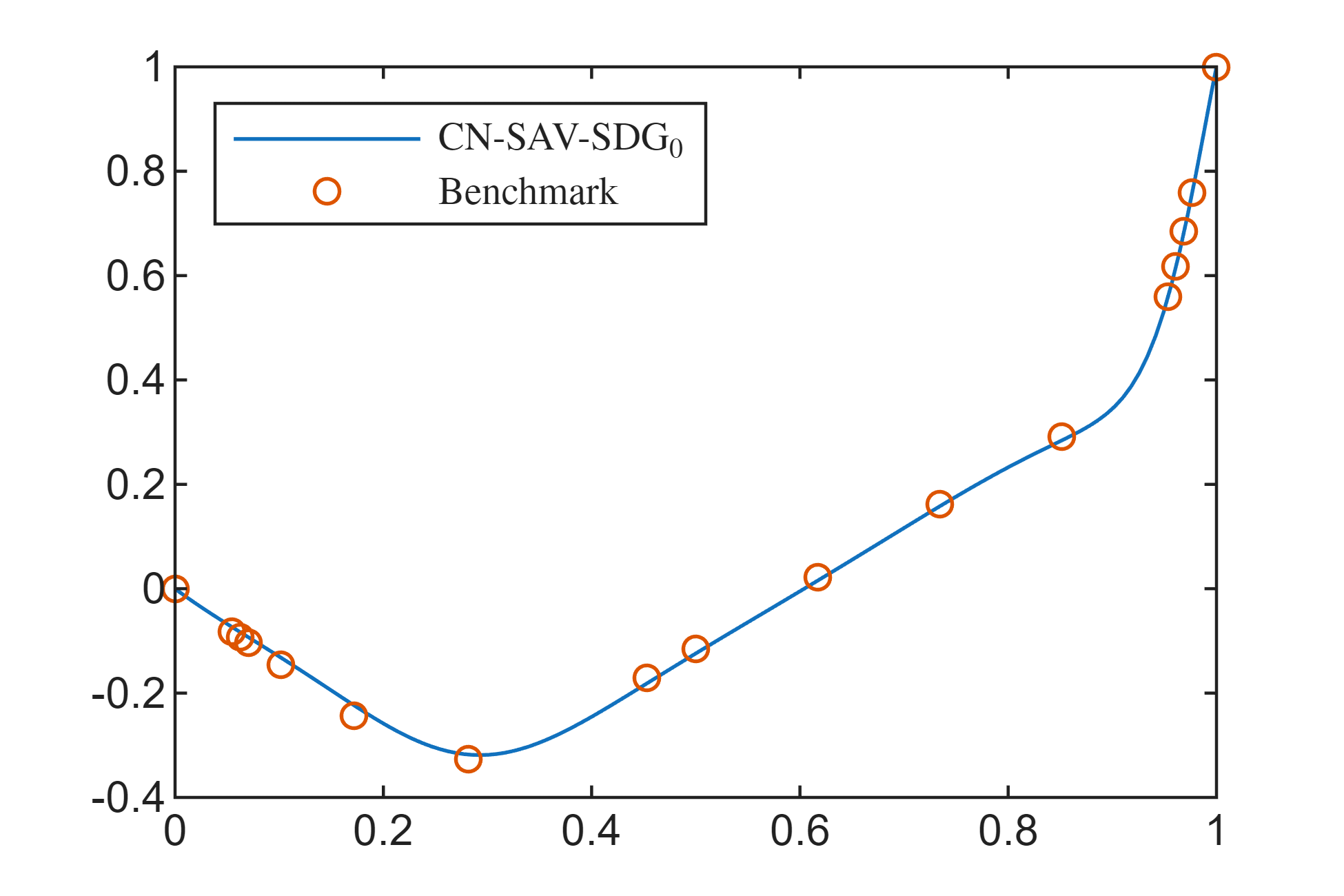}
        \caption{$\nu = 1/400$.}
        \label{fig: 7.6.1}
    \end{subfigure}
    \begin{subfigure}[h]{0.48\textwidth}
        \centering
        \includegraphics{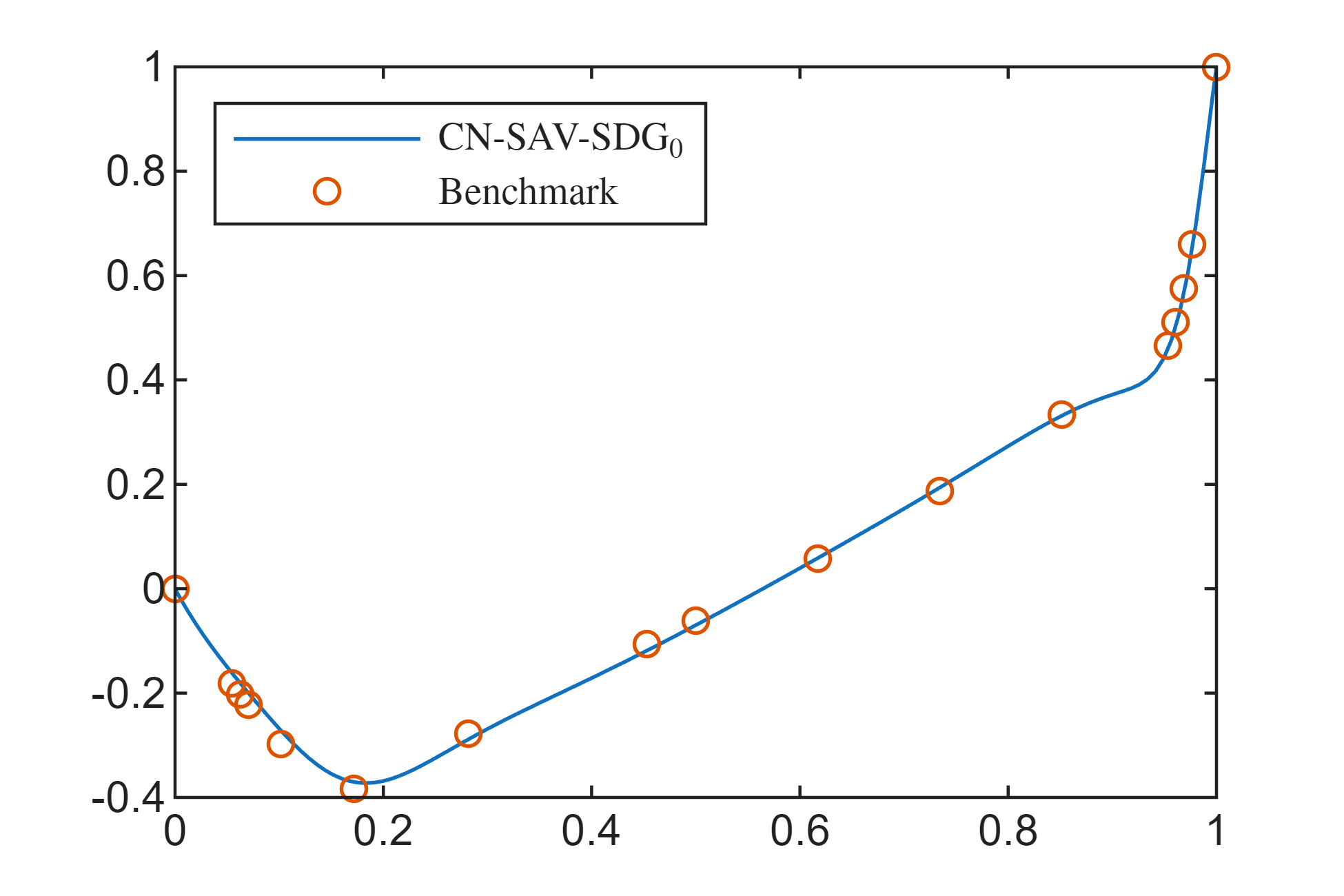}
        \caption{$\nu = 1/1000$.}
        \label{fig: 7.6.2}
    \end{subfigure}
    \caption{Profile of $u^x_h$ at $x = 0.5$ for the lid-driven cavity flow.}
    \label{fig: 7.6}
\end{figure}
\begin{figure}[htbp]
    \centering
    \begin{subfigure}[h]{0.48\textwidth}
        \centering
        \includegraphics{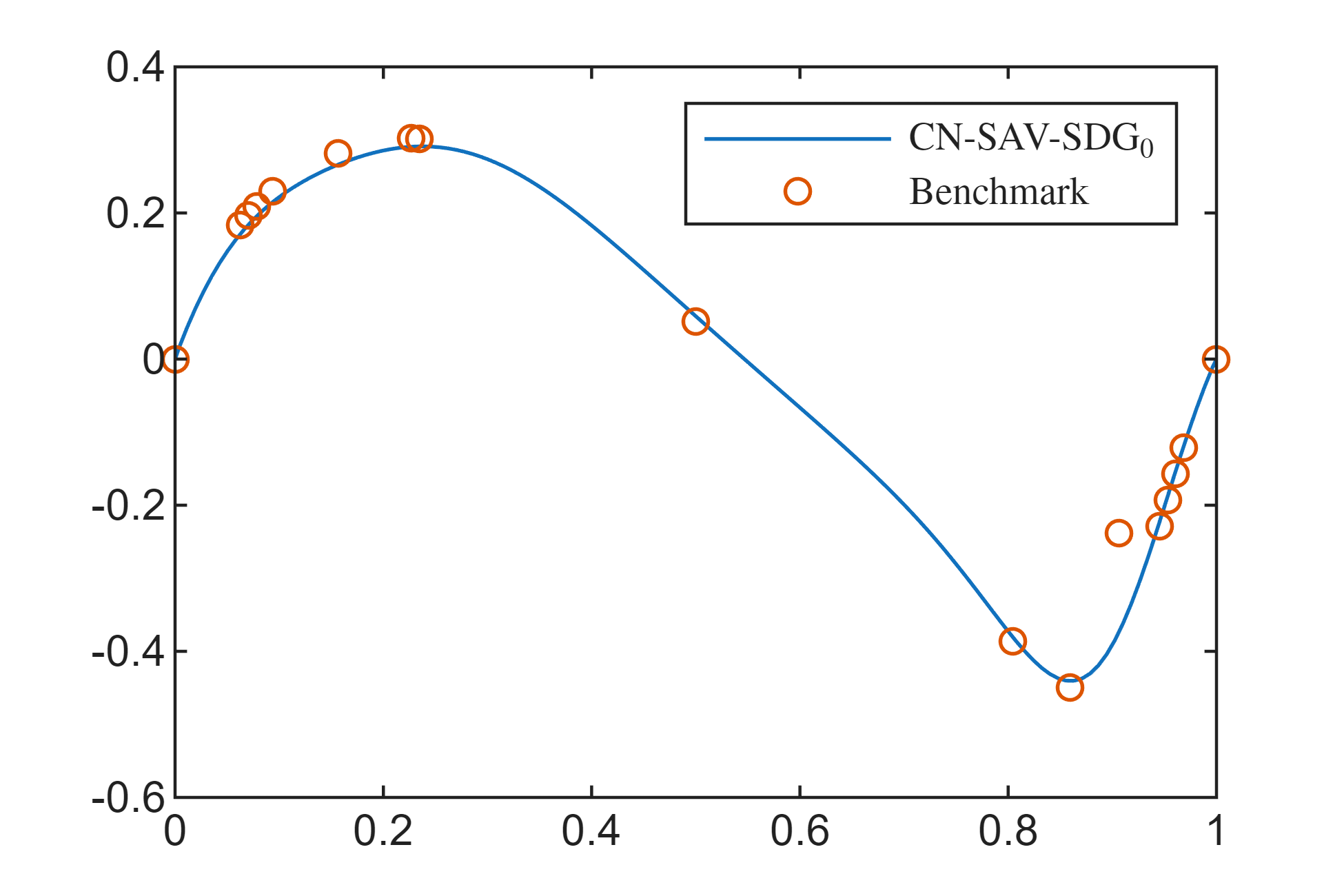}
        \caption{$\nu = 1/400$.}
        \label{fig: 7.7.1}
    \end{subfigure}
    \begin{subfigure}[h]{0.48\textwidth}
        \centering
        \includegraphics{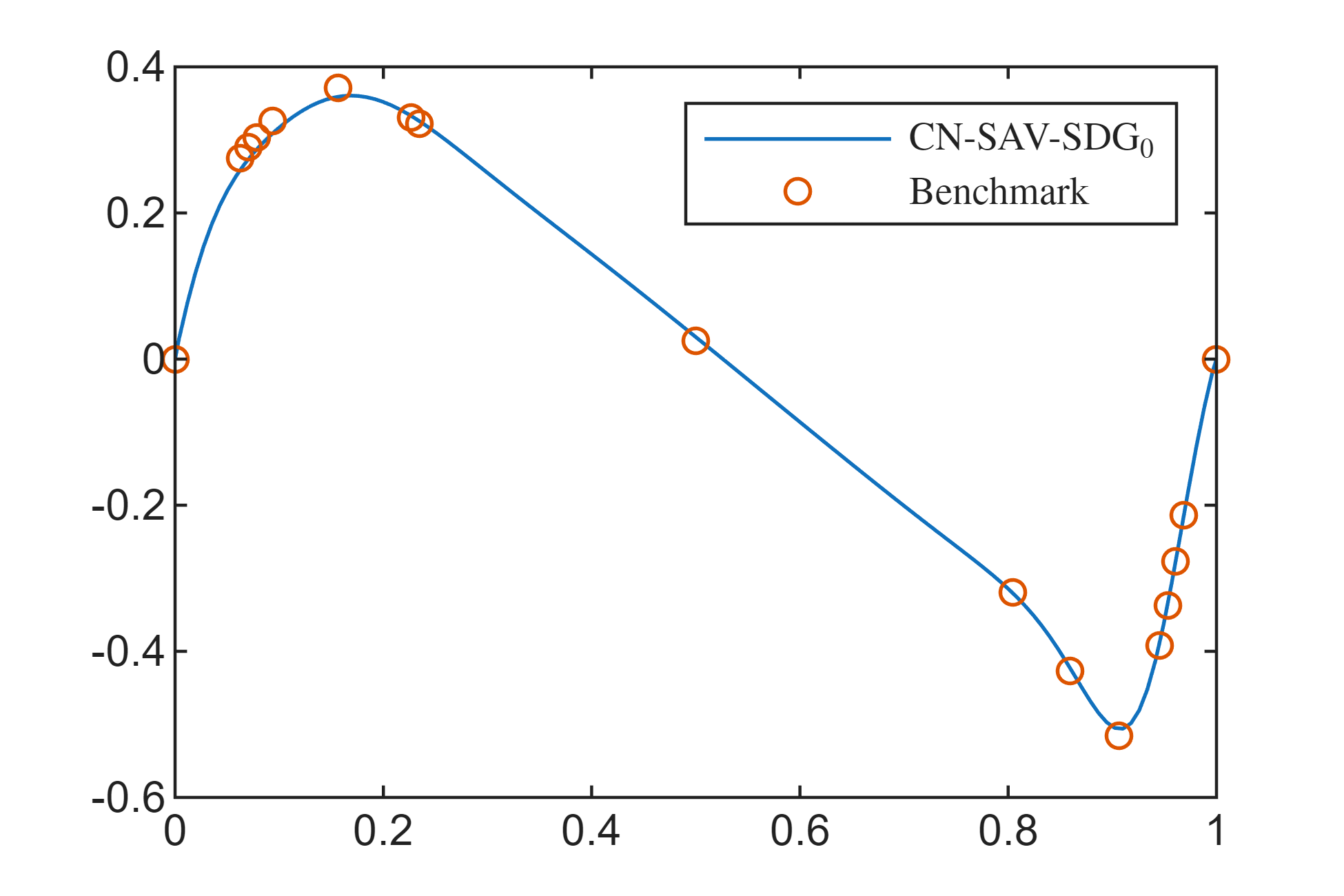}
        \caption{$\nu = 1/1000$.}
        \label{fig: 7.7.2}
    \end{subfigure}
    \caption{Profile of $u^y_h$ at $y = 0.5$ for the lid-driven cavity flow.}
    \label{fig: 7.7}
\end{figure}

\subsection{Performance Beyond the Theoretical Assumptions}

The theoretical analysis developed in this paper is based on two assumptions: sufficient smoothness of the exact solution \eqref{eq: solution regularity} and a bounded mesh aspect ratio \eqref{eq: mesh regularity}. Here, we deliberately relax each assumption in turn to investigate the behavior of the scheme outside its theoretical regime.

\subsubsection{Low Solution Regularity}

We solve the Stokes equations on the L-shaped domain $\Omega = (-1, 1)^2 \setminus \big( [0, 1) \times (-1, 0] \big)$, which has a re-entrant corner at the origin with interior angle $\omega = 3\pi/2$. Following the classical benchmark in \cite{wang2007new}, we set $\nu = 1$, $\mathbf{f} = \mathbf{0}$, and prescribe the exact solution in polar coordinates $(r,\theta)$ as
\begin{align*}
    \mathbf{u} &= r^{\lambda}
    \begin{bmatrix}
        (1 + \lambda) \sin\theta \, \psi(\theta) + \cos\theta \, \psi'(\theta) \\
        -(1 + \lambda) \cos\theta \, \psi(\theta) + \sin\theta \, \psi'(\theta)
    \end{bmatrix}, \\
    p &= -\frac{r^{\lambda - 1}}{1 - \lambda} \big[ (1 + \lambda)^2 \psi'(\theta) + \psi'''(\theta) \big],
\end{align*}
where
\begin{align*}
    \psi(\theta) = \frac{\sin((1 + \lambda)\theta) \cos(\lambda \omega)}{1 + \lambda}
    - \cos((1 + \lambda)\theta)
    - \frac{\sin((1 - \lambda)\theta) \cos(\lambda \omega)}{1 - \lambda}
    + \cos((1 - \lambda)\theta),
\end{align*}
and $\lambda \approx 0.54448373678246$ is the smallest positive root of $\sin(\lambda \omega) + \lambda \sin\omega = 0$. As shown in \cite{wang2007new}, this solution only satisfies $(\mathbf{u}, p) \in [H^{1+\lambda}(\Omega)]^2 \times H^\lambda(\Omega)$, so the regularity assumption in \eqref{eq: solution regularity} is violated, especially near the re-entrant corner.

Tables~\ref{tab: 7.5} and \ref{tab: 7.6} report the errors and convergence rates for the original and mass-lumped schemes. As discussed in Remark~\ref{rmk: solution regularity}, the reduced solution regularity limits the convergence behavior to that expected by the standard SDG theory. Specifically, the low-regularity a priori estimates in \cite{zhao2019staggered} imply convergence rates of $\lambda$ for the velocity gradient and the pressure, and of $2\lambda$ for the velocity. Our results are in close agreement with this estimate: the velocity gradient error converges at a rate close to $\lambda \approx 0.54$, while the velocity error exhibits the expected $2\lambda \approx 1.09$ rate. The pressure converges at a rate of about $0.85$, slightly better than the theoretical value of $\lambda$.

\begin{table}[htbp]
    \centering
    \caption{Errors and convergence rates of SDG$_0$ for Stokes equations on the L-shaped domain.}
    \begin{tabular}{ccccccc}
        \toprule
        $n_x \times n_y$ & $\| \boldsymbol{\sigma} - \breve{\boldsymbol{\sigma}}_h \|_{l^2}$ & Rate & $\| \mathbf{u} - \mathbf{u}_h \|_{l^2}$ & Rate & $\| p - p_h \|_{l^2}$ & Rate \\
        \midrule
        $8 \times 8$ & 1.36e+00 & -- & 1.62e-01 & -- & 2.93e+00 & -- \\
        $16 \times 16$ & 9.92e-01 & 0.45 & 1.00e-01 & 0.69 & 1.67e+00 & 0.81 \\
        $32 \times 32$ & 7.02e-01 & 0.50 & 5.43e-02 & 0.89 & 9.00e-01 & 0.89 \\
        $64 \times 64$ & 4.89e-01 & 0.52 & 2.75e-02 & 0.98 & 4.87e-01 & 0.89 \\
        $128 \times 128$ & 3.38e-01 & 0.53 & 1.35e-02 & 1.03 & 2.75e-01 & 0.83 \\
        \bottomrule
    \end{tabular}
    \label{tab: 7.5}
\end{table}
\begin{table}[htbp]
    \centering
    \caption{Errors and convergence rates of mass-lumped SDG$_0$ for Stokes equations on the L-shaped domain.}
    \begin{tabular}{ccccccc}
        \toprule
        $n_x \times n_y$ & $\| \boldsymbol{\sigma} - \breve{\boldsymbol{\sigma}}_h \|_{l^2}$ & Rate & $\| \mathbf{u} - \mathbf{u}_h \|_{l^2}$ & Rate & $\| p - p_h \|_{l^2}$ & Rate \\
        \midrule
        $8 \times 8$ & 1.37e+00 & -- & 1.63e-01 & -- & 2.95e+00 & -- \\
        $16 \times 16$ & 1.00e+00 & 0.45 & 1.01e-01 & 0.69 & 1.69e+00 & 0.80 \\
        $32 \times 32$ & 7.13e-01 & 0.49 & 5.47e-02 & 0.89 & 9.11e-01 & 0.89 \\
        $64 \times 64$ & 4.93e-01 & 0.53 & 2.79e-02 & 0.97 & 4.93e-01 & 0.89 \\
        $128 \times 128$ & 3.39e-01 & 0.54 & 1.36e-02 & 1.04 & 2.76e-01 & 0.83 \\
        \bottomrule
    \end{tabular}
    \label{tab: 7.6}
\end{table}

\subsubsection{Stretched Grids}

We next examine the robustness of the scheme with respect to the mesh aspect ratio $R = h^x/h^y$, using the test problem in Section~\ref{sec: 7.1}. Two settings are considered. First, we fix the aspect ratio at $R = 10$ and refine both directions simultaneously. Tables~\ref{tab: 7.7} and \ref{tab: 7.8} report the results for the original and mass-lumped schemes. All variables preserve second-order convergence, and the errors are not worse than those in the isotropic case (Tables~\ref{tab: 7.1} and \ref{tab: 7.2}). Second, we fix the partition in the $x$ direction with $n_x = 8$ and refine only the $y$ direction, so that the aspect ratio $R$ increases under refinement. Figure~\ref{fig: 7.8} plots the errors of all variables against $R$ for both schemes. The errors remain essentially unchanged as the mesh becomes increasingly anisotropic. These results demonstrate that the scheme is insensitive to the mesh aspect ratio.
\begin{table}[htbp]
    \centering
    \caption{Errors and convergence rates of SDG$_0$ for Stokes equations on stretched grids.}
    \begin{tabular}{ccccccc}
        \toprule
        $n_x \times n_y$ & $\| \boldsymbol{\sigma} - \breve{\boldsymbol{\sigma}}_h \|_{l^2}$ & Rate & $\| \mathbf{u} - \mathbf{u}_h \|_{l^2}$ & Rate & $\| p - p_h \|_{l^2}$ & Rate \\
        \midrule
        $4 \times 40$ & 3.00e-01 & -- & 4.99e-02 & -- & 6.55e-02 & -- \\
        $8 \times 80$ & 8.84e-02 & 1.77 & 1.37e-02 & 1.86 & 2.73e-02 & 1.26 \\
        $16 \times 160$ & 2.33e-02 & 1.92 & 3.54e-03 & 1.95 & 7.79e-03 & 1.81 \\
        $32 \times 320$ & 5.84e-03 & 2.00 & 8.91e-04 & 1.99 & 1.99e-03 & 1.97 \\
        $64 \times 640$ & 1.48e-03 & 1.98 & 2.24e-04 & 1.99 & 5.07e-04 & 1.97 \\
        \bottomrule
    \end{tabular}
    \label{tab: 7.7}
\end{table}
\begin{table}[htbp]
    \centering
    \caption{Errors and convergence rates of mass-lumped SDG$_0$ for Stokes equations on stretched grids.}
    \begin{tabular}{ccccccc}
        \toprule
        $n_x \times n_y$ & $\| \boldsymbol{\sigma} - \breve{\boldsymbol{\sigma}}_h \|_{l^2}$ & Rate & $\| \mathbf{u} - \mathbf{u}_h \|_{l^2}$ & Rate & $\| p - p_h \|_{l^2}$ & Rate \\
        \midrule
        $4 \times 40$ & 2.69e-01 & -- & 2.19e-02 & -- & 2.23e-01 & -- \\
        $8 \times 80$ & 7.86e-02 & 1.78 & 6.86e-03 & 1.68 & 7.77e-02 & 1.52 \\
        $16 \times 160$ & 2.17e-02 & 1.85 & 1.94e-03 & 1.82 & 2.23e-02 & 1.80 \\
        $32 \times 320$ & 5.52e-03 & 1.98 & 4.99e-04 & 1.96 & 5.68e-03 & 1.97 \\
        $64 \times 640$ & 1.39e-03 & 1.99 & 1.26e-04 & 1.98 & 1.43e-03 & 1.99 \\
        \bottomrule
    \end{tabular}
    \label{tab: 7.8}
\end{table}
\begin{figure}[htbp]
    \centering
    \begin{subfigure}[h]{0.48\textwidth}
        \centering
        \includegraphics{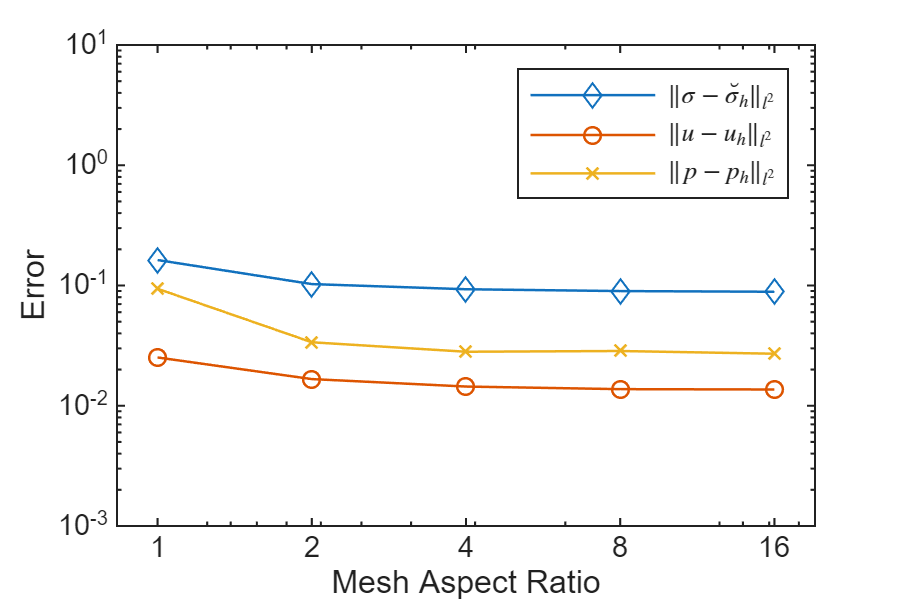}
        \caption{SDG$_0$.}
        \label{fig: 7.8.1}
    \end{subfigure}
    \begin{subfigure}[h]{0.48\textwidth}
        \centering
        \includegraphics{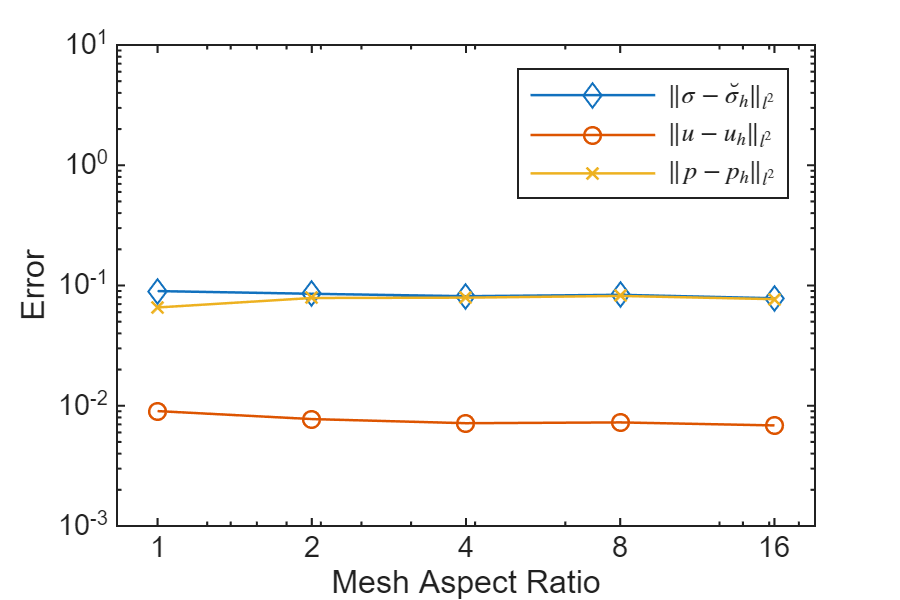}
        \caption{Mass-lumped SDG$_0$.}
        \label{fig: 7.8.2}
    \end{subfigure}
    \caption{Error evolution w.r.t. mesh aspect ratio.}
    \label{fig: 7.8}
\end{figure}

\subsection{Curved Domains and Non-Cartesian Meshes}

All experiments above are carried out on rectangular domains and Cartesian grids. Here we consider curved domains and quadrilateral meshes that are smooth diffeomorphic images of a rectangle and a Cartesian grid. Specifically, a Cartesian grid on the reference rectangle $\hat{\Omega} = [0,1]^2$ is mapped by a smooth transformation $\Phi$ to a body-fitted quadrilateral mesh on the physical domain $\Omega = \Phi(\hat{\Omega})$. The staggered construction in Section~\ref{sec 2} carries over directly. For the discrete scheme, the explicit pointwise formulation in Section~\ref{sec 4} is no longer available, and we therefore revert to the variational form in Section~\ref{sec 3}. In addition, the pointwise $l^2$ norms in Section~\ref{sec 5} are modified as
\begin{align*}
    \| \boldsymbol{\sigma}_h \|^2_{l^2} &= \sum_{(i, j^+) \in \Lambda^V_h} h^x_i h^y_{j^+} (\boldsymbol{\sigma}_h \mathbf{n} \cdot \mathbf{t})^2 |_{(x_i, y_{j^+})} + \sum_{(i^+, j) \in \Lambda^H_h} h^x_{i^+} h^y_j (\boldsymbol{\sigma}_h \mathbf{n} \cdot \mathbf{t})^2 |_{(x_{i^+}, y_j)} \\ &\quad + \sum_{(i^+, j^+) \in \Lambda^D_h} h^x_{i^+} h^y_{j^+} (\boldsymbol{\sigma}_h \mathbf{n})^2 |_{(x_{i^+}, y_{j^+})}, \\
    \| \mathbf{u}_h \|^2_{l^2} &= \sum_{(i, j^+) \in \Lambda^V_h} h^x_i h^y_{j^+} (\mathbf{u}_h \cdot \mathbf{n})^2 |_{(x_i, y_{j^+})} + \sum_{(i^+, j) \in \Lambda^H_h} h^x_{i^+} h^y_j (\mathbf{u}_h \cdot \mathbf{n})^2 |_{(x_{i^+}, y_j)}, \\
    \| p_h \|^2_{l^2} &= \sum_{(i^+, j^+) \in \Lambda^D_h} h^x_{i^+} h^y_{j^+} (p_h)^2 |_{(x_{i^+}, y_{j^+})}.
\end{align*}

\paragraph{Distorted quadrilateral meshes}
Consider the sinusoidal map used in \cite{kreeft2013mixed}:
\begin{equation*}
    \Phi(\xi, \eta) = \Big( \frac{1}{2} + \frac{1}{2} \big( \xi + \frac{1}{5} \sin(\pi \xi) \sin(\pi \eta) \big), \ \frac{1}{2} + \frac{1}{2} \big( \eta + \frac{1}{5} \sin(\pi \xi) \sin(\pi \eta) \big) \Big).
\end{equation*}
This mapping preserves the straight boundary and transforms the Cartesian grid into a distorted quadrilateral mesh as shown in Figure~\ref{fig: 7.9}. Set $\nu = 1$ and adopt the classical Taylor-Green vortex as the solution,
\begin{align*}
    u^x &= - \cos(\pi x) \sin(\pi y), \\
    u^y &= \sin(\pi x) \cos(\pi y), \\
    p &= -\frac{1}{4} \big( \cos(2 \pi x) + \cos(2 \pi y) \big).
\end{align*}
The errors and convergence rates are summarized in Table~\ref{tab: 7.9}, which indicates that all variables achieve almost second-order convergence on such distorted quadrilateral meshes.

\begin{figure}[htbp]
    \centering
    \begin{subfigure}[h]{0.4\textwidth}
        \centering
        \includegraphics{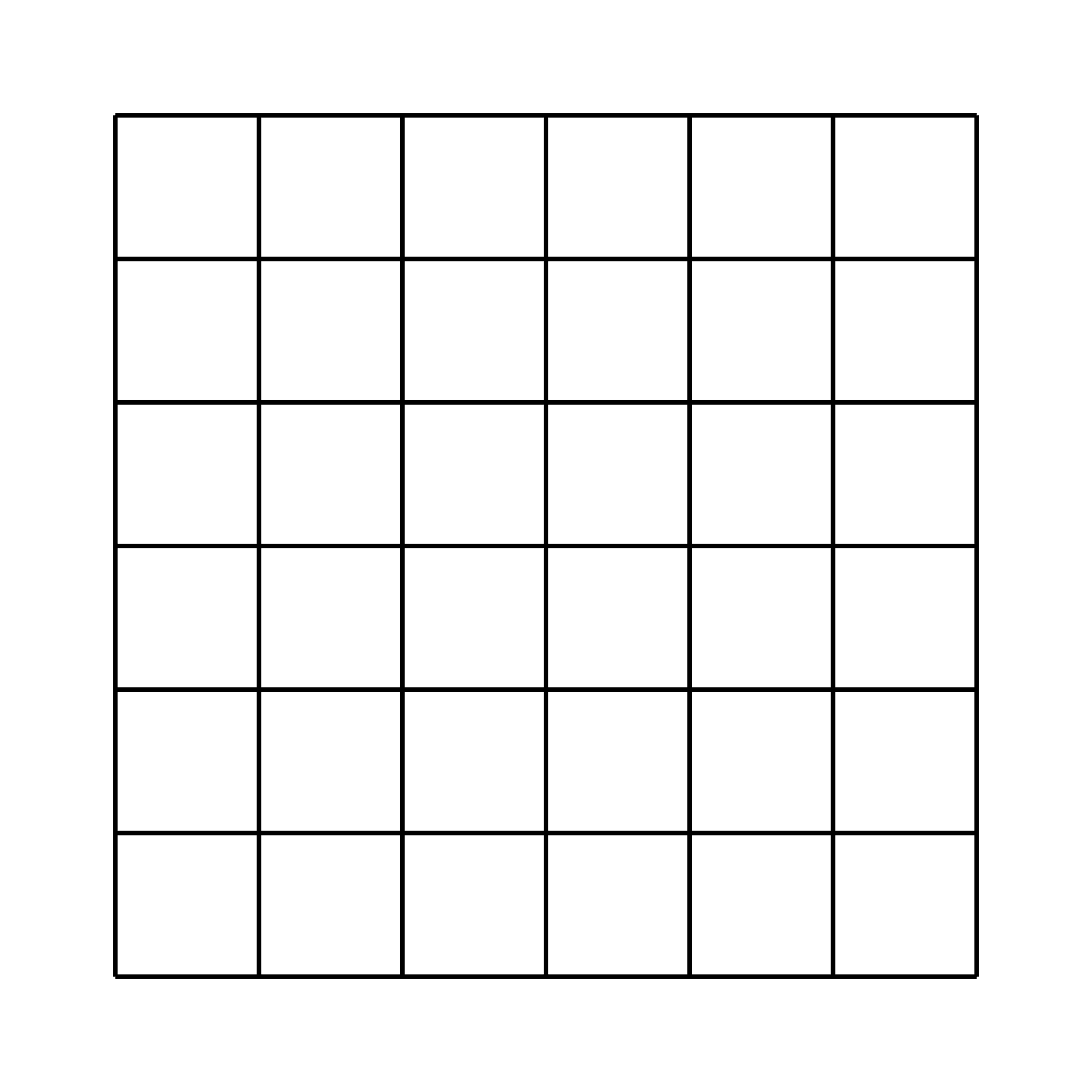}
        \caption{Reference grid.}
        \label{fig: 7.9.1}
    \end{subfigure}
    $\to$
    \begin{subfigure}[h]{0.4\textwidth}
        \centering
        \includegraphics{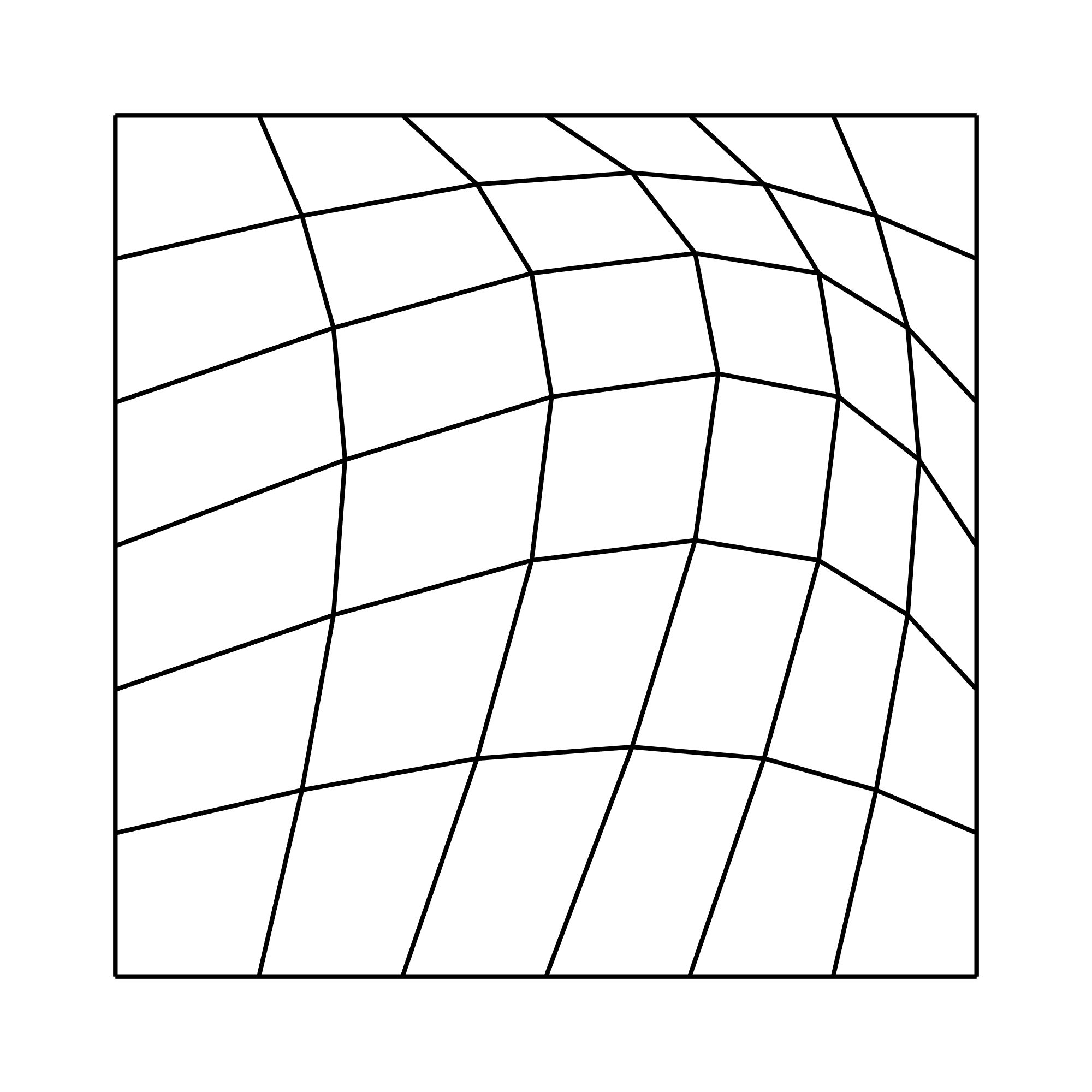}
        \caption{Physical mesh.}
        \label{fig: 7.9.2}
    \end{subfigure}
    \caption{Cartesian reference grid and its image under the sinusoidal mapping.}
    \label{fig: 7.9}
\end{figure}

\begin{table}[htbp]
    \centering
    \caption{Errors and convergence rates of SDG$_0$ for Stokes equations on the distorted quadrilateral mesh.}
    \begin{tabular}{ccccccc}
        \toprule
        $n_x \times n_y$ & $\| \boldsymbol{\sigma} - \boldsymbol{\sigma}_h \|_{l^2}$ & Rate & $\| \mathbf{u} - \mathbf{u}_h \|_{l^2}$ & Rate & $\| p - p_h \|_{l^2}$ & Rate \\
        \midrule
        $8 \times 8$ & 7.63e-02 & -- & 3.16e-03 & -- & 7.98e-02 & -- \\
        $16 \times 16$ & 2.15e-02 & 1.83 & 9.67e-04 & 1.71 & 2.41e-02 & 1.72 \\
        $32 \times 32$ & 5.83e-03 & 1.88 & 2.58e-04 & 1.91 & 7.02e-03 & 1.78 \\
        $64 \times 64$ & 1.56e-03 & 1.91 & 6.61e-05 & 1.96 & 2.00e-03 & 1.81 \\
        $128 \times 128$ & 4.11e-04 & 1.92 & 1.67e-05 & 1.99 & 5.61e-04 & 1.84 \\
        \bottomrule
    \end{tabular}
    \label{tab: 7.9}
\end{table}

\paragraph{Annulus domain}
We next consider the full annulus $\Omega = \{ (x, y) : R_1^2 < x^2 + y^2 < R_2^2 \}$ with $R_1 = 1$ and $R_2 = 2$, which is obtained from the reference square by the polar map (see Figure~\ref{fig: 7.10}):
\begin{equation*}
    \Phi(\xi, \eta) = (r \cos \theta, r \sin \theta), \quad r = R_1 + \xi (R_2 - R_1), \quad \theta = 2 \pi \eta.
\end{equation*}
Set $\nu = 1$ and adopt the classical co-axial Couette flow \cite{drazin2006navierstokes}, for which the exact solution in polar coordinates is
\begin{align*}
    u_r = 0, \quad u_\theta = \frac{2}{3} \Big( r - \frac{1}{r} \Big), \quad p = \frac{2}{9} r^2 - \frac{8}{9} \ln r - \frac{2}{9 r^2}.
\end{align*}
Table~\ref{tab: 7.10} reports the errors and convergence rates. The velocity and velocity gradient converge at second order, while the pressure converges at only about first order. This may be related to the geometric error caused by approximating the curved boundary with straight edges. For this tangential flow, such an error may affect the pressure through the incompressibility constraint, whereas the velocity remains protected by the pressure robustness of the scheme.

\begin{figure}[htbp]
    \centering
    \begin{subfigure}[h]{0.4\textwidth}
        \centering
        \includegraphics{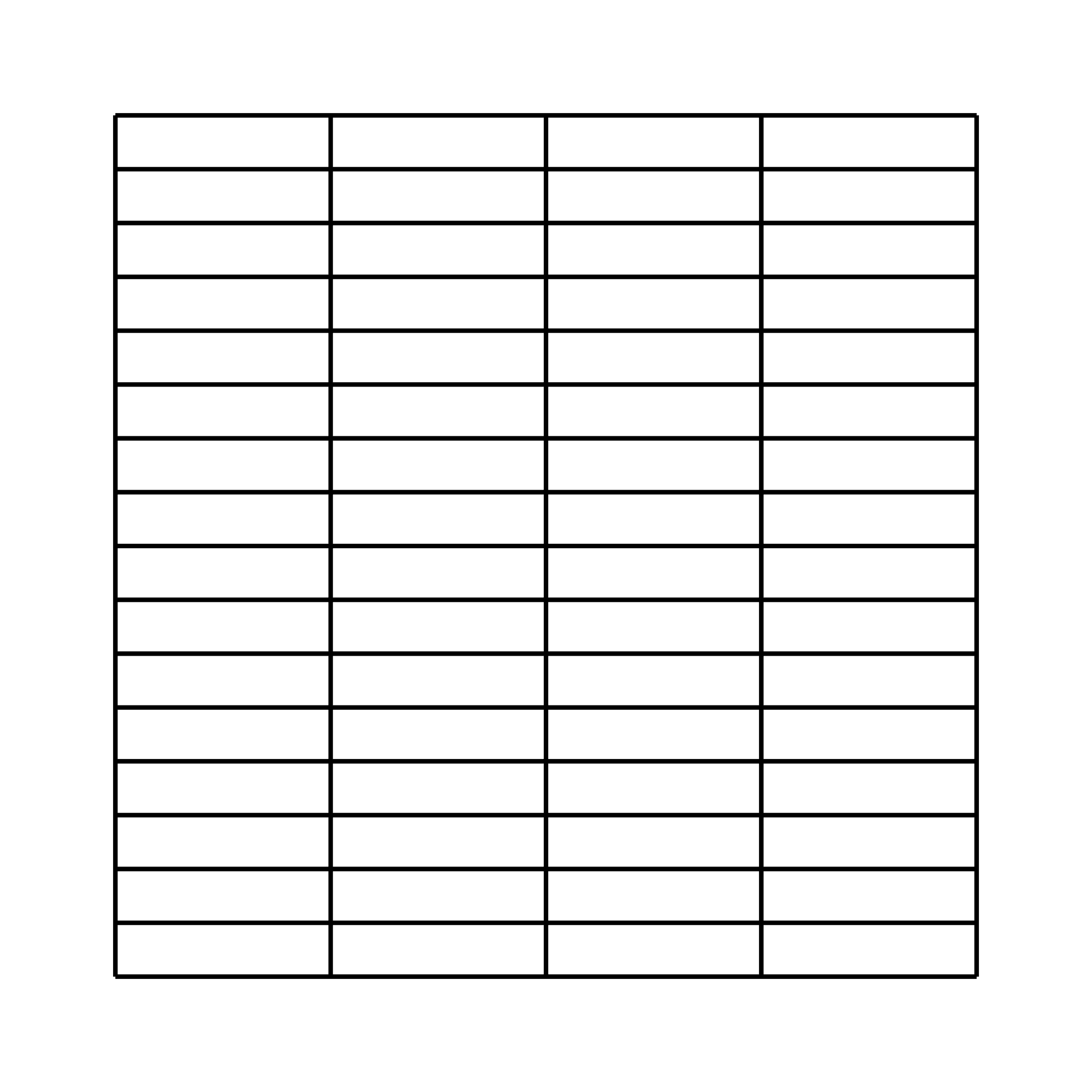}
        \caption{Reference grid.}
        \label{fig: 7.10.1}
    \end{subfigure}
    $\to$
    \begin{subfigure}[h]{0.4\textwidth}
        \centering
        \includegraphics{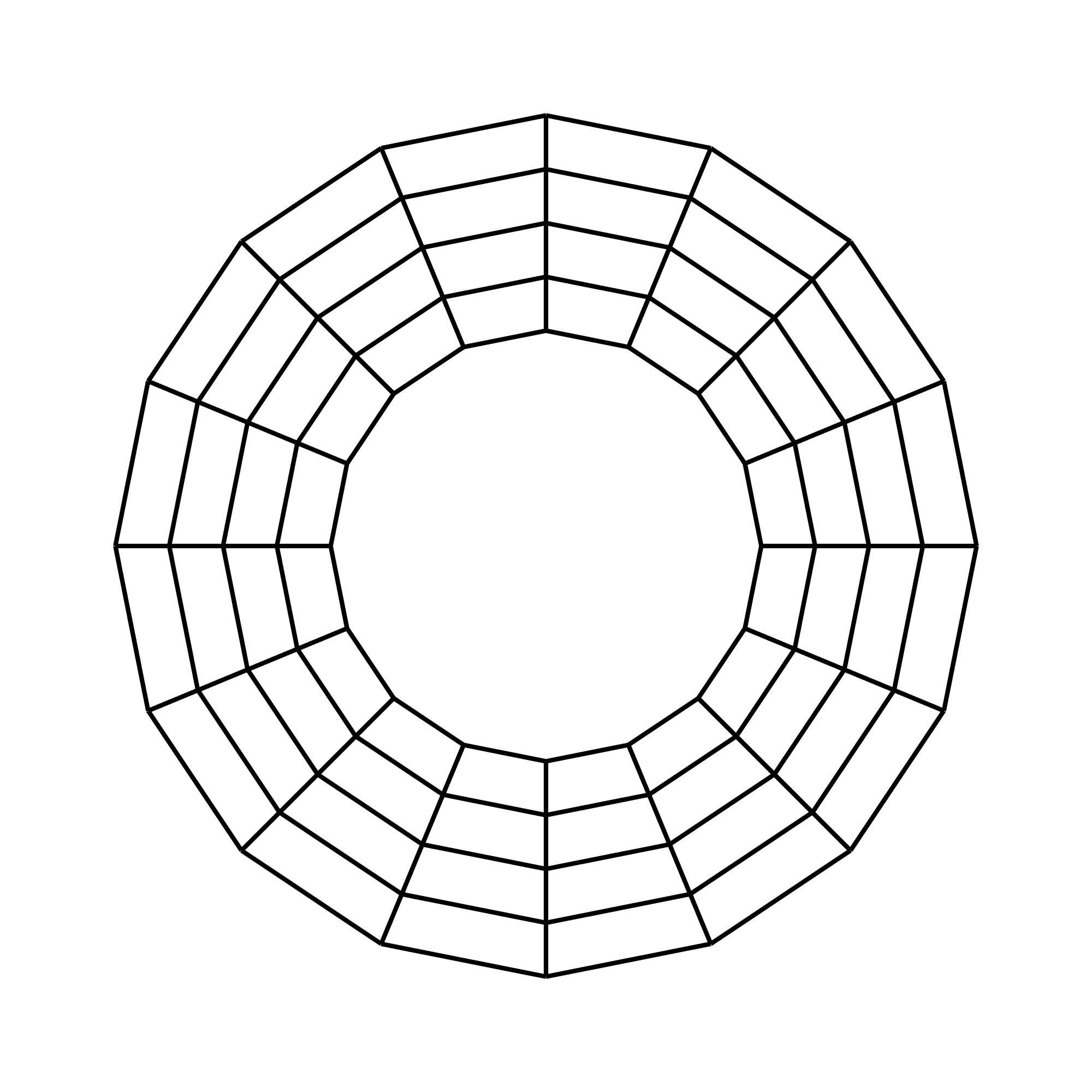}
        \caption{Physical mesh.}
        \label{fig: 7.10.2}
    \end{subfigure}
    \caption{Cartesian reference grid and its image on the annulus domain under the polar map.}
    \label{fig: 7.10}
\end{figure}

\begin{table}[htbp]
    \centering
    \caption{Errors and convergence rates of SDG$_0$ for stokes equations on the annulus domain.}
    \begin{tabular}{ccccccc}
        \toprule
        $n_x \times n_y$ & $\| \boldsymbol{\sigma} - \boldsymbol{\sigma}_h \|_{l^2}$ & Rate & $\| \mathbf{u} - \mathbf{u}_h \|_{l^2}$ & Rate & $\| p - p_h \|_{l^2}$ & Rate \\
        \midrule
        $8 \times 32$ & 8.05e-02 & -- & 7.85e-03 & -- & 1.19e+00 & -- \\
        $16 \times 64$ & 2.20e-02 & 1.87 & 2.37e-03 & 1.73 & 4.25e-01 & 1.48 \\
        $32 \times 128$ & 5.74e-03 & 1.94 & 6.27e-04 & 1.92 & 1.52e-01 & 1.48 \\
        $64 \times 256$ & 1.49e-03 & 1.95 & 1.59e-04 & 1.98 & 6.12e-02 & 1.32 \\
        $128 \times 512$ & 3.87e-04 & 1.94 & 4.00e-05 & 1.99 & 2.69e-02 & 1.19 \\
        \bottomrule
    \end{tabular}
    \label{tab: 7.10}
\end{table}

\subsection{Mixed Dirichlet and Natural Boundary Conditions}

The preceding experiments were conducted under purely Dirichlet boundary conditions.  We now consider a mixed setting: a Dirichlet condition $\mathbf{u} = \mathbf{g}_D$ on $\Gamma_D$ and a natural condition $(\nu \boldsymbol{\sigma} - p \mathbf{I}) \mathbf{n} = \mathbf{g}_N$ on $\Gamma_N$, with $\partial \Omega = \Gamma_D \cup \Gamma_N$. Here, $\nu \boldsymbol{\sigma} - p \mathbf{I}$ represents the pseudo-stress rather than the physical stress, and $(\nu \boldsymbol{\sigma} - p \mathbf{I}) \mathbf{n}$ is exactly the boundary term produced by integration by parts. Such a natural condition is standard in velocity-based formulations of the Stokes problem and is often used as an outflow condition \cite{heywood1996artificial}.

Let $\Gamma^D_h$ and $\Gamma^N_h$ denote the boundary edges lying on $\Gamma_D$ and $\Gamma_N$, respectively. Impose boundary conditions on the discrete spaces:
\begin{align*}
    U_{h, \mathbf{g}_D} &= \{ \mathbf{v}_h \in U_h : \mathbf{v}_h \cdot \mathbf{n} = \mathbf{g}_D \cdot \mathbf{n} \text{ on } \Gamma_D \}, \\
    \Sigma_{h, \mathbf{g}_N} &= \{ \boldsymbol{\tau}_h \in \Sigma_h : \boldsymbol{\tau}_h \mathbf{n} \cdot \mathbf{t} = \tfrac{1}{\nu} \, \mathbf{g}_N \cdot \mathbf{t} \text{ on } \Gamma_N \}.
\end{align*}
The scheme then reads: find $\boldsymbol{\sigma}_h \in \Sigma_{h, \mathbf{g}_N}$, $\mathbf{u}_h \in U_{h, \mathbf{g}_D}$ and $p_h \in P_h$ such that
\begin{align*}
    (\boldsymbol{\sigma}_h, \boldsymbol{\tau}_h) - B^*_h (\mathbf{u}_h, \boldsymbol{\tau}_h) &= \sum_{e \in \Gamma^D_h} \int_e (\mathbf{g}_D \cdot \mathbf{t})(\boldsymbol{\tau}_h \mathbf{n} \cdot \mathbf{t}) \, ds, & &\forall \boldsymbol{\tau}_h \in \Sigma_{h, 0}, \\
    \nu B_h (\boldsymbol{\sigma}_h, \mathbf{v}_h) - b^*_h (p_h, \mathbf{v}_h) &= (\mathbf{f}, \mathbf{v}_h) + \sum_{e \in \Gamma^N_h} \int_e (\mathbf{g}_N \cdot \mathbf{n})(\mathbf{v}_h \cdot \mathbf{n}) \, ds, & &\forall \mathbf{v}_h \in U_{h, 0}, \\
    b_h (\mathbf{u}_h, q_h) &= 0, & &\forall q_h \in P_h.
\end{align*}
Two features distinguish this from the pure-Dirichlet scheme. On $\Gamma_D$, the normal velocity is imposed essentially and the tangential velocity weakly, as before. On $\Gamma_N$, the velocity is left free and the natural condition is split: its tangential part is imposed essentially on the velocity-gradient space, while its normal part enters weakly as the surface load. Moreover, since the natural boundary condition determines the pressure level, the zero-mean constraint is dropped.

We test the scheme on $\Omega = (0, 1)^2$ with $\Gamma_N = \{ x = 1 \}$ and $\Gamma_D = \partial \Omega \setminus \Gamma_N$. Set $\nu = 1$ and take the manufactured solution
\begin{align*}
    \mathbf{u} = (- \cos (\pi x) \sin (\pi y), \ \sin (\pi x) \cos (\pi y)), \qquad p = \cos (\pi x) \cos (\pi y),
\end{align*}
from which $\mathbf{f}$, the Dirichlet data $\mathbf{g}_D = \mathbf{u}|_{\Gamma_D}$, and the natural data $\mathbf{g}_N = (\nu \nabla \mathbf{u} - p \mathbf{I}) \mathbf{n}|_{\Gamma_N}$ are computed. Table~\ref{tab: 7.11} reports the errors and convergence rates. The scheme retains second-order accuracy for all variables, confirming that the mixed Dirichlet--natural boundary treatment does not degrade the convergence order.

\begin{table}[htbp]
    \centering
    \caption{Errors and convergence rates under mixed boundary conditions.}
    \begin{tabular}{ccccccc}
        \toprule
        $n_x \times n_y$ & $\| \boldsymbol{\sigma} - \boldsymbol{\sigma}_h \|_{l^2}$ & Rate & $\| \mathbf{u} - \mathbf{u}_h \|_{l^2}$ & Rate & $\| p - p_h \|_{l^2}$ & Rate \\
        \midrule
        $8 \times 8$ & 1.34e-01 & -- & 7.09e-03 & -- & 1.13e-01 & -- \\
        $16 \times 16$ & 3.48e-02 & 1.94 & 1.55e-03 & 2.20 & 3.25e-02 & 1.80 \\
        $32 \times 32$ & 8.82e-03 & 1.98 & 3.71e-04 & 2.06 & 8.48e-03 & 1.94 \\
        $64 \times 64$ & 2.21e-03 & 1.99 & 9.15e-05 & 2.02 & 2.15e-03 & 1.98 \\
        $128 \times 128$ & 5.54e-04 & 2.00 & 2.28e-05 & 2.01 & 5.39e-04 & 1.99 \\
        \bottomrule
    \end{tabular}
    \label{tab: 7.11}
\end{table}

\section{Conclusion}

This work proposes an implementation-friendly SDG scheme based on Cartesian grids for Stokes and Navier-Stokes equations. Based on the staggered quadrilateral meshes extracted from the Cartesian grid, we construct piecewise-constant spaces with carefully-designed staggered continuity for velocity, pressure, and velocity gradient. Leveraging tailored basis functions, an explicit pointwise formulation of the scheme is derived. Local static condensation is employed to eliminate the diagonal entries of the velocity gradient; the remaining off-diagonal components can be further removed by standard mass lumping. The nonlinear convection term is discretized by a hybrid strategy that couples a MFE treatment of the velocity gradient with a DG-style discretization of the advective flux. For the temporal discretization of Navier-Stokes equations, we adopt the CN-SAV approach, which admits an efficient three-way splitting.

Concerning properties of the scheme, we provide a comprehensive theoretical analysis and numerical validation. For Stokes equations, we rigorously prove pressure robustness and second-order superconvergence for all variables on general non-uniform Cartesian grids. For Navier-Stokes equations, we prove the discrete convection term is second-order consistent; combined with the existing CN-SAV theory, the overall scheme is unconditionally energy-stable and second-order accurate. These properties are validated by numerical experiments: manufactured-solution confirms second-order accuracy; a no-flow problem verifies pressure robustness; Taylor vortex flow demonstrates accuracy under inhomogeneous boundary conditions; lid-driven cavity benchmarks match canonical reference results; tests on low-regularity solutions and anisotropic meshes show that the scheme remains robust beyond the theoretical assumptions; experiments on curved domains and non-Cartesian meshes confirm that the method remains effective on more general geometries; a mixed boundary condition test confirms second-order accuracy under more general boundary settings. Moreover, numerical experiments indicate that the mass-lumped variant achieves the same accuracy and robustness as the original scheme.

\section*{Acknowledgments}
The research of Eric Chung is partially supported by the Hong Kong RGC General Research Fund (Project numbers 14305423 and 14305624).

\appendix
\renewcommand{\theequation}{\thesection.\arabic{equation}}
\renewcommand{\theHequation}{\thesection.\arabic{equation}}
\setcounter{equation}{0}

\section{Derivation of the Pointwise Formulation} \label{sec: appendix}

We detail the computation leading to the condensed $\mathrm{SDG}_0$ scheme \eqref{eq: condensed SDG 1}--\eqref{eq: condensed SDG 5}. As in the proof of Lemma~\ref{lem: truncation error 1}, we scale the diagonal basis functions $\boldsymbol{\phi}^{\Sigma, x}_{e_{i^+, j^+}}$ and $\boldsymbol{\phi}^{\Sigma, y}_{e_{i^+, j^+}}$ so that their nonzero entry equals $\chi_{T_{i^+, j^+}}$. Taking $\boldsymbol{\tau}_h = \boldsymbol{\phi}^{\Sigma, x}_{e_{i^+, j^+}}$ in \eqref{eq: SDG 1} and applying \eqref{eq: velocity gradient relation},
\begin{align*}
    0 &= (\boldsymbol{\sigma}_h, \boldsymbol{\phi}^{\Sigma, x}_{e_{i^+, j^+}}) - B^*_h (\mathbf{u}_h, \boldsymbol{\phi}^{\Sigma, x}_{e_{i^+, j^+}}) \\
    &= \int_{T^-_{i^+, j^+}} \sigma^{x, x}_h \, dx \, dy + \int_{T^+_{i^+, j^+}} \sigma^{x, x}_h \, dx \, dy + \int_{e_{i, j^+}} u^x_h \, ds - \int_{e_{i + 1, j^+}} u^x_h \, ds \\
    &= \frac{\sigma^{n, x}_{i^+, j^+} - n^y_{i^+, j^+} \sigma^{x, y}_{i^+, j}}{n^x_{i^+, j^+}} \cdot \frac{h^x_{i^+} h^y_{j^+}}{2} + \frac{\sigma^{n, x}_{i^+, j^+} - n^y_{i^+, j^+} \sigma^{x, y}_{i^+, j + 1}}{n^x_{i^+, j^+}} \cdot \frac{h^x_{i^+} h^y_{j^+}}{2} + h^y_{j^+} u^x_{i, j^+} - h^y_{j^+} u^x_{i + 1, j^+} \\
    &= h^x_{i^+} l_{i^+, j^+} \sigma^{n, x}_{i^+, j^+} - \frac{(h^x_{i^+})^2}{2} ( \sigma^{x, y}_{i^+, j} + \sigma^{x, y}_{i^+, j + 1} ) + h^y_{j^+} ( u^x_{i, j^+} - u^x_{i + 1, j^+} ).
\end{align*}
Dividing by $h^x_{i^+} l_{i^+, j^+}$ gives the local condensation relation:
\begin{align} \label{eq: condensation relation}
    \sigma^{n, x}_{i^+, j^+} = \frac{h^x_{i^+}}{2 l_{i^+, j^+}} ( \sigma^{x, y}_{i^+, j} + \sigma^{x, y}_{i^+, j + 1} ) + \frac{h^y_{j^+}}{h^x_{i^+} l_{i^+, j^+}} ( u^x_{i + 1, j^+} - u^x_{i, j^+} ).
\end{align}
For $(i^+, j) \in \mathring{\Lambda}^H_h$, taking $\boldsymbol{\tau}_h = \boldsymbol{\phi}^\Sigma_{e_{i^+, j}}$ in \eqref{eq: SDG 1} and applying \eqref{eq: velocity gradient relation},
\begin{align*}
    0 &= (\boldsymbol{\sigma}_h, \boldsymbol{\phi}^\Sigma_{e_{i^+, j}}) - B^*_h (\mathbf{u}_h, \boldsymbol{\phi}^\Sigma_{e_{i^+, j}}) \\
    &= - \frac{h^x_{i^+}}{h^y_{j^+}} \int_{T^-_{i^+, j^+}} \sigma^{x, x}_h \, dx \, dy - \frac{h^x_{i^+}}{h^y_{j^-}} \int_{T^+_{i^+, j^-}} \sigma^{x, x}_h \, dx \, dy + \int_{T_{i^+, j}} \sigma^{x, y}_h \, dx \, dy \\ &\quad - \frac{h^x_{i^+}}{h^y_{j^+}} \int_{e_{i, j^+}} u^x_h \, ds + \frac{h^x_{i^+}}{h^y_{j^-}} \int_{e_{i + 1, j^-}} u^x_h \, ds \\
    &= - \frac{h^x_{i^+}}{h^y_{j^+}} \cdot \frac{\sigma^{n, x}_{i^+, j^+} - n^y_{i^+, j^+} \sigma^{x, y}_{i^+, j}}{n^x_{i^+, j^+}} \cdot \frac{h^x_{i^+} h^y_{j^+}}{2} - \frac{h^x_{i^+}}{h^y_{j^-}} \cdot \frac{\sigma^{n, x}_{i^+, j^-} - n^y_{i^+, j^-} \sigma^{x, y}_{i^+, j}}{n^x_{i^+, j^-}} \cdot \frac{h^x_{i^+} h^y_{j^-}}{2} \\ &\quad + h^x_{i^+} h^y_j \sigma^{x, y}_{i^+, j} - h^x_{i^+} u^x_{i, j^+} + h^x_{i^+} u^x_{i + 1, j^-} \\
    &= ( \frac{(h^x_{i^+})^3}{2 h^y_{j^+}} + \frac{(h^x_{i^+})^3}{2 h^y_{j^-}} + h^x_{i^+} h^y_j ) \sigma^{x, y}_{i^+, j} - \frac{(h^x_{i^+})^2 l_{i^+, j^+}}{2 h^y_{j^+}} \sigma^{n, x}_{i^+, j^+} - \frac{(h^x_{i^+})^2 l_{i^+, j^-}}{2 h^y_{j^-}} \sigma^{n, x}_{i^+, j^-} \\ &\quad - h^x_{i^+} u^x_{i, j^+} + h^x_{i^+} u^x_{i + 1, j^-}.
\end{align*}
Substituting \eqref{eq: condensation relation} yields the first equation of the condensed system:
\begin{align*}
    &[ h^x_{i^+} h^y_j + \frac{(h^x_{i^+})^3}{4 h^y_{j^+}} + \frac{(h^x_{i^+})^3}{4 h^y_{j^-}} ] \sigma^{x, y}_{i^+, j} - \frac{(h^x_{i^+})^3}{4 h^y_{j^+}} \sigma^{x, y}_{i^+, j + 1} - \frac{(h^x_{i^+})^3}{4 h^y_{j^-}} \sigma^{x, y}_{i^+, j - 1} \\
    & \quad - \frac{1}{2} h^x_{i^+} u^x_{i, j^+} + \frac{1}{2} h^x_{i^+} u^x_{i, j^-} - \frac{1}{2} h^x_{i^+} u^x_{i + 1, j^+} + \frac{1}{2} h^x_{i^+} u^x_{i + 1, j^-} = 0.
\end{align*}
For $(i, j^+) \in \mathring{\Lambda}^V_h$, taking $\mathbf{v}_h = \boldsymbol{\phi}^U_{e_{i, j^+}}$ in \eqref{eq: SDG 2},
\begin{align*}
    \int_{T_{i, j^+}} f^x \, dx \, dy &= \nu B_h (\boldsymbol{\sigma}_h, \boldsymbol{\phi}^U_{e_{i, j^+}}) - b^*_h (p_h, \boldsymbol{\phi}^U_{e_{i, j^+}}) \\
    &= \nu ( \int_{e_{i^+, j}} \sigma^{x, y}_h \, ds - \int_{e_{i^-, j + 1}} \sigma^{x, y}_h \, ds - \int_{e_{i^+, j^+}} \sigma^{n, x}_h \, ds + \int_{e_{i^-, j^+}} \sigma^{n, x}_h \, ds ) \\ &\quad + n^x_{i^+, j^+} \int_{e_{i^+, j^+}} p_h \, ds - n^x_{i^-, j^+} \int_{e_{i^-, j^+}} p_h \, ds \\
    &= \nu ( h^x_{i^+} \sigma^{x, y}_{i^+, j} - h^x_{i^-} \sigma^{x, y}_{i^-, j + 1} - l_{i^+, j^+} \sigma^{n, x}_{i^+, j^+} + l_{i^-, j^+} \sigma^{n, x}_{i^-, j^+} ) \\ &\quad + h^y_{j^+} p_{i^+, j^+} - h^y_{j^+} p_{i^-, j^+}.
\end{align*}
Substituting \eqref{eq: condensation relation} yields the third equation of the condensed system:
\begin{align*}
    & -\nu [ \frac{h^x_{i^+}}{2} \sigma^{x, y}_{i^+, j + 1} - \frac{h^x_{i^+}}{2} \sigma^{x, y}_{i^+, j} + \frac{h^x_{i^-}}{2} \sigma^{x, y}_{i^-, j + 1} - \frac{h^x_{i^-}}{2} \sigma^{x, y}_{i^-, j} + \frac{h^y_{j^+}}{h^x_{i^+}} u^x_{i + 1, j^+} - \frac{2 h^x_i h^y_{j^+}}{h^x_{i^+} h^x_{i^-}} u^x_{i, j^+} \\ &\quad + \frac{h^y_{j^+}}{h^x_{i^-}} u^x_{i - 1, j^+} ] + h^y_{j^+} p_{i^+, j^+} - h^y_{j^+} p_{i^-, j^+} = \int_{T_{i, j^+}} f^x \, dx \, dy.
\end{align*}
For $(i^+, j^+) \in \Lambda^D_h$, taking $q_h = \phi^P_{e_{i^+, j^+}}$ in \eqref{eq: SDG 3},
\begin{align*}
    0 &= b_h (\mathbf{u}_h, \phi^P_{e_{i^+, j^+}}) \\
    &= \int_{e_{i + 1, j^+}} u^x_h \, ds - \int_{e_{i, j^+}} u^x_h \, ds + \int_{e_{i^+, j + 1}} u^y_h \, ds - \int_{e_{i^+, j}} u^y_h \, ds \\
    &= h^y_{j^+} u^x_{i + 1, j^+} - h^y_{j^+} u^x_{i, j^+} + h^x_{i^+} u^y_{i^+, j + 1} - h^x_{i^+} u^y_{i^+, j},
\end{align*}
which is the fifth equation of the condensed system. The remaining equations of the condensed system are derived similarly.

Finally, we explain the reconstruction \eqref{eq: reconstructed diagonal velocity gradient}. By \eqref{eq: velocity gradient relation}, taking the average of $\sigma^{x, x}_h$ over $T_{i^+, j^+}$ yields
\begin{align*}
    \frac{1}{2} ( \sigma^{x, x}_h |_{T^-_{i^+, j^+}} + \sigma^{x, x}_h |_{T^+_{i^+, j^+}} ) = \frac{1}{n^x_{i^+, j^+}} [ \sigma^{n, x}_{i^+, j^+} - \frac{n^y_{i^+, j^+}}{2} ( \sigma^{x, y}_{i^+, j} + \sigma^{x, y}_{i^+, j + 1} ) ]
\end{align*}
Substituting \eqref{eq: condensation relation} into the above expression results in
\begin{align*}
    \frac{1}{2} ( \sigma^{x, x}_h |_{T^-_{i^+, j^+}} + \sigma^{x, x}_h |_{T^+_{i^+, j^+}} ) = \frac{u^x_{i + 1, j^+} - u^x_{i, j^+}}{h^x_{i^+}} = d_x u^x_{i^+, j^+}.
\end{align*}
Thus the cell average of $\sigma^{x, x}_h$ coincides with the difference quotient $d_x u^x_{i^+, j^+}$. The reconstruction for $\sigma^{y, y}_h$ follows analogously.

\end{document}